\documentclass[12pt]{amsart}
\usepackage{mathrsfs}
\usepackage{amssymb}
\usepackage{hyperref}
\hypersetup{hidelinks}
\usepackage{mathtools}
\usepackage{enumitem}
\newcommand{\Cramer}{\operatorname{Cram\acute{e}r}}
\newcommand{\Siegel}{\operatorname{Siegel}}
\usepackage[a4paper,inner=3cm,outer=3cm,top=4cm,bottom=4cm,pdftex]{geometry}
\usepackage{fancyhdr}
\newcommand{\C}{\mathbb{C}}
\newcommand{\Z}{\mathbb{Z}}
\newcommand{\E}{\mathbb{E}}
\newcommand{\1}{\mathbf{1}}

\let\oldpmod\pmod
\renewcommand{\pmod}[1]{\hspace{-0.12cm}\oldpmod {#1}}

\theoremstyle{plain}
\newtheorem{theorem}{Theorem}[section]
\newtheorem{proposition}[theorem]{Proposition}
\newtheorem{lemma}[theorem]{Lemma}
\newtheorem{corollary}[theorem]{Corollary}
\newtheorem{definition}[equation]{Definition}
\newtheorem{remark}[theorem]{Remark}

\numberwithin{equation}{section}
\title{Quantitative bounds for sets lacking polynomial progressions with shifted prime difference}

\author[B. Krause]{Ben Krause}
\address{BK: Department of Mathematics, University of Bristol, BS8 1QU, UK}
\email{ben.krause@bristol.ac.uk}

\author[H. Mousavi]{Hamed Mousavi}
\address{HM: Department of Mathematics, University of Bristol, BS8 1QU, UK}
\email{gj23799@bristol.ac.uk}

\author[T. Tao]{Terence Tao}
\address{TT: Department of Mathematics, University of California, Los Angeles, CA 90095-1555, USA}
\email{tao@math.ucla.edu}

\author[J. Ter\"{a}v\"{a}inen]{Joni Ter\"{a}v\"{a}inen}
\address{JT: Department of Pure Mathematics and Mathematical Statistics, University of Cambridge, CB3 0WB, UK}
\email{joni.p.teravainen@gmail.com}

\begin{document}

\begin{abstract}
We prove quantitative polynomial Szemer\'edi-type theorems involving polynomial progressions with shift parameter restricted to the set of shifted primes $\mathbb{P}-1$. The types of configurations covered are distinct degree progressions and progressions involving integer multiples of a fixed polynomial.

For nonlinear configurations of length at least three, these results provide the first quantitative versions of such theorems. In the linear case, our results improve on work by the last two authors. Our density bounds are strongest in the case of distinct degree polynomials, where they give polylogarithmic bounds, of the same shape as recent bounds by Shao and Wang with integer shifts.

The proofs combine recent quantitative results for polynomial configurations in the integers with quantitative Gowers uniformity bounds of the primes. For multiples of a fixed polynomial, we adapt a comparison argument of Altman and Sawhney to obtain uniformity over the polynomial families produced by the $W$-trick. For distinct degree progressions, we establish a comparison between prime-weighted and unweighted polynomial counts that is uniform throughout the density increment argument and accounts for a possible Siegel zero.
\end{abstract}

\maketitle

\tableofcontents

\section{Introduction}

Let $P_1,\ldots, P_{k-1}\in \mathbb{Z}[y]$ be polynomials in one variable with integer coefficients and zero constant term. A celebrated result of Bergelson and Leibman~\cite{bergelson} states that any subset of $\mathbb{N}=\{1,2,\ldots\}$ of positive asymptotic upper density contains \emph{polynomial progressions} of the form $x,x+P_1(d),\ldots, x+P_{k-1}(d)$ with $x,d\in \mathbb{N}$.

By a \emph{polynomial progression with shifted prime difference}, we mean a pattern of the form 
$x,x + P_1(p-1), \ldots, x + P_{k-1}(p-1)$, where $p$ is a prime (so that $p-1$ lies in the shifted primes ${\mathbb P}-1 = \{1,2,4,6,10,\dots\}$). Extending the Bergelson--Leibman theorem, it was shown by Wooley and Ziegler~\cite{Wooley-Ziegler} that if $A\subset [N]\coloneqq \{1,\dots,N\}$ does not contain any such progressions, then $|A|=o(N)$ for a fixed choice of $P_1,\dots,P_{k-1}$. See also~\cite{fhk1},~\cite{fhk2} for related work. However, as all these arguments were based on ergodic methods, no quantitative rate of decay for the $o(\cdot)$ notation was provided.

\subsection{Main results}

The main results of this paper are as follows. Firstly, in the linear case we can prove the following density bound for sets missing progressions with shifted prime difference. 

\begin{corollary}[Linear progressions]\label{thm_progression}
Let $N$ be sufficiently large, let $k\geq 3$, and let $A\subseteq [N]$ be a set lacking progressions of the form
\begin{align*}
    x,x+p-1,\ldots,x+(k-1)(p-1)
    \textup{ with } x\in\mathbb N \textup{ and } p\in\mathbb P.
\end{align*}
Then there exists $c_{k,\mathrm{lin}}>0$ such that the following hold.
 \begin{enumerate}
    \item If $k=3$, we have
    \begin{align*}
    |A|\ll N\exp(-(\log \log N)^{c_{3,\mathrm{lin}}}).    
    \end{align*}
    \item If $k\geq 4$, we have 
    \begin{align*}
    |A|\ll N\exp(-(\log \log \log N)^{c_{k,\mathrm{lin}}}).        
    \end{align*}
\end{enumerate}
\end{corollary}

The result improves on previously known bounds, discussed below, for all $k\geq 3$, and is a consequence of Theorem~\ref{thm:fixed-polynomial-shifted-prime} below.

The next two main theorems concern polynomial progressions with shifted prime difference for distinct degree polynomials and for integer multiples of a fixed polynomial, respectively. To our knowledge, for configurations of length at least three, these are the first quantitative bounds beyond the linear case.

\begin{theorem}[Distinct degree progressions]\label{thm_poly}
Let $N$ be sufficiently large, and let $k\in \mathbb{N}$. Let $P_1,\ldots, P_{k-1}\in \mathbb{Z}[y]$ be fixed nonconstant polynomials having distinct degrees and zero constant term. Let $A\subseteq [N]$ be a set lacking progressions of the form 
\begin{align*}
&x,x+P_1(p-1),\ldots,x+P_{k-1}(p-1)
    \textup{ with } x\in\mathbb N \textup{ and } p\in\mathbb P,\\
&0,P_1(p-1),\ldots,P_{k-1}(p-1)
    \textup{ are pairwise distinct}.
\end{align*}
 Then there exists $c_*=c(k,P_1,\ldots, P_{k-1})>0$ such that
\begin{align*}
    |A|\ll N(\log N)^{-c_*}.
\end{align*}
\end{theorem}

For $k\geq 3$, this result matches the best currently known bound for integer shifts~\cite{ShW25}, apart from the value of the exponent, as discussed below.

Our result for progressions involving integer multiples of a fixed polynomial states the following.

\begin{theorem}[Multiples of a fixed polynomial]
\label{thm:fixed-polynomial-shifted-prime}
Let $N$ be sufficiently large, let $k\geq3$, and let
$P\in\mathbb Z[y]$ be a fixed nonzero polynomial satisfying $P(0)=0$.
Let $a_1,\ldots,a_{k-1}\in\mathbb Z\setminus\{0\}$ be fixed and distinct.
Let $A\subseteq[N]$ be a set lacking progressions of the form
\begin{align*}
    x,x+a_1P(p-1),\ldots, x+a_{k-1}P(p-1) \textup{ with } x\in \mathbb{N},\,  p\in \mathbb{P},\, P(p-1)\neq 0.
\end{align*}
There exists $c=c(P,a_1,\ldots,a_{k-1},k)>0$ such that the following bounds
hold.
\begin{enumerate}
\item If $P'(0)\ne0$ and $k=3$, then
$$
|A|\ll_{P,a_1,a_2}
N\exp\bigl(- (\log\log N)^c\bigr).
$$
\item If $P'(0)\ne0$ and $k\geq4$, then
$$
|A|\ll_{P,a_1,\ldots,a_{k-1}}
N\exp\bigl(- (\log\log\log N)^c\bigr).
$$
\item If $P'(0)=0$, then
$$
|A|\ll_{P,a_1,\ldots,a_{k-1}}
N(\log\log\log N)^{-c}.
$$
\end{enumerate}
\end{theorem}

For non-monomial $P$, the shape of the bounds obtained matches the current best known bounds with integer shifts~\cite{AS25}. For monomial $P$ of degree at least $2$, part~(3) involves one additional iterated logarithm compared with the best known bound~\cite{prendiville}.

\subsection{Previous work} 

\subsubsection{Shifted prime differences}

In the linear case $P_i(y) = iy$, some quantitative density bounds have been obtained in the literature for the existence of shifted prime difference patterns. For $k=2$, this problem is known as S\'ark\"ozy's theorem and has a long history going back to~\cite{Sarkozy}. We summarize the previously best known results (excluding trivial cases such as $k=1$), due to Green~\cite{green-sarkozy}, Leng~\cite{leng-quadratic} and the last two authors~\cite{TT-JEMS}, as follows. Let $N$ be sufficiently large, let $k\in \mathbb{N}$, and let $A\subseteq [N]$ be a set lacking progressions of the form $x,x+p-1,\ldots, x+(k-1)(p-1)$ with $x\in \mathbb{N}$ and $p\in \mathbb{P}$.

Green~\cite[Theorem~1.1]{green-sarkozy} proved that, if $k=2$, then
\begin{align*}
    |A|\ll N^{1-c_2}.
\end{align*}
Leng~\cite[Theorem~6]{leng-quadratic} proved that, if $k=3$ and if there are no Siegel zeros\footnote{See \cite[Definition~2.2]{TT-JEMS} for a precise definition of a Siegel zero in this context.}, then
\begin{align*}
    |A|\ll N\exp\bigl(-(\log\log N)^{c_3}\bigr).
\end{align*}
The last two authors~\cite[Theorem~1.8]{TT-JEMS} proved unconditionally that, if $k\geq 3$, then
\begin{align*}
    |A|\ll
    \begin{cases}
        N\exp\bigl(-(\log\log\log N)^{c_3}\bigr),&k=3,\\
        N(\log\log\log N)^{-c_4},&k=4,\\
        N(\log\log\log\log N)^{-c_k},&k\geq 5,
    \end{cases}
\end{align*}
where the constants $c_k>0$ depend only on $k$.

Thus, Corollary~\ref{thm_progression} gives Leng's estimate unconditionally and improves the estimates of the last two authors in every case; when $k\geq 5$, the improvement is by nearly two iterated logarithms.

When it comes to shifted prime configurations with same degree monomials $x,x+a_1y^d,\ldots, x+a_{k-1}y^d$, in the two-point case $k=2$ and $d\geq 2$, Doyle and Rice~\cite[Theorem~1.5]{rice} give the estimate
\begin{align*}
    |A|\ll_{a_1,d}N(\log N)^{-c\log\log\log N}
\end{align*}
for some $c=c(a_1,d)>0$. For the linear case $k=2$ and $d=1$, a power saving bound for the size of $A$ follows from Green's theorem~\cite[Theorem~1.1]{green-sarkozy} after restricting $A$ to residue classes modulo $|a_1|$ and dilating (with the order of the two points reversed if $a_1<0$).

\subsubsection{Integer differences}

The corresponding problem with integer shifts amounts to quantitative bounds for the Bergelson--Leibman theorem. In the linear case, which corresponds to Szemer\'edi's theorem, the best result for $k=3$ is due to Raghavan~\cite[Theorem~1.4]{Raghavan-3AP}, improving on work of Bloom and Sisask~\cite{bloom-sisask} and Kelley and Meka~\cite{Kelley-Meka}; for $k=4$, the best result is due to Green and the third author~\cite{GT09}; and for $k\geq 5$ the best result is due to Leng, Sah and Sawhney~\cite{LSS-Szemeredi}. For comparison, these give bounds of the form
\begin{align*}
|A|\ll \begin{cases}N\exp\bigl(-c_3(\log N)^{1/6}(\log\log N)^{-1}\bigr),& k=3\\
 N(\log N)^{-c_4},&k=4\\
 N\exp(-(\log \log N)^{c_k}),&k\geq 5\end{cases}
\end{align*}
for some constants $c_k>0$.

For the Bergelson--Leibman theorem for polynomial progressions with distinct degree polynomials, strong quantitative bounds have been obtained in recent years by Peluse~\cite{Peluse-FMP}, Peluse--Prendiville~\cite{PP20},~\cite{PP22}, and Shao--Wang \cite{ShW25}. The conclusion from these works is the bound
\begin{align*}
 |A|\ll 
 N( \log N)^{-\widetilde{c}}
\end{align*}
for some $\widetilde{c}>0$ depending on $P_1,\ldots, P_{k-1}$. Thus, Theorem~\ref{thm_poly} matches this bound apart from the value of the exponent.

Quantitative bounds are also known for the progressions $x,x+a_1P(y),\ldots, x+a_{k-1}P(y)$ involving multiples of a fixed polynomial. For monomial $P$, the work of Prendiville~\cite{prendiville} gives the bound
\begin{align*}
|A|\ll N(\log \log N)^{-c'}    
\end{align*}
for some $c'>0$ depending on $k, a_1,\ldots, a_{k-1},d$. This was recently extended by Altman and Sawhney~\cite{AS25} to general $P$, with bounds of the same shape as in Theorem~\ref{thm:fixed-polynomial-shifted-prime}(1)--(3).

Finally, we mention that stronger density bounds for $A$ of quasipolynomial shape are known in the case of length two progressions $x,x+P(y)$ where the problem reduces to bounding the size of a set lacking polynomial differences; see~\cite{Green-Sawhney-Sarkozy},~\cite{Adajar}.

When combined, Corollary~\ref{thm_progression} and Theorems~\ref{thm_poly}--\ref{thm:fixed-polynomial-shifted-prime} give quantitative polynomial Szemer\'edi-type theorems with shifted prime differences for polynomials with no constant term in all the cases where a quantitative result with \emph{integer shifts} is currently known. For progressions of length at least three, these bounds are, depending on the case, either one iterated logarithm weaker than the bounds over the integers or of matching shape.

We additionally mention a result of Peluse, Sah and Sawhney~\cite{PSS} obtaining a quantitative bound for sets missing the progression $x,x+y^2-1,x+2(y^2-1)$ with $x,y\in \mathbb{N}$ with $y\neq 1$, as well as a generalization of that result by Kravitz, Kuca and Leng~\cite{KKL} where $y^2-1$ is replaced with any polynomial $P(y)$ having a simple integer root. In this paper, we confine ourselves to progressions involving polynomials with zero constant term, as in the Bergelson--Leibman theorem, and leave open the question whether the results for other polynomials could be extended to progressions with shifted prime difference.

\subsection{An overview of the proofs}\label{ss:overview}

To illustrate the main ideas, we discuss two representative cases. Let
$A\subseteq[N]$ have relative density $|A|/N=\delta$, and consider the
problem of finding in $A$ one of the following shifted prime polynomial
configurations:
$$
    x,\ x+\bigl((p-1)^2+(p-1)\bigr),
    \ x+2\bigl((p-1)^2+(p-1)\bigr)
$$
in the case of multiples of a fixed polynomial, or
$$
    x,\ x+(p-1),\ x+(p-1)^2
$$
in the case of distinct degree polynomials.

\subsubsection{Multiples of a fixed polynomial}
\label{ss:overview-common-polynomial}

For the first configuration, set $P(n)=n^2+n$. The relevant weighted
pattern count is
$$
\mathbb E_{x\leq N}\mathbb E_{n\leq N^{1/2}}\Lambda(n+1)
\1_A(x)\1_A(x+P(n))\1_A(x+2P(n)),
$$
for which the goal is to obtain a quantitative positive lower bound. Choose
$w=\lfloor(\log N)^\theta\rfloor$, where $\theta>0$ is a sufficiently
small fixed constant, and let $W$ be the powered smooth modulus used in
\eqref{e:AS-saturated-W} with $d=2$. In particular,
$$
    8\prod_{p\leq w}p\mid W,\qquad P^+(W)\leq w,
    \qquad \log W\ll w.
$$
After applying the $W$-trick, we restrict $A$ to a dense residue class
modulo $W$ and denote the resulting fiber by
$$
    A_{W,b}\coloneqq\{m\in[X]:Wm+b\in A\},
    \qquad X=N/W.
$$
It has density at least $\delta$. Since
$$
    P_W(y)\coloneqq\frac{P(Wy)}{W}=Wy^2+y,
$$
we are reduced to bounding from below
$$
\mathbb E_{x\leq X}\mathbb E_{y\leq M}
\frac{\phi(W)}{W}\Lambda(Wy+1)
\1_{A_{W,b}}(x)\1_{A_{W,b}}(x+P_W(y))
\1_{A_{W,b}}(x+2P_W(y)),
$$
where $M\asymp N^{1/2}/W$.

A key point is that $P_W$ varies with $W$, so a quantitative counting
result for one fixed polynomial cannot be applied directly. We instead adapt
an argument of Altman and Sawhney~\cite{AS25} uniformly over this family.
It compares the unweighted polynomial average along
\begin{equation}\label{e:overview-AS-original}
    x,\quad x+P_W(y),\quad x+2P_W(y),
\end{equation}
with an unweighted count along the model configurations
\begin{equation}\label{e:overview-AS-model}
    x,\quad x+(tW+1)y,\quad x+2(tW+1)y,
\end{equation}
where $t\asymp X^{1/2}$. To justify the comparison between
\eqref{e:overview-AS-original} and \eqref{e:overview-AS-model}, one
uses the quasipolynomial inverse theorem for the Gowers norms to replace the
input functions successively by nilsequences while retaining any putative
large discrepancy. The nilsequence comparison estimate of Altman and Sawhney
then rules out such a discrepancy. For fixed $t$, splitting $x$ into
residue classes modulo $tW+1$ reduces the model count to an ordinary
three-term progression count on intervals of length comparable to
$X^{1/2}/W$. The quantitative three-term progression count of Kelley and
Meka~\cite{Kelley-Meka} therefore supplies a lower bound uniform in $W$.
We now obtain
\begin{align}\label{e:m(delta)}
\mathbb E_{x\leq X}\mathbb E_{y\leq M}
\prod_{i=0}^{2}\1_{A_{W,b}}(x+iP_W(y))
\ge \mathfrak m(\delta),
\qquad
\mathfrak m(\delta)\coloneqq
\exp\bigl(-C(\log(3/\delta))^C\bigr),
\end{align}
where $C>0$ is a constant.

It remains to transfer this lower bound to the prime-weighted count. We use
the quantitative Gowers uniformity estimates from
\cite{Leng-efficient,MTW} to replace the von Mangoldt function by the Siegel
model introduced in~\cite{TT-JEMS}, which records the bias caused by a
possible exceptional zero. If the exceptional conductor does not divide
$W$, the Weil bound for character sums gives an $O(w^{-\eta})$ bound for
the required Gowers norms of the character term in the Siegel model along
$W\mathbb Z+1$. If it divides $W$, then, writing $\beta$ for the
exceptional zero, the character is constant there and the exceptional zero
factor may instead be replaced, in the required Gowers
norm, by
$\rho=1-(WM+1)^{\beta-1}$ with relative error $O(w^{-\eta})$. In all other
cases set $\rho=1$. Together with the Cram\'er model comparison, these estimates give an error
$O(\rho w^{-\eta})$ for some $\eta>0$. The effective exceptional zero
estimate gives
$$
    \rho\gg W^{-O(1)}=\exp(-O(w)).
$$
The error in replacing $\Lambda$ by the Siegel model is in fact
$O(\exp(-(\log N)^{c_1}))$. Since $w=(\log N)^\theta$, choosing
$\theta<c_1$ makes this $o(\rho w^{-\eta})$.

Choose $0<\gamma<\eta/4$. If
$\delta\gg\exp(-(\log w)^c)$, with $c>0$ sufficiently small, then the
formula~\eqref{e:m(delta)} for $\mathfrak m(\delta)$ gives
$\mathfrak m(\delta)>w^{-\gamma}$. Thus the main term is at least
$\rho w^{-\gamma}$, whereas the total comparison error is
$O(\rho w^{-\eta})$; the choice $\gamma<\eta/4$ leaves ample room
for the conclusion
$$
\mathbb E_{x\leq X}\mathbb E_{y\leq M}
\frac{\phi(W)}{W}\Lambda(Wy+1)
\1_{A_{W,b}}(x)\1_{A_{W,b}}(x+P_W(y))
\1_{A_{W,b}}(x+2P_W(y))
>\rho w^{-2\gamma}>0.
$$
The restriction on $\delta$ therefore comes from the Altman--Sawhney
counting estimate. Since $\log w\asymp\log\log N$, it becomes
$$
    \delta\gg\exp\bigl(-(\log\log N)^c\bigr)
$$
for a sufficiently small $c>0$.

The proof for a general fixed polynomial is similar in spirit, but the
quantitative bounds for three-term linear configurations~\cite{Kelley-Meka},
longer linear configurations~\cite{GT09,LSS-Szemeredi}, and higher degree
monomial configurations~\cite{prendiville} lead to the three alternatives in
Theorem~\ref{thm:fixed-polynomial-shifted-prime}.

\subsubsection{Distinct degree polynomials}

A strong bound in the distinct degree case involves additional difficulties.
The prime weight still has to be
replaced by a model that retains the possible exceptional zero bias. At the
same time, we need a density increment argument to find an increased density
for $A$ in sufficiently long arithmetic progressions, instead of merely
finding many configurations at one scale, since this would lose at least one
logarithm over the bound with integer shifts. In the model case, the density
increment naturally leads to the counting operator 
\begin{align}\label{e:cqcoeffoverview}
    T(w)\coloneqq
    \mathbb E_{x\le N}\mathbb E_{n\le M}
    w(qn+1)f_0(x)f_1(x+n)f_2(x+qn^2),
\end{align}
where $M=(N/q)^{1/2}$ and $q$ is a smooth modulus produced by the
iteration. If $A$ contains no required configuration, then $T(\Lambda)$
with $f_0=f_1=f_2=\1_A$ incurs only a negligible contribution from
higher prime powers, whereas $T(\Lambda)$ with
$f_0=f_1=f_2=\delta\1_{[N]}$ has the expected size. Their difference
is therefore a large prime-weighted discrepancy. The rest of the proof has two main
ingredients.

First, we transfer a discrepancy in the prime-weighted counting operator to an
unweighted discrepancy at the same spatial scale. We use a polynomial
generalized von Neumann estimate, with independent polynomial and spatial
scales, to bound changes in $T(w)$ by a Gowers norm of the weight $w$. We first
use quantitative uniformity of the primes to replace $\Lambda$ by the Siegel
model. We then remove the exceptional zero factor using character sum
estimates when the exceptional conductor does not divide $q$, and by partial
summation, at the cost of a slightly shorter shift scale, when it does.

We next compare the resulting Cram\'er model with the constant model.
For an integer cutoff $\sigma=\delta^{-O(1)}$, we arrange that $q$ is divisible by
every integer up to $\sigma$ and has no prime factor exceeding $\sigma$.
A linear forms estimate for the Cram\'er model then shows that the normalized Cram\'er model on $q\mathbb Z+1$ is close to
$1$ in the required Gowers norm. We therefore obtain at some scale $M_0$
only slightly smaller than $M$ an unweighted discrepancy. Writing
$$
    F_B(x,n)\coloneqq
    \1_B(x)\1_B(x+n)\1_B(x+qn^2),
$$
it has the form
\begin{equation}\label{e:overview-model-discrepancy}
\left|
\mathbb E_{x\le N}\mathbb E_{n\le M_0}
\bigl(F_A(x,n)-\delta^3F_{[N]}(x,n)\bigr)
\right|
\gg \delta^3.
\end{equation}

The second main ingredient turns this discrepancy into a density increment.
The local inverse theorem of Shao and Wang~\cite{ShW25} gives weak regularity
approximations of the copies of $\1_A$ by their conditional expectations
on local factors of common modulus $q^{B}$, where $B=O(1)$. This local inverse
theorem produces auxiliary moduli $q'\leq\sigma$, and the condition that
$\operatorname{lcm}(1,2,\ldots,\sigma)\mid q$ therefore ensures that every such $q'$ divides
$q$. If the conditional expectation of
$\1_A$ on every atom were at most $(1+c)\delta$, a telescoping
estimate would show that the discrepancy in~\eqref{e:overview-model-discrepancy} is $O(c\delta^3)$, which is impossible for
sufficiently small $c>0$. Hence some atom is an arithmetic progression of
length $N^{1/2-o(1)}$ on which the density of $A$ increases by a fixed
factor.

There are only $O(\log\delta^{-1}+1)$ density increment steps. At each step
the new modulus is a bounded power of the old one, so it satisfies the same divisibility
properties as the old modulus. If
$\delta>(\log N)^{-c_*}$ for a sufficiently small
$c_*>0$, the iteration would eventually produce a set of density greater
than $1$. This contradiction proves the bound in Theorem~\ref{thm_poly}.

\bigskip

\subsection{Organization of the paper}
The structure of the paper is as follows. In Section~\ref{s:notation}, we introduce the notation, prime models and preliminary estimates used throughout the paper. In Section~\ref{s:uniform}, we prove the Gowers norm and generalized von Neumann estimates used in the prime model comparisons. In Section~\ref{s:common-polynomial}, we prove the fixed polynomial result and deduce the linear progression result as a special case. Finally, in Section~\ref{s:distinct}, we prove the distinct degree result.

\subsection{Acknowledgments}
BK is supported by an EPSRC New Investigator grant and an ERC Starting Grant.
TT is supported by NSF grant DMS-2347850, the James and Carol Collins Chair, and
the Mathematical Analysis \& Application Research Fund. TT is particularly grateful
to recent donors to that fund. JT is supported by the European Union's Horizon
Europe research and innovation programme under ERC grant agreement no.~101162746.
ChatGPT was used as an auxiliary editorial tool for proofreading and preliminary
literature search assistance.  All mathematical statements, references, arguments,
and final wording were checked and approved by the authors.

\section{Notation and preliminaries}\label{s:notation}

The scale $N$ is assumed to be sufficiently large throughout so that expressions such as $\log\log\log \log N$ are well-defined and positive. 

We use the asymptotic notation $X \ll Y$, $Y \gg X$, or $X = O(Y)$ to denote the assertion $|X| \leq C Y$ for some constant $C$; if we need this constant to depend on a parameter, we will denote this by subscripts (unless explicitly stated otherwise). We also use $X\asymp Y$ to denote $X\ll Y\ll X$. The parameters $k,d\in \mathbb{N}$ will be thought of as fixed throughout the paper. We also fix some sufficiently small constant $c_0>0$; this appears in Definition~\ref{d:Siegel} below. 

We denote by $\1_{S}$ the indicator function of a set $S$, and for a finite nonempty set $A$ and for $f\colon A\to \mathbb{C}$ we use the usual averaging notation
\begin{align*}
\mathbb{E}_{x\in A}f(x)\coloneqq \frac{1}{|A|}\sum_{x\in A}f(x).    
\end{align*}
For real $X\geq 1$, we write
$$
    [X]\coloneqq \{1,\ldots,\lfloor X\rfloor\}.
$$
We also write
$$
    [\pm X]\coloneqq\{-\lfloor X\rfloor,\ldots,\lfloor X\rfloor\}.
$$
Throughout the paper, we adopt the convention that
$$
    \sum_{x\leq X} f(x)\coloneqq \sum_{x\in [X]} f(x),
    \qquad
    \mathbb E_{x\leq X} f(x)\coloneqq \mathbb E_{x\in [X]} f(x).
$$
We sometimes abbreviate $\sum_{x\in \mathbb{Z}}f(x)$ as $\sum_x f(x)$. 

We use the shorthand $e(x)=e^{2\pi i x}$ throughout the paper. For an integer $w\ge2$, set
\begin{align}\label{e:L(w)}
 \mathcal L(w)\coloneqq\operatorname{lcm}(1,2,\ldots,w).
\end{align}
We denote the largest prime factor of the positive integer $n$ by $P^{+}(n)$,
with the convention that $P^{+}(1)=1$.
Throughout, products indexed by $p$ are over primes.

Following Peluse~\cite[Definition~3.1]{Peluse-FMP}, we use the following
coefficient condition to track the polynomials under the rescalings in the
distinct degree density increment argument.

\begin{definition}
    A polynomial $P(y)=a_dy^d+\cdots+a_1y\in\mathbb Z[y]$ with zero constant term has $(C,q)$-coefficients if $|a_i|\leq C|a_d|$ for all $i=1,\ldots,d-1$, and $a_d=a'_dq^{d-1}$
    for some $a'_d\in\mathbb Z$ with $0<|a'_d|\leq C$.
\end{definition}

\bigskip

\subsection{Siegel zeros} For $X\ge2$, set
$$
    w_*(X)\coloneqq
    \left\lceil\exp\bigl((\log X)^{c_0}\bigr)\right\rceil.
$$
Whenever an argument is run at scale $X$, we use
Definition~\ref{d:Siegel} with $N=X$. Thus the exceptional zero, character,
and conductor at that scale are defined with cutoff $w_*(X)$.

\begin{definition}[Siegel zero and Siegel conductor]\label{d:Siegel}
Recall that $0<c_0<1/11$ is a small absolute constant. Here and in what follows, let
\begin{align*}
\widetilde{Q}\coloneqq w_*(N)
=\left\lceil\exp\bigl((\log N)^{c_0}\bigr)\right\rceil.    
\end{align*}
Let $L(s,\chi)$ be the $L$-function associated with the Dirichlet character $\chi$. A real number $\beta\in (0,1)$ is a \emph{Siegel zero} if there exists a primitive real Dirichlet character $\chi_{\Siegel}$ of conductor at most $\widetilde Q$ such that $L(\beta,\chi_{\Siegel})=0$ and
\begin{align}\label{e:siegel2}
     \beta\geq 1 - \frac{1}{1000\log \widetilde{Q} }. 
\end{align}
If such a zero $\beta$ exists, it is unique and simple, and the corresponding primitive character $\chi_{\Siegel}$ (called an \emph{exceptional character}) is unique and non-principal. We denote its conductor by $q_{\Siegel}$ and call it the \emph{Siegel conductor}.   
\end{definition}

The numerical constant in \eqref{e:siegel2}, as well as the uniqueness and
simplicity assertions, follow by combining
\cite[Theorem~5.26]{iwaniec-kowalski} with the explicit Landau and Page
theorems in~\cite[Theorems~1 and~2]{Pintz-Landau-Page}.

\begin{lemma}[Landau--Page estimate on $q\mathbb Z+1$]
\label{l:landau-page-compatible}
Let $A\ge1$. For $N$ sufficiently large in terms of $A$, set
$\widetilde Q=w_*(N)$ and use Definition~\ref{d:Siegel} at scale $N$.
Let $q,T\in\mathbb N$ satisfy
$$
 q\le\widetilde Q,
 \qquad
 N^{1/A}\le qT\le N^A.
$$
Then there is a constant $c_A>0$ such that, uniformly in $q$ and $T$,
\begin{equation}\label{e:landau-page-compatible}
 \mathbb E_{y\le T}\Lambda(qy+1)
 =\frac{q}{\phi(q)}
 \left(
 1-\frac{(qT+1)^{\beta-1}}{\beta}
 \1_{q_{\Siegel}\mid q}
 \right)
 +O_A\bigl(\exp(-c_A\sqrt{\log N})\bigr),
\end{equation}
where the exceptional term is interpreted as zero if the Siegel zero in
Definition~\ref{d:Siegel} does not exist.
\end{lemma}

\begin{proof}
For $q=1$, the result follows from the prime number theorem. We may therefore
suppose that $q\ge2$.
Set $X=qT+1$ and
$$
 \psi(X;q,1)\coloneqq
 \sum_{\substack{n\le X\\ n\equiv1\pmod q}}\Lambda(n).
$$
Since $\Lambda(1)=0$, we have
$$
 \sum_{y\le T}\Lambda(qy+1)=\psi(X;q,1).
$$
The Landau--Page theorem in the form of
\cite[Theorem~5.27]{iwaniec-kowalski} gives
$$
 \psi(X;q,1)
 =\frac{1}{\phi(q)}
 \left(
 X-\1_{\mathrm{exc}(q)}\frac{X^{\beta_q}}{\beta_q}
 \right)
 +O\bigl(X\mathcal E(X,q)\bigr),
$$
where
$$
 \mathcal E(X,q)
 =
 \exp\left(-\frac{c\log X}{\sqrt{\log X}+\log q}\right)
 (\log q)^4,
$$
where $c>0$ is an absolute constant and
$\1_{\mathrm{exc}(q)}=1$ precisely when there is a real Dirichlet
character $\chi$ modulo $q$ for which $L(s,\chi)$ has a real zero $\beta_q$
satisfying
$$
 \beta_q\ge 1-\frac{c}{\log(3q)};
$$
otherwise $\1_{\mathrm{exc}(q)}=0$. As $\chi$ ranges over the real
characters modulo $q$, there is at most one such zero.

Suppose that $\1_{\mathrm{exc}(q)}=1$, and let $\chi_q^*$ be a
primitive real character inducing the corresponding character modulo $q$.
Write $r_q\mid q$ for its conductor. If $\beta_q$ does not satisfy
\eqref{e:siegel2}, then, since $X\ge qT\ge N^{1/A}$,
$$
 X^{\beta_q-1}
 \le
 \exp\left(-\frac{\log X}{1000\log\widetilde Q}\right)
 \le
 \exp\bigl(-c_A(\log N)^{1-c_0}\bigr).
$$
The exceptional zero in \cite[Theorem~5.27]{iwaniec-kowalski} satisfies
$\beta_q^{-1}\ll1$, and
$$
 \frac{q}{\phi(q)}\ll 1+\log\log(3q).
$$
Since $X/T=q+1/T\le2q$, the contribution of this exceptional term after
division by $T$ is
$$
 \frac{X^{\beta_q}}{T\phi(q)\beta_q}
 \ll \frac{q}{\phi(q)}X^{\beta_q-1}
 \ll_A \exp(-c_A\sqrt{\log N}).
$$
We may therefore restrict attention to the case in which $\beta_q$ satisfies
\eqref{e:siegel2}. Since $r_q\le q\le\widetilde Q$, the uniqueness in the
Landau and Page theorems
\cite[Theorems~1 and~2]{Pintz-Landau-Page} gives
$\beta_q=\beta$ and $r_q=q_{\Siegel}$. In particular,
$q_{\Siegel}\mid q$. Conversely, if the Siegel zero in
Definition~\ref{d:Siegel} exists and $q_{\Siegel}\mid q$, then the character
induced by $\chi_{\Siegel}$ modulo $q$ has the zero $\beta$. By the choice of
the numerical constant in \eqref{e:siegel2}, this zero lies in the exceptional
range in \cite[Theorem~5.27]{iwaniec-kowalski}, and hence its exceptional term
occurs in the formula above.

Since $q\le\widetilde Q$ and $c_0<1/11$, we have
$$
 \log q\le(\log N)^{c_0}+O(1)=o(\sqrt{\log X}).
$$
It follows that
$$
 \frac{X}{T}\mathcal E(X,q)
 \ll_A \exp(-c_A\sqrt{\log N}).
$$
Moreover, $T\ge N^{1/(2A)}$ for $N$ sufficiently large in terms of $A$.
Thus replacing $X/T=q+1/T$ by $q$ in the normalized main terms costs
$O_A(\exp(-c_A\sqrt{\log N}))$. Dividing the estimate for $\psi(X;q,1)$ by
$T$ proves \eqref{e:landau-page-compatible}.
\end{proof}

\subsection{Models for the primes}

In order to obtain good quantitative bounds, we will compare the von Mangoldt
function $\Lambda$ with suitable model functions $\Lambda_{\mathrm{Model}}$,
for which $\Lambda-\Lambda_{\mathrm{Model}}$ has small uniformity norms.
There are several such models in the literature; in this paper, we use the
Cram\'{e}r model and the Siegel model, the latter taking into account a possible
Siegel zero.

Given $w\geq 2$, set 
$$W \coloneqq 8\prod_{p\leq w}p.$$
Define the \emph{Cram\'er model} to be 
\begin{align*}
    \Lambda_{\Cramer,w}(n) \coloneqq \frac{W}{\phi(W)}\1_{(n,W)=1},
\end{align*}
where $\phi$ is the Euler totient function and $(n,W)$ is the greatest common divisor of $n$ and $W$.
The factor $8$ changes the modulus but not the Cram\'er model, since
$$
    \frac{W}{\phi(W)}
    =\prod_{p\leq w}\left(1-\frac{1}{p}\right)^{-1},
    \qquad
    (n,W)=1
    \quad\Longleftrightarrow\quad
    \left(n,\prod_{p\leq w}p\right)=1.
$$
It is convenient because later progression moduli must absorb the possible
$2$-adic part of an exceptional conductor.
We recall the \emph{Siegel model} introduced in \cite{TT-JEMS}. In the presence of the Siegel zero $\beta$ (as defined in Definition~\ref{d:Siegel}), set
\begin{align*}
    \Lambda_{\Siegel,w}(n) \coloneqq \Lambda_{\Cramer,w}(n) ( 1 - n^{\beta-1}\chi_{\Siegel}(n)),\qquad  n\in \mathbb{N},
\end{align*}
where $\chi_{\Siegel}$ is the exceptional character. If no Siegel zero is present, we omit the term corresponding to such a zero, and
for $n\in\mathbb N$ set
\begin{align*}
    \Lambda_{\Siegel,w}(n) \coloneqq \Lambda_{\Cramer,w}(n).
\end{align*}

\subsection{Gowers norms}

For any integer $s \geq 1$ and any function $f \colon \mathbb{Z} \to \C$  having finite support, define the (unnormalized) \emph{Gowers uniformity norm}
$$ \| f \|_{U^{s}(\Z)} \coloneqq \left( \sum_{x,h_1,\dots,h_{s} \in \Z} \prod_{\omega \in \{0,1\}^{s}} \mathcal{C}^{|\omega|} f(x+\omega_1 h_1+\dots+\omega_{s} h_{s}) \right)^{1/2^{s}},$$
where $\omega = (\omega_1,\dots,\omega_{s})$, $|\omega| \coloneqq \omega_1+\dots+\omega_{s}$, and $\mathcal{C} \colon z \mapsto \overline{z}$ is the complex conjugation map.
For a finite nonempty interval $I\subseteq\mathbb Z$, define
$$
    \|f\|_{U^s[I]}
    \coloneqq
    \frac{\|f\1_I\|_{U^s(\mathbb Z)}}
         {\|\1_I\|_{U^s(\mathbb Z)}}.
$$
When $I=[a,b]$, we write $\|f\|_{U^s[a,b]}$ for
$\|f\|_{U^s[I]}$.
We use the standard monotonicity estimate
$$
    \|f\|_{U^s[I]}\ll_s\|f\|_{U^{s+1}[I]}
$$
for every finite nonempty interval $I$; this follows from the
Gowers--Cauchy--Schwarz inequality.

We also use the Gowers--Peluse norms of~\cite[Definition~2.4]{AS25}.
For a finitely supported function $f\colon\mathbb Z\to\mathbb C$, set
$$
    \Delta_{(h,h')}f(x)
    \coloneqq f(x+h)\overline{f(x+h')}.
$$
If $\mu_1,\ldots,\mu_s$ are probability measures on $\mathbb Z$, define
$$
    \|f\|_{U^s_{\mathrm{GP}}[N;\mu_1,\ldots,\mu_s]}
    \coloneqq
    \left(
    \frac{1}{N}\sum_{x\in\mathbb Z}
    \mathbb E_{h_i,h'_i\sim\mu_i\ (1\le i\le s)}
    \Delta_{(h_1,h'_1)}\cdots\Delta_{(h_s,h'_s)}f(x)
    \right)^{1/2^s}.
$$
Here
$\mathbb E_{h\sim\mu}f(h)
\coloneqq\sum_{h\in\mathbb Z}f(h)\mu(\{h\})$.
When all the measures are equal to $\mu$, we write
$\|f\|_{U^s_{\mathrm{GP}}[N;\mu]}$. For a finite nonempty multiset $E$,
we write $\mu_E$ for its uniform probability measure, and for an integer
$q$ we write $q\cdot E\coloneqq\{qh:h\in E\}$, retaining multiplicities.

We use the following dilation estimate for normalized Gowers norms.

\begin{lemma}[Dilation estimate for normalized Gowers norms]\label{le:dilate}
Let $d\geq 2$ and $N\geq1$ be integers, and let $b,q\in\mathbb N$ with
$q\le N$. Then, for any function $f\colon\mathbb Z\to\mathbb C$, we have
\begin{equation}\label{e:dilationus}
\|f(q\cdot+b)\|_{U^d[N/q]}
\ll_d q^{\frac{d+1}{2^d}}\|f\|_{U^d[b,N+b]}.
\end{equation}
\end{lemma}

\begin{proof}
Set $m\coloneqq\lfloor N/q\rfloor$ and
$J\coloneqq[q+b,qm+b]$. By the definition of the normalized Gowers norm
and a change of variables,
$$
\|f(q\cdot+b)\|_{U^d[m]}
=
\frac{\|f\1_{q\mathbb Z+b}\1_J\|_{U^d(\mathbb Z)}}
     {\|\1_{[m]}\|_{U^d(\mathbb Z)}}.
$$
The progression $(q\mathbb Z+b)\cap[b,N+b]$ is the disjoint union of
$J\cap(q\mathbb Z+b)$ and the single point $\{b\}$. Hence, by the
triangle inequality,
$$
\|f\1_{q\mathbb Z+b}\1_J\|_{U^d(\mathbb Z)}
\le
\|f\1_{q\mathbb Z+b}\1_{[b,N+b]}\|_{U^d(\mathbb Z)}
+|f(b)|.
$$
By the orthogonality of characters, we have
\begin{align*}
\1_{q\mathbb{Z}+b}(n)=\frac{1}{q}\sum_{0\leq r<q}e\left(\frac{r(n-b)}{q}\right).    
\end{align*}
The triangle inequality and invariance under multiplication by a linear
phase therefore give
$$
\|f\1_{q\mathbb Z+b}\1_{[b,N+b]}\|_{U^d(\mathbb Z)}
\le
\|f\1_{[b,N+b]}\|_{U^d(\mathbb Z)}.
$$
Moreover, writing
$\Delta_hF(n)\coloneqq F(n)\overline{F(n+h)}$, the recursive identity
$$
    \|F\|_{U^d(\mathbb Z)}^{2^d}
    =
    \sum_{h\in\mathbb Z}
    \|\Delta_hF\|_{U^{d-1}(\mathbb Z)}^{2^{d-1}}
$$
holds for every finitely supported function $F$. Retaining the nonnegative
term with $h=0$, we obtain
$$
    \|F\|_{U^d(\mathbb Z)}^{2^d}
    \ge
    \bigl\||F|^2\bigr\|_{U^{d-1}(\mathbb Z)}^{2^{d-1}}.
$$
Iterating the same argument gives
$$
    \|F\|_{U^d(\mathbb Z)}^{2^d}
    \ge
    \left(\sum_n|F(n)|^{2^{d-1}}\right)^2
    \ge
    \sum_n|F(n)|^{2^d}.
$$
Applying this with $F=f\1_{[b,N+b]}$ gives
$$
    |f(b)|
    \le
    \|f\1_{[b,N+b]}\|_{U^d(\mathbb Z)}.
$$
Finally, the interval cube count and $N+1\le2qm$ give
$$
\frac{\|\1_{[b,N+b]}\|_{U^d(\mathbb Z)}}
     {\|\1_{[m]}\|_{U^d(\mathbb Z)}}
\ll_d q^{(d+1)/2^d}.
$$
Combining these estimates proves \eqref{e:dilationus}.
\end{proof}

\section{Uniformity and inverse estimates for polynomial configurations}\label{s:uniform}

In this section we record the Gowers norm and inverse inputs used in the
prime-model comparisons of Section~\ref{s:common-polynomial} and in the
distinct degree argument of Section~\ref{s:distinct}. We first record the Gowers norm approximation of the
primes by the Siegel model and prove the uniformity of the Cram\'er model on
smooth progressions. We then introduce the counting forms and local factors,
prove a generalized von Neumann estimate with independent shift and spatial
scales, and state the Shao--Wang inverse input used in weak regularity.

\begin{lemma}[Gowers norm approximation of primes by the Siegel model]
\label{l:prime-siegel-gowers}
For every $s\in\mathbb N$, there exists $c_s>0$ such that the following holds.
Let $X,Y\ge3$, and use Definition~\ref{d:Siegel} at scale $X$, so that
$\widetilde Q=w_*(X)$. Suppose that
\begin{align*}
\exp\bigl((\log\log Y)^{1/c_s}\bigr)
\le \widetilde Q\le
\exp\bigl((\log Y)^{1/10}\bigr).
\end{align*}
Then
\begin{align}\label{e:prime-siegel-gowers}
\|\Lambda-\Lambda_{\Siegel,\widetilde Q}\|_{U^s[Y]}
\ll_s\exp\bigl(-(\log \widetilde Q)^{c_s}\bigr).
\end{align}
\end{lemma}

\begin{proof}
This is \cite[Proposition~4.6]{MTW}, applied at scale $Y$ with cutoff
$\widetilde Q$; that result is based on \cite{Leng-efficient}.
\end{proof}

\subsection{Gowers uniformity of the Cram\'er model}

The following consequence of the linear forms estimate in
\cite[Section~5]{TT-JEMS} is used in Lemma~\ref{l:powered-cramer} and
Proposition~\ref{lem:chaining}. Its formulation allows the progression modulus
to contain powers of the small primes.

\begin{proposition}[Gowers uniformity of Cram\'er model on a smooth progression]
\label{p:saturated-cramer-gowers}
Let $s\ge1$ be fixed, and let $M\ge3$. Let
$w,z,Q\in\mathbb N$, $b\in\mathbb Z$, and let $J\subseteq[M]$ be an
interval. Suppose that
$$
    2\le w\le z,
    \qquad
    \prod_{p\le w}p\mid Q,
    \qquad
    P^+(Q)\le w,
    \qquad
    (b,Q)=1,
$$
and that
$$
    Q\le\exp\bigl((\log M)^{2/5}\bigr),
    \qquad
    z+1\le\exp\bigl((\log M)^{1/10}\bigr).
$$
Then
\begin{align*}
\left\|
\1_J
\left(
\frac{\phi(Q)}{Q}\Lambda_{\Cramer,z}(Q\cdot+b)-1
\right)
\right\|_{U^s[M]}
\ll_s w^{-1/2^s}.
\end{align*}
\end{proposition}

\begin{proof}
When $M$ is bounded in terms of $s$, the hypotheses also bound
$w,z,Q$, so the result follows after enlarging the implied constant. We may
therefore suppose that $M$ is sufficiently large.
Set
$$
    \nu(n)\coloneqq
    \frac{\phi(Q)}{Q}\Lambda_{\Cramer,z}(Qn+b)
$$
and, if $J=\{j_1,j_1+1,\ldots, j_2\}$, write
$$
\Omega_J\coloneqq
\left\{(x,\mathbf h)\in\mathbb R^{s+1}:
x+\omega\cdot\mathbf h\in [j_1,j_2]
\text{ for every }\omega\in\{0,1\}^s\right\}.
$$
For each nonempty $\mathcal S\subseteq\{0,1\}^s$, apply
\cite[Proposition~5.2]{TT-JEMS}, at scale $M$, to the affine forms
\begin{equation}\label{e:cramer-cube-forms}
    Q(x+\omega\cdot\mathbf h)+b,
    \qquad \omega\in\mathcal S.
\end{equation}
Their linear coefficient vectors $Q(1,\omega)$ are pairwise linearly
independent and have magnitude $O_s(Q)$, so the coefficient hypothesis in
the cited proposition follows from the bound on $Q$. Our convention uses the
primes $p\le z$, whereas the cited proposition uses $p<z$; we apply it
there with cutoff $z+1$.

The hypotheses on $Q$ imply
$$
    \frac{\phi(Q)}{Q}=\prod_{p\le w}\left(1-\frac{1}{p}\right).
$$
Let $\beta_{p,\mathcal S}$ denote the local factor associated with the forms
in \eqref{e:cramer-cube-forms} in
\cite[Proposition~5.2, in particular (5.3) and (5.5)]{TT-JEMS}.
For $p\le w$, every affine form in \eqref{e:cramer-cube-forms} is congruent to
$b\not\equiv0\pmod p$, so the local factor
$\beta_{p,\mathcal S}=(p/(p-1))^{|\mathcal S|}$ is cancelled by the corresponding factor from
$(\phi(Q)/Q)^{|\mathcal S|}$. For $p>w$, multiplication by $Q$ is
invertible modulo $p$, and the usual local calculation gives
$$
    \beta_{p,\mathcal S}=1+O_s(p^{-2}).
$$
Consequently,
$$
\left(\frac{\phi(Q)}{Q}\right)^{|\mathcal S|}
\prod_{p\le z}\beta_{p,\mathcal S}
=1+O_s(w^{-1}).
$$
It follows that, uniformly for nonempty $\mathcal S$,
\begin{align}
&\sum_{(x,\mathbf h)\in\Omega_J\cap\mathbb Z^{s+1}}
\prod_{\omega\in\mathcal S}\nu(x+\omega\cdot\mathbf h)\notag\\
&\qquad=
\operatorname{vol}(\Omega_J)
+O_s\left(
M^{s+1}w^{-1}
+M^{s+1}\exp\bigl(-c_s(\log M)^{4/5}\bigr)
\right)
\label{e:cramer-cube-correlation}
\end{align}
for some $c_s>0$. For $\mathcal S=\varnothing$, the same formula
holds with error $O_s(M^s)$, by the standard lattice point estimate.
Expanding the $2^s$-th power of the Gowers norm of
$\1_J(\nu-1)$, the common volume terms cancel. Noting that
$$
    \|\1_{[M]}\|_{U^s(\mathbb Z)}^{2^s}\asymp_s M^{s+1},
$$
and using \eqref{e:cramer-cube-correlation}, we obtain
$$
\left\|
\1_J(\nu-1)
\right\|_{U^s[M]}^{2^s}
\ll_s
w^{-1}+\exp\bigl(-c_s(\log M)^{4/5}\bigr)+M^{-1}.
$$
The two final terms are $O_s(w^{-1})$ in the stated range, which proves
the result.
\end{proof}

\subsection{Multilinear forms}

In this subsection, we define the multilinear form used throughout the
paper. Let $k\ge2$ and $N,M,Q_0\in\mathbb N$, and let
$P_1,\ldots,P_{k-1}\in\mathbb Z[y]$ satisfy $P_i(0)=0$. Let
$f_0,\ldots,f_{k-1},w\colon\mathbb Z\to\mathbb C$. Whenever the sum below
is absolutely convergent, define
\begin{align}\label{e:Q0avgmulti}
T_{N,M,w,Q_0}(f_0,\ldots,f_{k-1})
\coloneqq
\frac{1}{N}\sum_{x\in\mathbb Z}\mathbb E_{n\le M}
w(Q_0n+1)f_0(x)
\prod_{i=1}^{k-1}f_i(x+P_i(n)).
\end{align}
For notational convenience, we omit the explicit dependence on the
polynomials $P_i$, except in the proof of Theorem~\ref{thm_poly}, where
this dependence is reinstated.
\subsection{Local factors}\label{ss:localfactor}

A \emph{factor} $\mathcal{B}$ of $[N]$ is a partition of $[N]$ into disjoint
subsets $\{B_i\}_i$ with $[N]=\bigcup_i B_i$.

For a factor $\mathcal{B}$, define the projection $\Pi_{\mathcal{B}}$ of a
function $f\colon [N]\to\mathbb C$ by
$$
    \Pi_{\mathcal B}f(x)\coloneqq \mathbb E_{y\in B_x}f(y),
$$
where $B_x$ is the unique atom of $\mathcal B$ containing $x$.  Thus
$\Pi_{\mathcal B}f$ is the conditional expectation of $f$ with respect to the
$\sigma$-algebra generated by the atoms of $\mathcal B$.
In particular, $\Pi_{\mathcal B}f$ has the same total sum as $f$, and
$f-\Pi_{\mathcal B}f$ has mean zero on every atom of $\mathcal B$.

A \emph{local factor} of resolution $M$ and modulus $q$ is a factor
$\mathcal B$ of $[N]$ such that every atom of $\mathcal B$ is an arithmetic
progression of common difference $q$ and length in $[M,2M]$.  Equivalently,
each atom is of the form
$$
    (q\mathbb Z+a)\cap I
$$
for some residue class $a\pmod q$ and some interval $I$, with
$$
    M\leq |(q\mathbb Z+a)\cap I|\leq 2M.
$$
This is the notion of local factor used in \cite[Definition~3.4]{ShW25}.

\subsection{A weighted generalized von Neumann estimate}\label{s:GVNT1}

We need a variant of the polynomial generalized von Neumann estimate of
\cite[Theorem~4.2]{Mobius-ergodic} in which the polynomial variable has length
$M$, independently of the spatial support scale $N$. This is used later
when the shift scale shortens while the other functions remain supported at
scale $N$. The PET induction argument is the same as in the proof of that
theorem. The differences are that we keep the spatial scale fixed when the
degree decreases and continue the induction until the polynomial system is
constant. We record these changes below.

\begin{proposition}[Weighted polynomial generalized von Neumann estimate with independent scales]
\label{l:fixed-scale-gvn}
Let $d,k,M,N\in\mathbb N$ and $C\ge1$. Let
$P_1,\ldots,P_k\in\mathbb Z[y]$ have degrees at most $d$, and let
$g\colon[M]\to\mathbb C$ be arbitrary. For each $0\le i\le k$, let
$f_i\colon\mathbb Z\to\mathbb C$ be a $1$-bounded function supported on a
set of cardinality at most $CN$. Then there is an integer
$2\le s\ll_{d,k}1$ such that
\begin{align*}
    \left|
    \frac{1}{N}\sum_{x\in\mathbb Z}\E_{y\le M}
    g(y)f_0(x)\prod_{i=1}^kf_i(x+P_i(y))
    \right|
    \ll_{C,d,k}\|g\|_{U^s[M]}.
\end{align*}
\end{proposition}

\begin{proof}
We follow the PET induction in
\cite[Proof of Theorem~4.2]{Mobius-ergodic}
(see also~\cite{bergelson}).
Constant polynomials may be
absorbed into $f_0$. If two polynomials differ by a
constant, we combine their factors. These operations preserve
$1$-boundedness and do not increase the cardinality of the supports. We may
therefore suppose that the remaining polynomials are nonconstant and that no
two differ by a constant.

For such a family $\mathcal P=\{P_1,\ldots,P_\ell\}$, write
$$
    \Lambda_{\mathcal P}(g)
    \coloneqq
    \frac{1}{N}\sum_{x\in\mathbb Z}\E_{y\le M}
    g(y)f_0(x)\prod_{i=1}^{\ell}f_i(x+P_i(y)).
$$
Choose $P_1$ to have the least positive degree. Applying
Cauchy--Schwarz in $x$, expanding the square, writing the second polynomial
variable as $y+h$, and then relabeling $x+P_1(y)$ as $x$ gives
\begin{equation}\label{e:fixed-scale-pet-step}
    |\Lambda_{\mathcal P}(g)|^2
    \ll_C \frac{1}{M}\sum_{|h|<M}|\mathcal A_h|,
\end{equation}
where, on extending $g$ by zero outside $[M]$ and writing
$\Delta_hg(y)\coloneqq g(y)\overline{g(y+h)}$,
\begin{align*}
    \mathcal A_h\coloneqq
    \frac{1}{N}\sum_{x\in\mathbb Z}\E_{y\le M}\Delta_hg(y)f_1(x)
    &\overline{f_1\bigl(x+P_1(y+h)-P_1(y)\bigr)}
    \prod_{i=2}^{\ell}f_i\bigl(x+P_i(y)-P_1(y)\bigr)\\
    &\quad\times\prod_{i=2}^{\ell}
    \overline{f_i\bigl(x+P_i(y+h)-P_1(y)\bigr)}.
\end{align*}
The factor $f_1(x)$ is the new function at the zero shift. Every new spatial
function is a translate, conjugate, or product of translates of the original
functions. Each remains $1$-bounded and is supported on a set of cardinality at
most $CN$. Thus the spatial normalization in every subsequent
Cauchy--Schwarz step remains $1/N$.

After constant polynomials and repetitions have again been absorbed, the new
polynomial family is
$$
    \{P_i(y)-P_1(y):2\le i\le\ell\}
    \cup
    \{P_i(y+h)-P_1(y):1\le i\le\ell\}.
$$
We use the PET semi-invariant from
\cite[Proof of Theorem~4.2]{Mobius-ergodic}, which associates to each tuple of
polynomials a vector with nonnegative integer entries. In degree $j$, its
entry is the number of distinct leading coefficients among the polynomials of
degree $j$. The vectors are ordered lexicographically from the highest
degree down. If $d_1=\deg P_1$, the entries in degrees greater than $d_1$ are
unchanged. In degree $d_1$, the leading coefficients $c_i$ are replaced by
the nonzero differences $c_i-c_1$, while $P_1(y+h)-P_1(y)$ has degree at most
$d_1-1$. Thus the semi-invariant strictly decreases at every step.

The proof of \cite[Theorem~4.2]{Mobius-ergodic} stops the induction at a linear family and applies its linear
generalized von Neumann estimate. That estimate uses the relation between its
two scales, so we instead continue the PET step until we reach constant
polynomials. If
$c_1,\ldots,c_t$ are the distinct nonzero leading coefficients of a linear
family and $P_1$ has leading coefficient $c_1$, then the new nonzero leading
coefficients are
$$
    c_2-c_1,\ldots,c_t-c_1.
$$
Thus their number decreases by $1$. Repeating this step leaves only constant
polynomials.

The base case has no nonconstant polynomials. After absorbing all constant
factors into $f_0$, it is bounded by
$$
    \frac{1}{N}\sum_{x\in\mathbb Z}|f_0(x)|
    \left|\E_{y\le M}g(y)\right|
    \ll_C \|g\|_{U^1[M]}.
$$
The same quantity is $\ll_C\|g\|_{U^2[M]}$ by monotonicity of the normalized
Gowers norms, so the order may be taken to be at least $2$.
For exceptional values of $h$, degrees or polynomial classes may collapse,
but such degeneracies only lower the PET semi-invariant. We verify that the length of the descent is uniformly bounded. Consider in general a branch starting from at most $m$ polynomials and having semi-invariants with $e$ components. We argue by induction on $e$, with the case $e=0$ being immediate. 

Split the branch into maximal blocks on which the highest component of the semi-invariant is constant. Within each of these blocks, the remaining $e-1$ components strictly decrease, so induction bounds the length of the block in terms of $e$ and the number of polynomials at its beginning. Between consecutive blocks the highest degree component decreases. Since this component is initially at most $m$, there are at most $m+1$ blocks. A PET step creates at most twice as many polynomials. Applying the induction hypothesis successively to the at most $m+1$ blocks therefore bounds the starting size and length of every block in terms of $e$ and $m$, and hence the length of the branch and the size of every descendant family.

Taking $e=d$ and $m=k$, every branch has depth $\ll_{d,k}1$ and every descendant family contains $\ll_{d,k}1$ polynomials.  By monotonicity of the normalized Gowers norms, we may consequently use a common order
$r\ll_{d,k}1$, independent of $h$, such that
$$
    |\mathcal A_h|
    \ll_{C,d,k}\|\Delta_hg\|_{U^r[M]}.
$$
Hence \eqref{e:fixed-scale-pet-step} gives
$$
    |\Lambda_{\mathcal P}(g)|^2
    \ll_{C,d,k}
    \frac{1}{M}\sum_{|h|<M}\|\Delta_hg\|_{U^r[M]}.
$$
The interval cube count gives
$$
    \|\1_{[M]}\|_{U^r(\mathbb Z)}
    \asymp_r M^{(r+1)/2^r}.
$$
Moreover, with $g$ extended by zero to $\mathbb Z$, expanding the Gowers
norms gives the exact identity
$$
    \sum_{h\in\mathbb Z}
    \|\Delta_hg\|_{U^r(\mathbb Z)}^{2^r}
    =
    \|g\|_{U^{r+1}(\mathbb Z)}^{2^{r+1}}.
$$
Only the shifts $|h|<M$ contribute. Hence H\"older's inequality gives
\begin{align*}
    \frac{1}{M}\sum_{|h|<M}\|\Delta_hg\|_{U^r[M]}
    &\le
    \frac{(2M-1)^{1-2^{-r}}}
         {M\|\1_{[M]}\|_{U^r(\mathbb Z)}}
    \left(
    \sum_h\|\Delta_hg\|_{U^r(\mathbb Z)}^{2^r}
    \right)^{2^{-r}}\\
    &=
    \frac{(2M-1)^{1-2^{-r}}
          \|\1_{[M]}\|_{U^{r+1}(\mathbb Z)}^2}
         {M\|\1_{[M]}\|_{U^r(\mathbb Z)}}
    \|g\|_{U^{r+1}[M]}^2\\
    &\ll_r\|g\|_{U^{r+1}[M]}^2.
\end{align*}
This closes the induction.
\end{proof}

\subsection{A local inverse theorem}

We use the following form of the Shao--Wang inverse theorem
\cite[Theorem~2.4]{ShW25}, extending work of \cite{PP22}. We record it because the density increment
argument needs auxiliary integers $q'$ and $b$ that work for every shifted
input.

\begin{lemma}[Local inverse theorem]\label{l:SW-local-form}
Let $C\ge1$, let $q\in\mathbb N$, and let
$P_1,\ldots,P_m\in\mathbb Z[y]$ be polynomials of distinct degrees with
$(C,q)$-coefficients and $\deg(P_1)<\cdots<\deg(P_m)=d$.
Let $N,M\in\mathbb N$ satisfy $M\le (N/q^{d-1})^{1/d}$, and let
$0<\rho<1/2$. Let $F_0,\ldots,F_m\colon\mathbb Z\to\mathbb C$ be
$1$-bounded functions supported on $[N]$. Suppose that
$$
\left|
\frac{1}{N}\sum_x\mathbb E_{y\le M}
F_0(x)F_1(x+P_1(y))\cdots F_m(x+P_m(y))
\right|\ge \rho .
$$
Then either
$$
    M\ll_{C,d}(q/\rho)^{O_d(1)},
$$
or there exist positive integers
$$
    q'\ll_{C,d}\rho^{-O_d(1)},
    \qquad b=O_d(1),
$$
and, for each $1\le i\le m$, a $1$-bounded function
$\phi_i\colon\mathbb Z\to\mathbb C$ such that
$$
    \left|\sum_x F_i(x)\phi_i(x)\right|
    \gg_{C,d}\rho^{O_d(1)}N,
$$
with $\phi_i$ having the Lipschitz property
$$
    |\phi_i(x+q'q^b y)-\phi_i(x)|
    \ll_{C,d}
    (q/\rho)^{O_d(1)}M^{-\deg(P_i)}|y|
    \qquad\text{for all }x,y\in\mathbb Z.
$$
The same integers $q'$ and $b$ work simultaneously for all
$1\le i\le m$.
\end{lemma}

\begin{proof}
If
$$
    M\ll_{C,d}(q/\rho)^{O_d(1)},
$$
there is nothing to prove. Otherwise, \cite[Theorem~2.4]{ShW25} gives, for
each $1\le i\le m$, a function $\phi_i$ and integers $q'_i,b_i$, which
we initially allow to depend on $i$. Set
$$
    q'\coloneqq\prod_{i=1}^m q'_i,
    \qquad
    b\coloneqq\max_{1\le i\le m}b_i.
$$
Since the degrees are distinct, $m\le d$, and hence
$q'\ll_{C,d}\rho^{-O_d(1)}$ and $b=O_d(1)$. For the function supplied
for the $i$th input, apply its Lipschitz estimate with
$$
    y_i\coloneqq(q'/q'_i)q^{b-b_i}y.
$$
Then $q'_iq^{b_i}y_i=q'q^by$, and the extra multiplier
$(q'/q'_i)q^{b-b_i}$ is at most $(q/\rho)^{O_d(1)}$. It is therefore
absorbed into the stated Lipschitz constant. The correlations are unchanged,
which proves the result.
\end{proof}

\section{Progressions with multiples of a common polynomial}
\label{s:common-polynomial}

We first prove a counting statement that is uniform over the $W$-tricked
family $P_W$ (cf. Section~\ref{ss:overview-common-polynomial}). We then
establish the prime model comparison needed to transfer this
count to shifted primes. Combining the two proves
Theorem~\ref{thm:fixed-polynomial-shifted-prime}, and
Corollary~\ref{thm_progression} follows by taking $P(y)=y$ and $a_i=i$.

\subsection{Uniform counting for a fixed polynomial}
\label{ss:fixed-polynomial-counting}

For a nonzero polynomial $P\in\mathbb Z[y]$ with $P(0)=0$, write
$$
r=\operatorname{ord}_0P
\qquad\text{and}\qquad
d=\deg P.
$$
Thus
$$
P(y)=\sum_{j=r}^d b_jy^j,
\qquad b_rb_d\ne0.
$$

We first record the quantitative counting consequence of
Altman--Sawhney that we need. Let $w$ be a sufficiently large integer, and
set
\begin{equation}\label{e:AS-saturated-W}
W_0=W_{d,w}
\coloneqq
\prod_{\substack{p\le w\\(p,d)=1}}p
\prod_{\substack{p\le w\\(p,d!)>1}}p^{2d}
\prod_{p\le w^{1/2}}p^{\lceil w^{1/3}\rceil}.
\end{equation}
For $w$ sufficiently large
in terms of $d$, this modulus satisfies
\begin{equation}\label{e:AS-W-properties}
8\prod_{p\le w}p\mid W_0,
\qquad P^+(W_0)\le w,
\qquad \mathcal L(w)\mid W_0,
\qquad \log W_0\ll_d w,
\end{equation}
where $\mathcal{L}(w)$ is as in~\eqref{e:L(w)}.

Set $B=|b_r|$, $\sigma=\operatorname{sgn}(b_r)$, and
\begin{equation}\label{e:AS-normalized-polynomial}
\widetilde P(y)\coloneqq \sigma B^{-r-1}P(By)\in\mathbb Z[y].
\end{equation}
The coefficient of $y^r$ in $\widetilde P$ is $1$. Set
\begin{equation}\label{e:AS-total-W}
W\coloneqq BW_0
\end{equation}
and define
\begin{equation}\label{e:AS-PW}
P_W(y)\coloneqq W^{-r}P(Wy)
=\sum_{j=r}^d b_jW^{j-r}y^j\in\mathbb Z[y].
\end{equation}
Observe that
\begin{equation}\label{e:AS-normalization-identity}
P_W(y)=\sigma B\,\widetilde P_{W_0}(y),
\qquad
\widetilde P_{W_0}(y)\coloneqq W_0^{-r}\widetilde P(W_0y).
\end{equation}

For functions $f_0,\ldots,f_{k-1}\colon\mathbb Z\to\mathbb C$ supported on
$[X]$, define
\begin{equation}\label{e:AS-counting-form}
\begin{split}
\Lambda_{P,W,X}(f_0,\ldots,f_{k-1})
\coloneqq{}&2\int_{1/2}^1
\mathbb E_{x\le X}
\mathbb E_{y\le M_z}
\prod_{i=0}^{k-1}f_i\bigl(x+a_iP_W(y)\bigr)\,\mathrm dz,\\
M_z\coloneqq{}&
\left\lfloor
\left(\frac{zBX}{|b_d|W^{d-r}}\right)^{1/d}
\right\rfloor,
\qquad a_0\coloneqq0.
\end{split}
\end{equation}
If $\widetilde b_d$ denotes the leading coefficient of $\widetilde P$, then
\begin{equation}\label{e:AS-leading-scale-identity}
|\widetilde b_d|W_0^{d-r}
=|b_d|B^{d-r-1}W_0^{d-r}
=\frac{|b_d|W^{d-r}}{B}.
\end{equation}
Consequently, the range in \eqref{e:AS-counting-form} is exactly the natural
range in \cite[Definition~2.5]{AS25}.
As usual, a function supported on $[X]$ is extended by zero to all of
$\mathbb Z$.

We also use the model operator from
\cite[Definition~2.5]{AS25}.  In the present normalization it is
\begin{equation}\label{e:AS-model-form}
\begin{split}
\Lambda_{\mathrm{Model}}(f_0,\ldots,f_{k-1})
\coloneqq{}&
\mathbb E_{\substack{x\le X\\
y\le (X/W_0^2)^{1/(2r)}\\
X^{1/2}/2\le z\le X^{1/2}}}
\nu(y)\prod_{i=0}^{k-1}
f_i\bigl(x+\widetilde a_i(\widetilde\epsilon zW_0+1)y^r\bigr),\\
\nu(y)\coloneqq{}&
\frac{r}{d}\left(\frac{X}{W_0^2}\right)^{(d-r)/(2dr)}
y^{-(d-r)/d},
\end{split}
\end{equation}
where $\widetilde a_i=\sigma Ba_i$ and
$\widetilde\epsilon=\widetilde b_d/|\widetilde b_d|$.  The variables $y$
and $z$ in this display are averaged over the indicated integer intervals.

\begin{proposition}[Uniform polynomial counting]
\label{p:AS-uniform-counting}
Fix $k\ge3$, a nonzero polynomial $P\in\mathbb Z[y]$ with $P(0)=0$, and
distinct nonzero integers $a_1,\ldots,a_{k-1}$.  Let $r$ and $d$ be as defined
above.  There exist constants
$C^\sharp\ge1$ and $w^\sharp\ge2$ such that the following
holds. Let $X\in\mathbb N$ and $0<\delta\le1/2$, and define
\begin{equation}\label{e:AS-statistical-size}
\mathfrak m(\delta)\coloneqq
\begin{cases}
\exp\bigl(-C^\sharp(\log(2/\delta))^{C^\sharp}\bigr),
&r=1,\ k=3,\\[1mm]
\exp\bigl(-\exp(C^\sharp(\log(2/\delta))^{C^\sharp})\bigr),
&r=1,\ k\ge4,\\[1mm]
\exp\bigl(-\exp(C^\sharp\delta^{-C^\sharp})\bigr),
&r\ge2.
\end{cases}
\end{equation}
Let $w\ge w^\sharp$, and let $W_0$, $W$, and $P_W$ be given by
\eqref{e:AS-saturated-W}, \eqref{e:AS-total-W}, and \eqref{e:AS-PW},
respectively. Suppose that
\begin{align}
w&\ge
\exp\left((\log(2/\mathfrak m(\delta)))^{C^\sharp}\right),
\label{e:AS-counting-condition-w}\\
X&\ge
W^{C^\sharp}
\exp\left((\log(2/\mathfrak m(\delta)))^{C^\sharp}\right),
\label{e:AS-counting-condition-X}\\
\frac{X^{1/2}}{W}&\ge
\mathfrak m(\delta)^{-C^\sharp}.
\label{e:AS-counting-condition-model}
\end{align}
If $f\colon\mathbb Z\to[0,1]$ is supported on $[X]$ and
$\mathbb E_{x\le X}f(x)\ge\delta$, then
\begin{equation}\label{e:AS-uniform-counting-conclusion}
\Lambda_{P,W,X}(f,\ldots,f)
\ge\mathfrak m(\delta).
\end{equation}
Here $C^\sharp$ and $w^\sharp$ may depend on the fixed polynomial $P$ and the
fixed coefficients $a_i$, but not on $w$.

In particular, for every fixed $\gamma>0$ there are
$c_\gamma>0$, $C_\gamma\ge1$, and $w_\gamma\ge2$, depending only on the
fixed data and $\gamma$, such that the following holds. If $w\ge w_\gamma$,
\begin{equation}\label{e:AS-simplified-size}
X\ge\exp(w^{C_\gamma})
\end{equation}
and
\begin{equation}\label{e:AS-simplified-density}
\delta\ge
\begin{cases}
\exp\bigl(-(\log w)^{c_\gamma}\bigr),
&r=1,\ k=3,\\[1mm]
\exp\bigl(-(\log\log w)^{c_\gamma}\bigr),
&r=1,\ k\ge4,\\[1mm]
(\log\log w)^{-c_\gamma},
&r\ge2,
\end{cases}
\end{equation}
then
\begin{equation}\label{e:AS-simplified-conclusion}
\Lambda_{P,W,X}(f,\ldots,f)>w^{-\gamma}.
\end{equation}
\end{proposition}

\begin{remark}\label{r:AS-headline-not-enough}
The final density theorem in~\cite[Theorem~1.1]{AS25} does not by itself imply
Proposition~\ref{p:AS-uniform-counting}. Indeed, $P_W$ varies with $W$, while
the implicit constant and the lower threshold for $X$ in a theorem stated for
a fixed polynomial may depend on all its coefficients. The parameter uniform
counting statement above is therefore needed.
\end{remark}

\begin{proof}
We use \eqref{e:AS-normalization-identity} and
\eqref{e:AS-leading-scale-identity} throughout.  Thus
$a_iP_W=\widetilde a_i\widetilde P_{W_0}$, where
\begin{align*}
\widetilde a_i=\sigma Ba_i,
\end{align*}
and the range in \eqref{e:AS-counting-form} is exactly the range in
\cite[Definition~2.5]{AS25} for the fixed normalized polynomial
$\widetilde P$.  We enlarge $w^\sharp$ so that $w$ dominates every
coefficient of $\widetilde P$ and every $\widetilde a_i$.

\emph{The comparison estimate.}
We claim that there is a constant $C\ge1$, depending only on the fixed data, such that for
$0<\eta<1/2$ one has
\begin{equation}\label{e:AS-comparison-claim}
\left|
\Lambda_{P,W,X}(f_0,\ldots,f_{k-1})
-\Lambda_{\mathrm{Model}}(f_0,\ldots,f_{k-1})
\right|\le\eta
\end{equation}
for all $1$-bounded inputs supported on $[X]$, provided
\begin{equation}\label{e:AS-comparison-claim-hypotheses}
w\ge\exp\bigl((\log(2/\eta))^C\bigr)
\quad\text{and}\quad
X\ge W^C\exp\bigl((\log(2/\eta))^C\bigr).
\end{equation}
We now prove this estimate. Write
\begin{align*}
\Lambda_{\mathrm{Diff}}
\coloneqq \Lambda_{P,W,X}-\Lambda_{\mathrm{Model}}.
\end{align*}
For $0\le j\le k-1$ and $1$-bounded functions $g_i$ supported on $[X]$,
define
\begin{align*}
D_j(x)\coloneqq{}&\1_{[X]}(x)
\left(
2\int_{1/2}^1\mathbb E_{y\le M_z}
\prod_{i\ne j}g_i\bigl(x+(\widetilde a_i-\widetilde a_j)
\widetilde P_{W_0}(y)\bigr)\,\mathrm dz\right.\\
&\left.\hspace{16mm}
-\mathbb E_{\substack{y\le (X/W_0^2)^{1/(2r)}\\
X^{1/2}/2\le z\le X^{1/2}}}
\nu(y)\prod_{i\ne j}g_i\bigl(x+(\widetilde a_i-\widetilde a_j)
(\widetilde\epsilon zW_0+1)y^r\bigr)
\right).
\end{align*}
Translating $x$ by $-\widetilde a_j\widetilde P_{W_0}(y)$ in the polynomial
term and by $-\widetilde a_j(\widetilde\epsilon zW_0+1)y^r$ in the model
term, with all functions extended by zero, gives the exact identities
\begin{align}\label{e:Lambdadiff}
\Lambda_{\mathrm{Diff}}(g_0,\ldots,g_{k-1})
&=\mathbb E_{x\le X}D_j(x)g_j(x),\\
\mathbb E_{x\le X}|D_j(x)|^2
&=\Lambda_{\mathrm{Diff}}
(g_0,\ldots,g_{j-1},\overline{D_j},g_{j+1},\ldots,g_{k-1}).
\end{align}
Also $D_j$ is supported on $[X]$ and $\|D_j\|_\infty\ll_d1$, since
$\mathbb E_{y\le (X/W_0^2)^{1/(2r)}}\nu(y)\ll_d1$.

Suppose that the left-hand side of \eqref{e:AS-comparison-claim} is larger
than $\eta$, and set $\eta_0=\eta$.  At the $j$th stage, let $g_i$ be the
inputs obtained at the preceding stages and suppose that
$|\Lambda_{\mathrm{Diff}}(g_0,\ldots,g_{k-1})|\ge\eta_j$.  Then~\eqref{e:Lambdadiff} and Cauchy--Schwarz give
\begin{align*}
\mathbb E_{x\le X}|D_j(x)|^2\ge\eta_j^2.
\end{align*}
Choose a fixed constant $C_D\gg_d1$ such that
$G_j\coloneqq C_D^{-1}\overline{D_j}$ is $1$-bounded. Applying
\cite[Lemmas~2.6 and~2.7]{AS25} to~\eqref{e:Lambdadiff} therefore gives
\begin{align*}
\|G_j\|_{U^s_{\mathrm{GP}}
[X;\mu_{W_0^{d-r}\cdot[\pm X/W_0^{d-r}]}]}
\gg\eta_j^{C_0}
\qquad\text{or}\qquad
\|G_j\|_{U^s_{\mathrm{GP}}
[X;\mu_{[\pm X]}]}
\gg\eta_j^{C_0},
\end{align*}
for fixed $s\ge2$ and $C_0$ depending only on the fixed data. In the first
case set $q_j=W_0^{d-r}$, and in the second set $q_j=1$. There is a constant
$c_1>0$, depending only on the fixed data, such that the relevant norm is at
least $\kappa_j\coloneqq c_1\eta_j^{C_0}$.

We now restrict $G_j$ to residue classes modulo $q_j$. After increasing $C$
in \eqref{e:AS-comparison-claim-hypotheses}, we have $X\ge2q_j$. For $c$
modulo $q_j$, let
$$
    I_c\coloneqq\{n\in\mathbb Z:c+q_jn\in[X]\}.
$$
Thus $|I_c|=X/q_j+O(1)$. After translating $I_c$ to $[|I_c|]$, denote the
resulting restriction of $G_j(c+q_j\,\cdot)$ by $G_{j,c}$, and set
$$
    u_c\coloneqq
    \|G_{j,c}\|_{U^s_{\mathrm{GP}}
    [|I_c|;\mu_{[\pm X/q_j]}]}.
$$
Since the support of $\mu_{q_j\cdot[\pm X/q_j]}$ consists of multiples of
$q_j$, the definition of the Gowers--Peluse norm gives the exact identity
$$
\|G_j\|_{U^s_{\mathrm{GP}}[X;\mu_{q_j\cdot[\pm X/q_j]}]}^{2^s}
=\sum_{c\bmod q_j}\frac{|I_c|}{X}u_c^{2^s}.
$$
Every $u_c$ is at most $1$. Hence, for
$$
    \mathcal C_j\coloneqq
    \{c\bmod q_j:u_c\ge\kappa_j/2\},
$$
we have
\begin{equation}\label{e:AS-good-class-mass}
    \sum_{c\in\mathcal C_j}\frac{|I_c|}{X}
    \ge(1-2^{-2^s})\kappa_j^{2^s};
\end{equation}
indeed, the contribution of the complementary classes is at most
$(\kappa_j/2)^{2^s}$, whereas every $u_c^{2^s}$ is at most $1$.

Fix $c\in\mathcal C_j$. We have the pointwise measure bound
$$
    \mu_{[\pm X/q_j]}\le2\mu_{[\pm|I_c|]}.
$$
Hence \cite[Lemma~B.1(ii)--(iii)]{AS25} gives
$$
    \|G_{j,c}\|_{U^s[|I_c|]}\gg_s\kappa_j^{O_s(1)}.
$$
Since $q_j\le W_0^{d-r}\le W^{d-r}$, the hypotheses in
\eqref{e:AS-comparison-claim-hypotheses}, with $C$ increased if necessary,
make $|I_c|=X/q_j+O(1)$ large enough for the inverse theorem below.
The Leng--Sah--Sawhney quasipolynomial inverse theorem
\cite[Theorem~2.13]{AS25}, applied separately for each
$c\in\mathcal C_j$, supplies a $1$-bounded nilsequence $\Phi_{j,c}$ on
$\mathbb Z$ for which
\begin{align*}
\left|
\mathbb E_{n\in I_c}D_j(c+q_jn)\Phi_{j,c}(n)
\right|
\ge\varepsilon_j,
\qquad
\varepsilon_j\ge
\exp\bigl(- (\log(2/\kappa_j))^{C_0}\bigr),
\end{align*}
after increasing $C_0$. Here we have conjugated the correlation supplied by
the inverse theorem and restored the fixed factor $C_D$. We multiply each
$\Phi_{j,c}$ by a
unit complex number so that its local correlation is real and nonnegative.
Define
\begin{align*}
F_j(c+q_jn)\coloneqq
\begin{cases}
\Phi_{j,c}(n),&c\in\mathcal C_j,\\
0,&c\notin\mathcal C_j,
\end{cases}
\qquad n\in\mathbb Z.
\end{align*}
Thus $F_j$ is $1$-bounded, and the zero function is the nilsequence used on
each discarded class. The phase choices make the local correlations add
together. By \eqref{e:AS-good-class-mass},
\begin{align*}
\mathbb E_{x\le X}D_j(x)F_j(x)
&=\sum_{c\in\mathcal C_j}\frac{|I_c|}{X}
\mathbb E_{n\in I_c}D_j(c+q_jn)\Phi_{j,c}(n)
\gg_s\kappa_j^{2^s}\varepsilon_j,
\end{align*}
which is at least $\exp(- (\log(2/\eta_j))^{C_0})$ after one further
increase of $C_0$.

Set
\begin{equation}\label{e:AS-classwise-modulus}
Q\coloneqq W_0^{d-r}.
\end{equation}
When $q_j=Q$, the function $F_j$ is a
nilsequence on every residue class modulo $Q$, including the zero
nilsequences on the discarded classes. When $q_j=1$, it is a global
nilsequence, and its restrictions to the residue classes modulo $Q$ are
obtained by precomposing its polynomial orbit with an affine map. Thus in
both cases the restriction of $F_j$ to every residue class modulo $Q$ is a
nilsequence. We may choose $\eta_{j+1}$ so that
\begin{align*}
\left|
\mathbb E_{x\le X}D_j(x)F_j(x)
\right|&\ge\eta_{j+1},\\
\eta_{j+1}&\coloneqq
\exp\bigl(- (\log(2/\eta_j))^{C_0}\bigr).
\end{align*}
The classwise nilsequences have dimension at most
$(\log(2/\eta_j))^{C_0}$ and complexity at most
$\exp((\log(2/\eta_j))^{C_0})$. These bounds are uniform in $Q$, even though the underlying nilmanifolds, Lipschitz functions, and polynomial maps may depend on $Q$ and on the residue class.

Moreover,
\begin{align*}
\bigl|\Lambda_{\mathrm{Diff}}
(g_0,\ldots,g_{j-1},F_j\1_{[X]},g_{j+1},\ldots,g_{k-1})\bigr|
=\bigl|\mathbb E_{x\le X}D_j(x)F_j(x)\bigr|
\ge\eta_{j+1}.
\end{align*}
We may therefore repeat this for $j=0,\ldots,k-1$. As $k$ is fixed, we
eventually obtain
\begin{align*}
\left|\Lambda_{\mathrm{Diff}}
(F_0\1_{[X]},\ldots,F_{k-1}\1_{[X]})\right|&\ge\rho,\\
\rho&\ge\exp\bigl(- (\log(2/\eta))^{C_1}\bigr),\\
\mathcal D&\le(\log(2/\eta))^{C_1},
\qquad
\mathcal M\le\exp\bigl((\log(2/\eta))^{C_1}\bigr),
\end{align*}
where, uniformly in $Q$, $\mathcal D$ bounds the dimensions and
$\mathcal M$ bounds the filtered nilmanifold complexities and the Lipschitz norms of the nilsequence restrictions of the functions $F_j$ to the residue classes modulo $Q$. 
This classwise conclusion is sufficient: after fixing $x_0$ below and
conditioning the argument of $H$ modulo $Q$, every factor is restricted to
one residue class and the factors combine into a single nilsequence. (We do
not patch the classwise nilsequences globally, since doing so would introduce
$Q$ into the quantitative complexity bounds.)

We next remove the cutoffs $\1_{[X]}$ attached to the functions
$F_j$ without a loss depending on $X$. Set $u=\rho/(100k)$. By a standard
Fourier approximation lemma \cite[Lemma~B.2]{AS25}, there exists a function
$\psi_u\colon\mathbb Z\to[0,1]$ such that
\begin{align*}
\psi_u(n)&=1 &&\text{for }n\in[uX,(1-u)X],\\
\psi_u(n)&=0 &&\text{for }n\notin[-uX,(1+u)X],
\end{align*}
and
\begin{align*}
\psi_u(n)&=\int_0^1\widehat\psi_u(\theta)e(\theta n)\,\mathrm d\theta,
&
\int_0^1|\widehat\psi_u(\theta)|\,\mathrm d\theta
&\ll u^{-O(1)}.
\end{align*}
The size condition needed to apply the lemma follows from
\eqref{e:AS-comparison-claim-hypotheses}, after increasing $C$.

For fixed polynomial variables, replacing each $\1_{[X]}$ by $\psi_u$
changes the polynomial form only when one of the shifted arguments lies
within distance $uX$ of an endpoint of $[X]$. There are at most $O(kuX)$
such values of the base point, so this contribution is $O(ku)$. The
contribution to the model form is also $O_d(ku)$, since
\begin{align*}
\mathbb E_{1\le y\le (X/W_0^2)^{1/(2r)}}\nu(y)\ll_d1.
\end{align*}
We may therefore replace all the cutoffs by $\psi_u$ and retain a discrepancy
of size $\gg\rho$. Fourier inversion and the triangle inequality then give
frequencies which may be absorbed into the classwise nilsequences, so that,
after renaming the resulting functions,
\begin{align}\label{e:AS-cutoff-free-discrepancy}
\left|\Lambda_{\mathrm{Diff}}(F_0,\ldots,F_{k-1})\right|&\ge\rho',\\
\rho'&\ge\exp\bigl(-(\log(2/\eta))^{C_2}\bigr).\notag
\end{align}
The same bounds for $\mathcal D$ and $\mathcal M$ continue to hold after
increasing $C_2$.

Expanding \eqref{e:AS-cutoff-free-discrepancy} and averaging over $x$, we may
fix $x_0$ such that the absolute value of the resulting inner difference is
at least $\rho'$. Set
\begin{align*}
H(n)\coloneqq\prod_{i=0}^{k-1}F_i(x_0+\widetilde a_i n).
\end{align*}
With this notation,
\begin{equation}\label{e:AS-one-variable-discrepancy}
\Biggl|
2\int_{1/2}^1\mathbb E_{y\le M_z}
H\bigl(\widetilde P_{W_0}(y)\bigr)\,\mathrm dz
-
\mathbb E_{\substack{
y\le (X/W_0^2)^{1/(2r)}\\
X^{1/2}/2\le z\le X^{1/2}}}
\nu(y)H\bigl((\widetilde\epsilon zW_0+1)y^r\bigr)
\Biggr|
\ge\rho'.
\end{equation}
We split each of these averages according to the residue class modulo
$W_0^{d-r}$ of the argument of $H$.
The hypotheses of \cite[Lemma~2.2]{AS25} hold for
$\widetilde P_{W_0}$: its coefficient of $y^r$ is $1$, while every higher
coefficient has the required divisibility by the primes below $w$.
They also hold for $(\widetilde\epsilon zW_0+1)y^r$, where
$\widetilde\epsilon=\widetilde b_d/|\widetilde b_d|$ as defined after
\eqref{e:AS-model-form}, since its coefficient is congruent to $1$ modulo
all the required prime powers. Hence, on every complete block of length
$W_0^{d-r}$, the two polynomials assign the same proportion to each residue
class modulo $W_0^{d-r}$, and this proportion is independent of $z$. For
each fixed $z$, truncate each $y$-range in
\eqref{e:AS-one-variable-discrepancy} to the
largest union of such complete blocks. In the first average the discarded
proportion is $O_d(W_0^{C_d}X^{-1/d})$ for some $C_d>0$. In the second average, the discarded
values of $y$ lie at the upper end of the range, where $\nu(y)\ll_d1$, so
their total contribution is
\begin{align*}
O_d\left(\frac{W_0^{d-r}}{(X/W_0^2)^{1/(2r)}}\right).
\end{align*}
Every residue class in the image of $\widetilde P_{W_0}$ modulo
$W_0^{d-r}$ has complete block density at least $W_0^{-(d-r)}$. Thus,
after conditioning on a residue class, the changes in the conditional
averages and in their class weights are $O_d(W_0^{C_d}X^{-c_d})$ for some
$c_d>0$, after increasing $C_d$ if necessary.
For every residue class $\ell$ in this image, define
\begin{align*}
\mathcal A_\ell\coloneqq{}&
2\int_{1/2}^1
\mathbb E_{\substack{y\le M_z\\
\widetilde P_{W_0}(y)\equiv\ell\pmod{W_0^{d-r}}}}
H(\widetilde P_{W_0}(y))\,\mathrm dz,\\
\mathcal B_\ell\coloneqq{}&
\mathbb E_{\substack{
y\le (X/W_0^2)^{1/(2r)},\ X^{1/2}/2\le z\le X^{1/2}\\
(\widetilde\epsilon zW_0+1)y^r
\equiv\ell\pmod{W_0^{d-r}}}}
\nu(y)H((\widetilde\epsilon zW_0+1)y^r).
\end{align*}
On the complete blocks, \cite[Lemma~2.2]{AS25} gives common nonnegative
class weights whose sum is $1$. Consequently, the expression inside the
absolute value in \eqref{e:AS-one-variable-discrepancy} is, up to an error
$O_d(W_0^{C_d}X^{-c_d})$, a convex combination of
$\mathcal A_\ell-\mathcal B_\ell$ over these residue classes. After increasing $C$ in
\eqref{e:AS-comparison-claim-hypotheses}, this error is smaller than a
sufficiently small fixed multiple of $\rho'$. Hence there is a residue class
$\ell$ in the image of $\widetilde P_{W_0}$ such that
\begin{align*}
|\mathcal A_\ell-\mathcal B_\ell|\gg\rho'.
\end{align*}
The residue class $\ell$ is admissible in the terminology preceding
\cite[Lemma~2.9]{AS25},
since by construction there exists $y$ such that
\begin{align*}
\widetilde P_{W_0}(y)\equiv\ell\pmod{W_0^{d-r}}.
\end{align*}

Recall the choice of $Q$ from \eqref{e:AS-classwise-modulus} and write
$n=\ell+Qt$. For each $i$, the argument
$x_0+\widetilde a_i n$ then stays in a fixed residue class modulo $Q$ and is
affine in $t$. Hence the classwise nilsequence representing $F_i$ on that
class remains a nilsequence after this affine precomposition, and their
product is a single nilsequence in $t$ on a product nilmanifold. Its dimension is
$O_k(\mathcal D)$ and its complexity is $\mathcal M^{O_k(1)}$, uniformly in $Q$; affine
reparametrization on a progression does not change these bounds.  We may
enlarge the complexity by a fixed
constant so that it dominates the coefficients of $\widetilde P$.
The two size hypotheses in \cite[Lemma~2.9]{AS25} are now implied by
\begin{align*}
w&\ge\exp\bigl((\log(2/\eta))^C\bigr),\\
X&\ge W_0^C\exp\bigl((\log(2/\eta))^C\bigr),
\end{align*}
because
\begin{align*}
(\mathcal M/\rho')^{O(\mathcal D^{O(1)})}
\le\exp\bigl((\log(2/\eta))^{O(1)}\bigr).
\end{align*}
These bounds follow from \eqref{e:AS-comparison-claim-hypotheses}, since
$W_0\le W$. The nilsequence comparison result of
\cite[Lemma~2.9]{AS25}, applied with a sufficiently small fixed multiple of
$\rho'$ as its error parameter, contradicts
$|\mathcal A_\ell-\mathcal B_\ell|\gg\rho'$. This proves
\eqref{e:AS-comparison-claim}. Here, the coefficient hypothesis in~\cite[Lemma~2.9]{AS25} concerns the coefficients of the fixed polynomial $\widetilde P$, which are independent of $Q$, instead of the coefficients of the polynomial map defining the nilsequence. The growing
modulus $W_0$ occurs separately in the size
hypotheses.

\emph{The model lower bound.}
We record explicitly the quantitative counting input used here.
There are constants $c_P>0$ and $C_3\ge1$, depending only on the fixed
data, with the following property. Define
$$
\mathfrak b(\delta)=
\begin{cases}
\exp\bigl(-C_3(\log(2/\delta))^{C_3}\bigr),&r=1,\ k=3,\\
\exp\bigl(-\exp(C_3(\log(2/\delta))^{C_3})\bigr),&r=1,\ k\ge4,\\
\exp\bigl(-\exp(C_3\delta^{-C_3})\bigr),&r\ge2
\end{cases}.
$$
Let $s\in\{-1,1\}$, let $I$ be an interval of cardinality $H$, and let
$g\colon\mathbb Z\to[0,1]$ be supported on $I$ with
$\mathbb E_{n\in I}g(n)\ge\delta/2$. Assume that
$H\ge\mathfrak b(\delta)^{-C_3}$. Then
\begin{equation}\label{e:AS-homogeneous-counting}
\mathbb E_{n\in I}\mathbb E_{1\le y\le c_PH^{1/r}}
\prod_{i=0}^{k-1}g(n+s\widetilde a_i y^r)
\ge\mathfrak b(\delta).
\end{equation}
This is the quantitative counting input
collected in \cite[Lemma~2.11]{AS25}. In the case where $g$ is an indicator of
a set, the case $r=1$, $k=3$ follows from
Kelley--Meka~\cite{Kelley-Meka}, the remaining linear estimates from
Green--Tao~\cite{GT09} and Leng--Sah--Sawhney~\cite{LSS-Szemeredi}, and the
homogeneous polynomial estimate from Prendiville~\cite{prendiville} and the
standard Varnavides averaging.  The functional statement follows from the
set version by applying it to
$\{n\in I:g(n)\ge\delta/4\}$, which has density at least $\delta/4$, and
absorbing the factor $(\delta/4)^k$ into $\mathfrak b(\delta)$.

We apply this input to the model form. For an integer
$z\in[X^{1/2}/2,X^{1/2}]$, set
\begin{align*}
q_z\coloneqq|\widetilde\epsilon zW_0+1|,
\qquad
s_z\coloneqq\operatorname{sgn}(\widetilde\epsilon zW_0+1),
\end{align*}
and, for $c$ modulo $q_z$, define
\begin{align*}
I_{c,z}&\coloneqq\{n\in\mathbb Z:c+q_zn\in[X]\},\\
f_{c,z}(n)&\coloneqq f(c+q_zn).
\end{align*}
For fixed $z$ and $y$ we have the exact identity
$$
\mathbb E_{x\le X}\prod_{i=0}^{k-1}
f\bigl(x+s_zq_z\widetilde a_i y^r\bigr)
=
\frac{1}{X}\sum_{c\bmod q_z}\sum_{n\in I_{c,z}}
\prod_{i=0}^{k-1}f_{c,z}(n+s_z\widetilde a_i y^r).
$$
Moreover, $q_z\asymp W_0X^{1/2}$, and hence
\begin{align*}
|I_{c,z}|\asymp\frac{X^{1/2}}{W_0}
\end{align*}
uniformly in $c$ and $z$.  If
$\alpha_{c,z}=\mathbb E_{n\in I_{c,z}}f_{c,z}(n)$, then
\begin{align*}
\sum_{c\bmod q_z}|I_{c,z}|\alpha_{c,z}
=\sum_{x\le X}f(x)\ge\delta X.
\end{align*}
Since $0\le\alpha_{c,z}\le1$, the classes $c\bmod q_z$ with
$\alpha_{c,z}\ge\delta/2$ satisfy
$$
\sum_{\substack{c\bmod q_z\\\alpha_{c,z}\ge\delta/2}}
|I_{c,z}|\ge\frac{\delta X}{2}.
$$

For every such class, \eqref{e:AS-homogeneous-counting} applies because
by \eqref{e:AS-counting-condition-model},
\begin{align*}
\frac{X^{1/2}}{W_0}=B\frac{X^{1/2}}{W}
\ge B\mathfrak m(\delta)^{-C^\sharp}.
\end{align*}
After increasing $C^\sharp$, this is larger than
$\mathfrak b(\delta)^{-C_3}$. Since
\begin{align*}
|I_{c,z}|\asymp\frac{X^{1/2}}{W_0}
\qquad\text{and}\qquad
\left(\frac{X}{W_0^2}\right)^{1/(2r)}
=\left(\frac{X^{1/2}}{W_0}\right)^{1/r},
\end{align*}
after decreasing $c_P$ by a fixed factor we have
\begin{align*}
c_P|I_{c,z}|^{1/r}
\le\left(\frac{X}{W_0^2}\right)^{1/(2r)}.
\end{align*}
This subrange has length $\asymp(X/W_0^2)^{1/(2r)}$. As the summands are
nonnegative, retaining this subrange in the average over
$y\le(X/W_0^2)^{1/(2r)}$ loses only a fixed factor. Summing over the dense
residue classes $c$ gives
\begin{align*}
\mathbb E_{x\le X}\mathbb E_{y\le (X/W_0^2)^{1/(2r)}}
\prod_{i=0}^{k-1}f\bigl(x+\widetilde a_i
(\widetilde\epsilon zW_0+1)y^r\bigr)
\gg\delta\mathfrak b(\delta).
\end{align*}
Finally, for all $1\le y\le(X/W_0^2)^{1/(2r)}$, we have
\begin{align*}
\nu(y)
=\frac{r}{d}
\left(
\frac{(X/W_0^2)^{1/(2r)}}{y}
\right)^{(d-r)/d}
\ge\frac{r}{d}.
\end{align*}
Choosing $C^\sharp$ sufficiently large in terms of $C_3$ ensures
$\delta\mathfrak b(\delta)\gg\mathfrak m(\delta)$.  Averaging in $z$ and
adjusting this choice by a fixed amount therefore gives
\begin{equation}\label{e:AS-model-lower}
\Lambda_{\mathrm{Model}}(f,\ldots,f)
\ge2\mathfrak m(\delta).
\end{equation}

Take $\eta=\mathfrak m(\delta)$. Conditions
\eqref{e:AS-counting-condition-w} and
\eqref{e:AS-counting-condition-X} permit us to use
\eqref{e:AS-comparison-claim}; combining it with
\eqref{e:AS-model-lower} proves
\eqref{e:AS-uniform-counting-conclusion}.

Finally, fix $\gamma>0$.  Since $\mathfrak m$ is increasing, it suffices to
consider the lower endpoints in \eqref{e:AS-simplified-density}.  We choose
$c_\gamma>0$ sufficiently small that in all three cases, for $w$
sufficiently large,
\begin{align*}
\log(2/\mathfrak m(\delta))
\le(\log w)^{1/(2C^\sharp)}
\qquad\text{and}\qquad
\mathfrak m(\delta)>w^{-\gamma}.
\end{align*}
The first inequality implies \eqref{e:AS-counting-condition-w}.  Since
$W\le\exp(O_{P,d}(w))$, choosing $C_\gamma$ sufficiently large makes
\eqref{e:AS-simplified-size} imply both
\eqref{e:AS-counting-condition-X} and
\eqref{e:AS-counting-condition-model}.  This proves
\eqref{e:AS-simplified-conclusion}.
\end{proof}

\subsection{Prime comparison on smooth progressions}
\label{ss:common-prime-comparison}

We next establish the Gowers norm estimates needed to simplify the Siegel model
on smooth arithmetic progressions. We begin with the Cram\'er estimate, which
is independent of the dichotomy below and also settles the case in which no
Siegel zero exists.

When a Siegel zero exists, we distinguish between the \emph{non-oscillatory}
and \emph{oscillatory} cases, according to whether the exceptional character is
constant or retains oscillation on the progression. Lemma~\ref{l:powered-cramer}
contains the Cram\'er estimate and handles the non-oscillatory case.
Lemma~\ref{l:twisted-cramer} handles the oscillatory case by first reducing the
Siegel model to the Cram\'er model.
In Proposition~\ref{prop:wbound}, we then combine
Lemmas~\ref{l:powered-cramer} and~\ref{l:twisted-cramer} to obtain a prime
comparison for polynomial configurations.

\begin{lemma}[Gowers norm bounds for the Cram\'er model and the non-oscillatory Siegel model]
\label{l:powered-cramer}
\label{l:powered-absorbed-siegel}\label{l:absorbed-siegel}
Let $s\ge2$ and $d\in\mathbb N$ be fixed, and let $X\ge3$. Set
$$
\widetilde Q\coloneqq
\left\lceil\exp\bigl((\log X)^{c_0}\bigr)\right\rceil,
$$
so $\widetilde{Q}$ is the cutoff from Definition~\ref{d:Siegel} at scale $X$.
There exist $c=c(s,d)\in(0,1/10]$ and
$\eta=\eta(s,d)>0$ with the following property. Let $w\in\mathbb N$
satisfy
$$
2\le w\le(\log\widetilde Q)^c.
$$
Let $Q\in\mathbb N$ and $b\in\mathbb Z$ satisfy
$$
8\prod_{p\le w}p\mid Q,
\qquad P^+(Q)\le w,
\qquad \log Q\le w^2,
\qquad 1\le b\le Q,
\qquad (b,Q)=1.
$$
Suppose that $M\in\mathbb N$ satisfies
$$
X^{1/d}\exp\bigl(-d\sqrt{\log X}\bigr)
\le M\le X^{1/d}.
$$
Then the following hold.
\begin{enumerate}[label=\textup{(\roman*)}]
\item For every interval $J\subseteq[M]$,
\begin{equation}\label{e:powered-cramer}
\left\|
\1_J
\left(
\frac{\phi(Q)}{Q}\Lambda_{\Cramer,\widetilde Q}(Q\cdot+b)-1
\right)
\right\|_{U^s[M]}
\ll_{s,d}w^{-\eta}.
\end{equation}
\item Suppose that a Siegel zero $\beta$, with character
$\chi_{\Siegel}$ and conductor $q_{\Siegel}$, exists at scale $X$, and
that $q_{\Siegel}\mid Q$. Define
$$
\rho_{Q,b,M}
\coloneqq
1-\chi_{\Siegel}(b)(QM+b)^{\beta-1}.
$$
Then
\begin{equation}\label{e:powered-absorbed-siegel}
\left\|
\frac{\phi(Q)}{Q}
\Lambda_{\Siegel,\widetilde Q}(Q\cdot+b)
-\rho_{Q,b,M}
\right\|_{U^s[M]}
\ll_{s,d}\rho_{Q,b,M}w^{-\eta}.
\end{equation}
\end{enumerate}
\end{lemma}

\begin{proof}
We first prove part~\textup{(i)}. We have $\log M\asymp_d\log X$, and
$$
    Q\le\exp(w^2)
    \le\exp\bigl((\log\widetilde Q)^{2c}\bigr)
    \le\exp\bigl((\log M)^{2/5}\bigr),
$$
where the last inequality follows from $c\le1/10$. Since $c_0<1/11$, we
also have
$$
    \widetilde Q+1\le\exp\bigl((\log M)^{1/10}\bigr)
$$
for $X$ sufficiently large. Proposition~\ref{p:saturated-cramer-gowers},
with $z=\widetilde Q$, now proves \eqref{e:powered-cramer}, after
decreasing $\eta$ so that $\eta\le2^{-s}$.

We then prove part~\textup{(ii)}. Since
$q_{\Siegel}\mid Q$, we have
$\chi_{\Siegel}(Qy+b)=\chi_{\Siegel}(b)$, so the exceptional character has
no oscillation in $y$. We then localize the slowly varying factor
$(Qy+b)^{\beta-1}$. Write
$$
\nu(n)\coloneqq
\frac{\phi(Q)}{Q}\Lambda_{\Cramer,\widetilde Q}(Qn+b)
$$
and
$$
\rho(n)\coloneqq
1-\chi_{\Siegel}(b)(Qn+b)^{\beta-1}.
$$
Since $q_{\Siegel}\mid Q$, we have
$\chi_{\Siegel}(Qn+b)=\chi_{\Siegel}(b)$, and hence
$$
\frac{\phi(Q)}{Q}\Lambda_{\Siegel,\widetilde Q}(Qn+b)
=\nu(n)\rho(n).
$$

Fix $B_0>0$, to be chosen sufficiently large in terms of $s$, and set
$$
I_{B_0}\coloneqq
\{n\in\mathbb N:w^{-B_0}M<n\le M\}.
$$

\noindent\textbf{Claim.} Uniformly for $n\in I_{B_0}$,
\begin{equation}\label{e:powered-rho-local}
\rho(n)=
\rho_{Q,b,M}\left(1+O_{s,d,B_0}(w^{-\eta_1})\right)
\end{equation}
for some $\eta_1>0$.

To prove the claim, set $\alpha=1-\beta$, $u=\log(Qn+b)$, and
$U=\log(QM+b)$. For $n\in I_{B_0}$,
\begin{align*}
0\le U-u
&=\log\left(\frac{QM+b}{Qn+b}\right)
\le\log\left(\frac{M+1}{n}\right)
\ll_{B_0}\log w,\\
&\text{and}\qquad u,U\asymp_d\log X.
\end{align*}
If $\chi_{\Siegel}(b)=1$, the function
$F(v)=1-e^{-\alpha v}$ is increasing and concave with $F(0)=0$.
Concavity gives $F(u)\ge (u/U)F(U)$, and hence
\begin{equation}\label{e:powered-rho-positive-error}
0\le\frac{F(U)-F(u)}{F(U)}\le\frac{U-u}{U}
\ll_{d,B_0}\frac{\log w}{\log X}.
\end{equation}
If $\chi_{\Siegel}(b)=-1$, the mean value theorem instead
gives
\begin{equation}\label{e:powered-rho-negative-error}
\frac{|(1+e^{-\alpha u})-(1+e^{-\alpha U})|}
{1+e^{-\alpha U}}
\le \alpha(U-u)
\ll_{B_0}\frac{\log w}{\log\widetilde Q},
\end{equation}
using the defining bound $\alpha\ll1/\log\widetilde Q$ for the exceptional
zero. Since $w\le(\log\widetilde Q)^c$, the bounds on the right of
\eqref{e:powered-rho-positive-error} and
\eqref{e:powered-rho-negative-error} are $O_{s,d,B_0}(w^{-\eta_1})$ for
some $\eta_1>0$. This proves the claim.

Part~\textup{(i)}, applied with $J=I_{B_0}$, gives
$$
\|\1_{I_{B_0}}(\nu-1)\|_{U^s[M]}
\ll_{s,d}w^{-\eta_2}
$$
for some $\eta_2>0$. It also gives
$\|\1_{I_{B_0}}\nu\|_{U^s[M]}\ll_{s,d}1$. On $I_{B_0}$ we have
\begin{align*}
\1_{I_{B_0}}(\nu\rho-\rho_{Q,b,M})
={}&\rho_{Q,b,M}\1_{I_{B_0}}(\nu-1)
 +\1_{I_{B_0}}\nu(\rho-\rho_{Q,b,M}).
\end{align*}
The triangle inequality, followed by \eqref{e:powered-rho-local} and the
cube definition of the Gowers norm and the nonnegativity of $\nu$, therefore
gives, for some $\eta_3>0$,
\begin{align}\begin{split}\label{e:1B0}
&\left\|
\1_{I_{B_0}}(\nu\rho-\rho_{Q,b,M})
\right\|_{U^s[M]}\\
&\qquad\le
\rho_{Q,b,M}\|\1_{I_{B_0}}(\nu-1)\|_{U^s[M]}
+\|\1_{I_{B_0}}\nu(\rho-\rho_{Q,b,M})\|_{U^s[M]}\\
&\qquad\le
\rho_{Q,b,M}\|\1_{I_{B_0}}(\nu-1)\|_{U^s[M]}
+\|\rho-\rho_{Q,b,M}\|_{L^\infty(I_{B_0})}
\|\1_{I_{B_0}}\nu\|_{U^s[M]}\\
&\qquad
\ll_{s,d,B_0}\rho_{Q,b,M}w^{-\eta_3}.
\end{split}
\end{align}

On $[M]\setminus I_{B_0}$, we have
\begin{equation}\label{e:powered-rho-complement}
0\le\rho(n)\le2\rho_{Q,b,M}.
\end{equation}
For $\chi_{\Siegel}(b)=1$, one may replace $2$ by $1$ in
\eqref{e:powered-rho-complement}; indeed, for $n\le M$ the function
$1-(Qn+b)^{\beta-1}$ is at most $\rho_{Q,b,M}$, and for
$\chi_{\Siegel}(b)=-1$ we use $\rho(n)\le2$ and $\rho_{Q,b,M}\ge1$.
Therefore
$$
\|\1_{[M]\setminus I_{B_0}}\nu\rho\|_{U^s[M]}
\le2\rho_{Q,b,M}
\|\1_{[M]\setminus I_{B_0}}\nu\|_{U^s[M]}.
$$
By part~\textup{(i)},
$$
\|\1_{[M]\setminus I_{B_0}}(\nu-1)\|_{U^s[M]}
\ll_{s,d}w^{-\eta_2}.
$$
Also,
$$
\|\1_{[M]\setminus I_{B_0}}\|_{U^s[M]}
\ll_sw^{-B_0(s+1)/2^s}.
$$
It follows that
$$
\|\1_{[M]\setminus I_{B_0}}\nu\rho\|_{U^s[M]}
\ll_{s,d,B_0}
\rho_{Q,b,M}
\left(w^{-\eta_2}+w^{-B_0(s+1)/2^s}\right).
$$
Since
\begin{align*}
\1_{[M]\setminus I_{B_0}}(\nu\rho-\rho_{Q,b,M})
={}&\1_{[M]\setminus I_{B_0}}\nu\rho
-\rho_{Q,b,M}\1_{[M]\setminus I_{B_0}},
\end{align*}
another application of the triangle inequality gives
\begin{align}\label{e:1IB0C}
\left\|
\1_{[M]\setminus I_{B_0}}(\nu\rho-\rho_{Q,b,M})
\right\|_{U^s[M]}
\ll_{s,d,B_0}
\rho_{Q,b,M}
\left(w^{-\eta_2}+w^{-B_0(s+1)/2^s}\right).
\end{align}
Finally,
\begin{align*}
\nu\rho-\rho_{Q,b,M}
={}&\1_{I_{B_0}}(\nu\rho-\rho_{Q,b,M})\\
&+\1_{[M]\setminus I_{B_0}}(\nu\rho-\rho_{Q,b,M}).
\end{align*}
Hence, the triangle inequality,~\eqref{e:1B0} and~\eqref{e:1IB0C} give
\eqref{e:powered-absorbed-siegel} after fixing $B_0=B_0(s)$ and decreasing
$\eta>0$.
\end{proof}

We next treat the oscillatory case.

\begin{lemma}[Gowers norm bound in the oscillatory case]\label{l:twisted-cramer}
Let $s,d\in\mathbb N$ be fixed, and let $N$ be sufficiently large.
Let $\widetilde Q,\beta,\chi_{\Siegel},q_{\Siegel}$ be as in
Definition~\ref{d:Siegel}, and suppose that a Siegel zero exists.  Let
$z,w\in\mathbb N$ satisfy $2\le z\le w\le \widetilde Q$, and let
$Q,b\in\mathbb N$ satisfy
$$
    8\prod_{p\le z}p\mid Q,
    \qquad
    P^+(Q)\le w,
    \qquad
    (b,Q)=1.
$$
Assume moreover that $Qq_{\Siegel}\le \exp\bigl((\log N)^{1/10}\bigr)$.
Let $M$ be an integer satisfying
\begin{equation}\label{eq:twisted-flexible-scale}
    N^{1/d}\exp\bigl(-d\sqrt{\log N}\bigr)
    \le M\le N^{1/d}.
\end{equation}
If $q_{\Siegel}\nmid Q$, then there exists $t=t(s,d)>0$ such that
\begin{align}\label{eq:twisted-needed0}
    \left\|
    \frac{\phi(Q)}{Q}
    \Lambda_{\Cramer,w}(Q\cdot+b)
    \chi_{\Siegel}(Q\cdot+b)(Q\cdot+b)^{\beta-1}
    \right\|_{U^s[M]}
    \ll_{s,d} z^{-t}.
\end{align}
\end{lemma}

\begin{remark}
Lemma~\ref{l:twisted-cramer} is stated with a general progression modulus $Q$,
rather than only with the $W$-trick modulus, because we use it here with $Q=W$
and later in Section~\ref{s:distinct} with an auxiliary modulus $Q_0$.
\end{remark}

\begin{proof}
\emph{Reduction to a uniform cube estimate.}\par
Set
$$
    F(n)\coloneqq
    \frac{\phi(Q)}{Q}
    \Lambda_{\Cramer,w}(Qn+b)
    \chi_{\Siegel}(Qn+b).
$$
For $n\in[M]$, the fundamental theorem of calculus gives
$$
    (Qn+b)^{\beta-1}
    =
    (QM+b)^{\beta-1}
    +
    \int_1^M
    (1-\beta)Q(Qu+b)^{\beta-2}\1_{[u]}(n)\,\mathrm du.
$$
Approximating the integral by Riemann sums and applying the triangle inequality
for the normalized $U^s[M]$-norm (or seminorm when $s=1$) to each sum gives
\begin{align*}
&\left\|F(\cdot)(Q\cdot+b)^{\beta-1}\right\|_{U^s[M]}\\
&\qquad\le
    (QM+b)^{\beta-1}\|F\|_{U^s[M]}
    +
    \int_1^M
    (1-\beta)Q(Qu+b)^{\beta-2}
    \|\1_{[u]}F\|_{U^s[M]}\,\mathrm du.
\end{align*}
The nonnegative scalar coefficients have total mass
\begin{align*}
    (QM+b)^{\beta-1}
    +
    \int_1^M(1-\beta)Q(Qu+b)^{\beta-2}\,\mathrm du
    =
    (Q+b)^{\beta-1}
    \le1.
\end{align*}
It is therefore enough to prove, uniformly for $1\le T\le M$, that
$$
    \|\1_{[T]}F\|_{U^s[M]}
    \ll_s z^{-1/(5\cdot2^s)}.
$$

Let
$$
    \Omega_T
    =
    \left\{
    (x,\mathbf y)\in\mathbb R^{s+1}:
    0<x+\omega\cdot\mathbf y\le T
    \text{ for every }\omega\in\{0,1\}^s
    \right\}.
$$
For $\omega\in\{0,1\}^s$, write
$$
    \psi_\omega(n,\mathbf h)
    \coloneqq Q(n+\omega\cdot\mathbf h)+b.
$$
Since $F$ is real-valued, expanding the Gowers norm and using
$
\|\1_{[M]}\|_{U^s(\mathbb Z)}^{2^s}\asymp_s M^{s+1}
$
shows that it is enough to prove
\begin{align}\label{eq:twisted-cube-sum}
&\left|
\sum_{(n,\mathbf h)\in\Omega_T\cap\mathbb Z^{s+1}}
\prod_{\omega\in\{0,1\}^s}
\frac{\phi(Q)}{Q}
\Lambda_{\Cramer,w}\bigl(\psi_\omega(n,\mathbf h)\bigr)
\chi_{\Siegel}\bigl(\psi_\omega(n,\mathbf h)\bigr)
\right|
\ll_s M^{s+1}z^{-1/5}.
\end{align}

\emph{The Cram\'er-model correlation on residue classes.}\par
By \eqref{eq:twisted-flexible-scale}, for $N$ sufficiently large,
$$
    q_{\Siegel}\le\widetilde Q
    \ll\exp\bigl((\log N)^{c_0}\bigr)
    \le M^{1/100}.
$$
Split the variables into residue classes modulo $q_{\Siegel}$, and write
$$
    \psi_{\omega,a,\mathbf c}(n',\mathbf h')
    \coloneqq
    Q(q_{\Siegel}n'+a+
    \omega\cdot(q_{\Siegel}\mathbf h'+\mathbf c))+b.
$$
Set
\begin{align*}
G_T(a,\mathbf c)
\coloneqq
\sum_{\substack{
(n',\mathbf h')\in q_{\Siegel}^{-1}(\Omega_T-(a,\mathbf c))\\
(n',\mathbf h')\in\mathbb Z^{s+1}}}
\prod_{\omega\in\{0,1\}^s}
\frac{\phi(Q)}{Q}
\Lambda_{\Cramer,w}
\bigl(\psi_{\omega,a,\mathbf c}(n',\mathbf h')\bigr).
\end{align*}
Then the sum in \eqref{eq:twisted-cube-sum} becomes
\begin{equation}\label{eq:chisiegel}
    \sum_{(a,\mathbf c)\in[q_{\Siegel}]^{s+1}}
    \prod_{\omega\in\{0,1\}^s}
    \chi_{\Siegel}\bigl(Q(a+\omega\cdot\mathbf c)+b\bigr)
    G_T(a,\mathbf c).
\end{equation}

For fixed $(a,\mathbf c)$, the region
$$
    \mathcal R_T(a,\mathbf c)
    \coloneqq q_{\Siegel}^{-1}(\Omega_T-(a,\mathbf c))
$$
is convex and is contained in a cube of side length
$O_s(M/q_{\Siegel}+1)$.  Set
$$
    R\coloneqq C_s(M/q_{\Siegel}+1)
$$
with $C_s$ sufficiently large.

The affine forms $\psi_{\omega,a,\mathbf c}$ have linear coefficient vectors
$$
    Qq_{\Siegel}(1,\omega),
    \qquad \omega\in\{0,1\}^s,
$$
which are pairwise linearly independent and have magnitude
$O_s(Qq_{\Siegel})$. By \eqref{eq:twisted-flexible-scale} and the bound
$q_{\Siegel}\le M^{1/100}$ above, we have $\log R\asymp_d\log N$.
The hypothesis on $Qq_{\Siegel}$ therefore gives
$$
    Qq_{\Siegel}
    \le
    \exp\bigl((\log N)^{1/10}\bigr)
    \le
    \exp\bigl((\log R)^{2/5}\bigr)
$$
for $N$ sufficiently large. Since $c_0<1/11$, we also have, for
$N$ sufficiently large,
$$
    w+1
    \le
    \exp\bigl((\log R)^{1/10}\bigr).
$$
Since
$$
    \operatorname{vol}(\mathcal R_T(a,\mathbf c))
    =
    q_{\Siegel}^{-(s+1)}\operatorname{vol}(\Omega_T),
$$
\cite[Proposition~5.2]{TT-JEMS} gives
\begin{equation}\label{eq:Gac2}
G_T(a,\mathbf c)
=
q_{\Siegel}^{-(s+1)}\operatorname{vol}(\Omega_T)
\prod_{p\le w}\beta_p
+
O_{s,d}\!\left(
M^{s+1}\exp\bigl(-(\log N)^{1/11}\bigr)
\right).
\end{equation}

Here the local factors $\beta_p$ in \eqref{eq:Gac2} include the $p$-parts of
the normalizing factor $(\phi(Q)/Q)^{2^s}$. Since $P^+(Q)\le w$, every prime
dividing $Q$ occurs in the product over $p\le w$. Due to the support of
$\chi_{\Siegel}$ in
\eqref{eq:chisiegel}, only tuples $(a,\mathbf c)$ satisfying
\begin{equation}\label{eq:twisted-coprimality}
    \bigl(Q(a+\omega\cdot\mathbf c)+b,q_{\Siegel}\bigr)=1
    \qquad\text{for every }\omega\in\{0,1\}^s
\end{equation}
contribute. For such tuples, the local factors are independent of
$(a,\mathbf c)$. More precisely, for primes $p\le w$,
$$
\beta_p=
\begin{cases}
   1, & p\mid Q,\\[4pt]
   \left(\dfrac{p}{p-1}\right)^{2^s},
      & p\mid q_{\Siegel}\ \text{and}\ p\nmid Q,\\[8pt]
   1+O_s(p^{-2}), & p\nmid Qq_{\Siegel}.
\end{cases}
$$
If $p\mid Q$, then every form is congruent to
$b\not\equiv0\pmod p$. If $p\mid q_{\Siegel}$ and $p\nmid Q$,
\eqref{eq:twisted-coprimality} excludes the zero class for every form. Finally,
if $p\nmid Qq_{\Siegel}$, an invertible affine change of variables
modulo $p$ gives the factor $1+O_s(p^{-2})$.
Thus
$$
    \prod_{p\le w}\beta_p
    =
    Y
    \prod_{\substack{p\le w\\ p\mid q_{\Siegel}\\ p\nmid Q}}
    \left(\frac{p}{p-1}\right)^{2^s},
$$
where
$$
    Y=
    \prod_{\substack{p\le w\\ p\nmid Qq_{\Siegel}}}
    \left(1+O_s(p^{-2})\right)
$$
satisfies $Y\ll_s1$, by the convergence of $\sum_p p^{-2}$.
Consequently, \eqref{eq:Gac2} becomes
\begin{align}\label{eq:Gac3}
    G_T(a,\mathbf c)
=
q_{\Siegel}^{-(s+1)}
\operatorname{vol}(\Omega_T)Y
\prod_{\substack{p\le w\\ p\mid q_{\Siegel}\\ p\nmid Q}}
    \left(\frac{p}{p-1}\right)^{2^s}
+
O_{s,d}\!\left(
M^{s+1}\exp\bigl(-(\log N)^{1/11}\bigr)
\right).
\end{align}

The product over primes $p\le w$ with $p\mid q_{\Siegel}$ and
$p\nmid Q$ will be kept together with the complete character sum.  Set
$$
    \mathcal C(a,\mathbf c)
    \coloneqq
    \prod_{\omega\in\{0,1\}^s}
    \chi_{\Siegel}\bigl(Q(a+\omega\cdot\mathbf c)+b\bigr).
$$
Using \eqref{eq:Gac3}, and noting that only tuples with
$\mathcal C(a,\mathbf c)\neq0$ contribute, the left-hand side of
\eqref{eq:chisiegel} is
\begin{align*}
\sum_{a,\mathbf c\bmod q_{\Siegel}}
\mathcal C(a,\mathbf c)G_T(a,\mathbf c)
&=q_{\Siegel}^{-(s+1)}\operatorname{vol}(\Omega_T)Y
\prod_{\substack{p\le w\\ p\mid q_{\Siegel}\\ p\nmid Q}}
\left(\frac{p}{p-1}\right)^{2^s}
\sum_{a,\mathbf c\bmod q_{\Siegel}}
\mathcal C(a,\mathbf c)\\
&\quad+
O_{s,d}\!\left(
M^{s+1}q_{\Siegel}^{s+1}
\exp\bigl(-(\log N)^{1/11}\bigr)
\right).
\end{align*}
Since $q_{\Siegel}^{s+1}
    \le
    \widetilde Q^{s+1}
    \le
    \exp\bigl(O_s((\log N)^{c_0})\bigr)$
and $c_0<1/11$, the resulting error is
$$
    \ll_{s,d}
    M^{s+1}
    \exp\bigl(-\tfrac{1}{2}(\log N)^{1/11}\bigr),
$$
say, and is negligible compared with $M^{s+1}z^{-1/5}$. Since
$\operatorname{vol}(\Omega_T)\ll_sM^{s+1}$ and $Y\ll_s1$, we have
$\operatorname{vol}(\Omega_T)Y\ll_sM^{s+1}$.
It remains to prove
\begin{align}\label{eq:chisiegelsum}
\prod_{\substack{p\le w\\ p\mid q_{\Siegel}\\ p\nmid Q}}
\left(\frac{p}{p-1}\right)^{2^s}
\left|\sum_{(a,\mathbf c)\in[q_{\Siegel}]^{s+1}}
\mathcal C(a,\mathbf c)
\right|
\ll_s q_{\Siegel}^{s+1}z^{-1/5}.
\end{align}
\emph{The complete character sum.}\par\nopagebreak
Since $\chi_{\Siegel}$ is a primitive non-principal real character,
$q_{\Siegel}$ is the absolute value of a fundamental discriminant. Thus
$$
    q_{\Siegel}=2^a q_0,
    \qquad
    a\in\{0,2,3\},
$$
where $q_0$ is odd and square-free, and we may write
$$
    \chi_{\Siegel}
    =
    \chi_{2^a}
    \prod_{\substack{p\mid q_{\Siegel}\\ p\ \mathrm{odd}}}\chi_p,
$$
where each $\chi_p$ is the primitive quadratic character modulo $p$, and
$\chi_{2^a}$ is the $2$-adic factor. By the Chinese remainder theorem, the
complete character sum in \eqref{eq:chisiegelsum} factors into local complete
sums. For each odd prime $p\mid q_{\Siegel}$, define
$$
A_p
\coloneqq
\frac{1}{p^{s+1}}
\left|
\sum_{(a,\mathbf c)\in(\mathbb Z/p\mathbb Z)^{s+1}}
\prod_{\omega\in\{0,1\}^s}
\chi_p\bigl(Q(a+\omega\cdot\mathbf c)+b\bigr)
\right|.
$$
We claim that
$$
A_p=1\quad\text{if }p\mid Q,
\qquad
A_p\ll_s p^{-1/4}\quad\text{if }p\nmid Q.
$$
The $2$-adic local factor equals $1$. In fact, $8\mid Q$, so the $2$-adic
part of $q_{\Siegel}$ divides $Q$ and the local $2$-adic character is constant
on the forms $Q(a+\omega\cdot\mathbf c)+b$. Its $2^s$-fold product is therefore
$1$.

If $p\mid Q$, then $Q(a+\omega\cdot\mathbf c)+b\equiv b\pmod p$
for every $\omega$. Since $(b,Q)=1$, we have $p\nmid b$, and therefore
$A_p=1$.

If $p\nmid Q$, then the change of variables
$$
    a'=Qa+b,\qquad \mathbf c'=Q\mathbf c \pmod p
$$
is invertible. Hence \cite[Lemma~5.6]{TT-JEMS} gives $ A_p\ll_s p^{-1/4}.$
Moreover, since $\prod_{p\le z}p\mid Q$, every such prime satisfies $p>z$.
The sum
$$
\sum_{(a,\mathbf c)\in[q_{\Siegel}]^{s+1}}
\mathcal C(a,\mathbf c)
$$
factors over the prime divisors of $q_{\Siegel}$ by the Chinese remainder
theorem. After dividing this sum by $q_{\Siegel}^{s+1}$ and multiplying by the
prefactor
$$
\prod_{\substack{p\le w\\ p\mid q_{\Siegel}\\ p\nmid Q}}
\left(\frac{p}{p-1}\right)^{2^s},
$$
the normalized local factor attached to an odd prime $p\mid q_{\Siegel}$ in
\eqref{eq:chisiegelsum} is
$$
\begin{cases}
1, & p\mid Q,\\[4pt]
\left(\dfrac{p}{p-1}\right)^{2^s} A_p,
    & p\nmid Q\ \text{and}\ p\le w,\\[10pt]
A_p, & p\nmid Q\ \text{and}\ p>w.
\end{cases}
$$
Since $(p/(p-1))^{2^s}\le 2^{2^s}$, every local factor with $p\nmid Q$
is $O_s(p^{-1/4})$.

Since $q_{\Siegel}\nmid Q$, there exists an odd prime
$p_0\mid q_{\Siegel}$ with $p_0\nmid Q$. Hence $p_0>z$. If $z$ is below a
fixed threshold depending only on $s$, the desired bound is absorbed into the
implicit constant. Otherwise, every local factor with $p\nmid Q$ is at most
$1$, while at the prime $p_0$ we keep the saving $O_s(p_0^{-1/4})$.
Therefore
$$
\prod_{\substack{p\le w\\ p\mid q_{\Siegel}\\ p\nmid Q}}
\left(\frac{p}{p-1}\right)^{2^s}
\frac{1}{q_{\Siegel}^{s+1}}
\left|
\sum_{(a,\mathbf c)\in[q_{\Siegel}]^{s+1}}
\mathcal C(a,\mathbf c)
\right|
\ll_s p_0^{-1/4}
\ll_s z^{-1/4}.
$$
This proves \eqref{eq:chisiegelsum}, with $z^{-1/5}$ in place of
$z^{-1/4}$, and hence proves \eqref{eq:twisted-cube-sum} uniformly for
$1\le T\le M$. Dividing by
$\|\1_{[M]}\|_{U^s(\mathbb Z)}^{2^s}\asymp_s M^{s+1}$ and taking
$2^s$-th roots gives the required uniform bound for
$\|\1_{[T]}F\|_{U^s[M]}$. The triangle inequality and the coefficient
mass calculation above now give \eqref{eq:twisted-needed0}, with
$t=1/(5\cdot2^s)$.
\end{proof}

The following proposition combines the Gowers norm estimates in
Lemmas~\ref{l:powered-cramer} and~\ref{l:twisted-cramer} with the generalized
von Neumann estimate in Proposition~\ref{l:fixed-scale-gvn} to establish a
prime comparison for polynomial configurations with prime parameter in the
progression $Q\mathbb Z+b$.

\begin{proposition}[Prime-weighted counting operator comparison on a smooth progression]
\label{prop:wbound}
Let $d,k\in\mathbb N$. There are constants $c'=c'(d,k)>0$ and
$\eta'=\eta'(d,k)>0$, and a threshold $w_0=w_0(d,k)\ge2$, with the following
property. Let $X\ge3$, set
$$
\widetilde Q\coloneqq
\left\lceil\exp\bigl((\log X)^{c_0}\bigr)\right\rceil,
$$
and let $w\in\mathbb N$ satisfy
$$
w_0\le w\le(\log\widetilde Q)^{c'}.
$$
Suppose that $Q\in\mathbb N$ and $b\in\mathbb Z$ satisfy
\begin{equation}\label{e:powered-modulus-hypotheses}
8\prod_{p\le w}p\mid Q,
\qquad P^+(Q)\le w,
\qquad \log Q\le w^2,
\qquad 1\le b\le Q,
\qquad (b,Q)=1,
\end{equation}
and that $M\in\mathbb N$ satisfies
\begin{equation}\label{e:powered-scale-hypotheses}
X^{1/d}\exp\bigl(-d\sqrt{\log X}\bigr)
\le M\le X^{1/d}.
\end{equation}
Let $R_1,\ldots,R_k\in\mathbb Z[y]$ have degrees at most $d$, and let
$f_0,\ldots,f_k\colon\mathbb Z\to\mathbb C$ be $1$-bounded functions
supported on $[X]$.

If a Siegel zero exists at scale $X$, define
\begin{equation}\label{e:powered-rho}
\rho_{Q,b,M}
\coloneqq
1-\chi_{\Siegel}(b)(QM+b)^{\beta-1}
\1_{q_{\Siegel}\mid Q};
\end{equation}
if no Siegel zero exists, set $\rho_{Q,b,M}=1$. Then
\begin{align}
&\left|
\mathbb E_{x\le X}\mathbb E_{y\le M}
\frac{\phi(Q)}{Q}\Lambda(Qy+b)
f_0(x)\prod_{i=1}^{k}f_i(x+R_i(y))
\right.\notag\\
&\qquad\left.
-\rho_{Q,b,M}
\mathbb E_{x\le X}\mathbb E_{y\le M}
f_0(x)\prod_{i=1}^{k}f_i(x+R_i(y))
\right|
\ll_{d,k}
\rho_{Q,b,M}w^{-\eta'}.
\label{e:powered-polynomial-comparison}
\end{align}
\end{proposition}

\begin{proof}
\emph{Gowers norm reduction.}\par
Apply Proposition~\ref{l:fixed-scale-gvn} to the polynomials
$R_1,\ldots,R_k$. There is an integer $s\ll_{d,k}1$, which we may take to
be at least $2$ and which is independent of their coefficients, such that
\begin{align}
&\left|
\mathbb E_{x\le X}\mathbb E_{y\le M}
g(y)f_0(x)\prod_{i=1}^{k}f_i(x+R_i(y))
\right|
\ll_{d,k}\|g\|_{U^s[M]}.
\label{e:powered-gvn}
\end{align}
Define
\begin{align*}
\nu_{\Lambda}(y)&\coloneqq\frac{\phi(Q)}{Q}\Lambda(Qy+b),\\
\nu_{\Siegel}(y)&\coloneqq
\frac{\phi(Q)}{Q}\Lambda_{\Siegel,\widetilde Q}(Qy+b).
\end{align*}
By \eqref{e:powered-gvn}, it suffices to prove
\begin{equation}\label{e:powered-gowers-target}
\|\nu_{\Lambda}-\rho_{Q,b,M}\|_{U^s[M]}
\ll_{d,k}\rho_{Q,b,M}w^{-\eta'}.
\end{equation}
We use the decomposition
\begin{equation}\label{e:powered-gowers-decomposition}
\nu_{\Lambda}-\rho_{Q,b,M}
=(\nu_{\Lambda}-\nu_{\Siegel})
 +(\nu_{\Siegel}-\rho_{Q,b,M}).
\end{equation}
Choose $c'$ sufficiently small in terms of $d$ and $k$ to satisfy all the
restrictions below, with $c'\le1/10$ and with $c'$ at most the constant $c$
from Lemma~\ref{l:powered-cramer}; increase $w_0$ as needed.

\emph{Comparison with the Siegel model.}\par
Equations~\eqref{e:powered-modulus-hypotheses} and
\eqref{e:powered-scale-hypotheses} give
$\log(QM)\asymp_d\log X$. Let $c_s>0$ be supplied by
Lemma~\ref{l:prime-siegel-gowers}. Since $c_0<1/11$, after increasing $w_0$ we
have
\begin{equation}\label{e:powered-MTW-range}
\exp\bigl((\log\log(QM+b))^{1/c_s}\bigr)
\le\widetilde Q
\le\exp\bigl((\log(QM+b))^{1/10}\bigr).
\end{equation}
Lemma~\ref{l:prime-siegel-gowers}, applied with model scale $X$ and Gowers-norm
scale $Y=QM+b$, bounds $\Lambda-\Lambda_{\Siegel,\widetilde Q}$ on the
interval $[QM+b]$. To pass to $[b,QM+b]$, write
$$
\1_{[b,QM+b]}=\1_{[QM+b]}-\1_{[b-1]}.
$$
Equation~\eqref{e:prime-siegel-gowers}, the pointwise bound, valid for
$n\le QM+b$,
$$
|\Lambda(n)-\Lambda_{\Siegel,\widetilde Q}(n)|
\ll\log(QM+\widetilde Q+2),
$$
and the triangle inequality for the unnormalized Gowers norm therefore give
\begin{equation}\label{e:powered-MTW-interval}
\|\Lambda-\Lambda_{\Siegel,\widetilde Q}\|_{U^s[b,QM+b]}
\ll_s
\exp\bigl(-(\log\widetilde Q)^{\eta_0}\bigr)
+\log(QM+\widetilde Q+2)
\left(\frac{b}{QM}\right)^{(s+1)/2^s},
\end{equation}
for some $\eta_0=\eta_0(s)\in(0,1]$, where we used
$\|\1_I\|_{U^s(\mathbb Z)}\asymp_s|I|^{(s+1)/2^s}$.

We apply Lemma~\ref{le:dilate}. After increasing $w_0$, the hypotheses
$\log Q\le w^2$, $w\le(\log\widetilde Q)^{c'}$, and
\eqref{e:powered-scale-hypotheses} give
\begin{equation}\label{e:powered-Q-over-M}
\frac{Q}{M}\le X^{-1/(2d)}.
\end{equation}
Consequently, the term involving $b/(QM)$ in
\eqref{e:powered-MTW-interval}, including the
$Q^{(s+1)/2^s}$ factor from Lemma~\ref{le:dilate}, is
\begin{equation}\label{e:powered-dilated-endpoint-error}
Q^{(s+1)/2^s}\log(QM+\widetilde Q+2)
\left(\frac{b}{QM}\right)^{(s+1)/2^s}
\ll_{s,d}X^{-(s+1)/(2^{s+2}d)}.
\end{equation}
The same dilation factor is absorbed by the first term in
\eqref{e:powered-MTW-interval}: indeed,
$w^2\le(\log\widetilde Q)^{2c'}$, and we choose
$2c'<\eta_0$. Thus
\begin{equation}\label{e:powered-dilated-main-error}
Q^{(s+1)/2^s}
\exp\bigl(- (\log\widetilde Q)^{\eta_0}\bigr)
\le\exp\left(-\tfrac{1}{2}(\log\widetilde Q)^{\eta_0}\right).
\end{equation}
Equations~\eqref{e:powered-MTW-interval},
\eqref{e:powered-dilated-endpoint-error}, and
\eqref{e:powered-dilated-main-error} yield
\begin{equation}\label{e:powered-prime-siegel}
\|\nu_{\Lambda}-\nu_{\Siegel}\|_{U^s[M]}
\ll_s\exp\left(-\tfrac{1}{2}(\log\widetilde Q)^{\eta_0}\right).
\end{equation}

We next establish the lower bound
\begin{equation}\label{e:powered-rho-lower}
\rho_{Q,b,M}\gg\exp(-w^2),
\end{equation}
which is needed to express \eqref{e:powered-prime-siegel} as a relative
estimate. If $\rho_{Q,b,M}=1$, then \eqref{e:powered-rho-lower} holds.
Otherwise
$q_{\Siegel}\mid Q$. If $\chi_{\Siegel}(b)=-1$, then
$\rho_{Q,b,M}\ge1$, while if $\chi_{\Siegel}(b)=1$, then
the standard effective exceptional zero bound
\cite[Theorem~5.28(2)]{iwaniec-kowalski}, applied with $\varepsilon=1$, gives
$$
\rho_{Q,b,M}\gg1-\beta\gg q_{\Siegel}^{-1}
\ge Q^{-1}\ge\exp(-w^2).
$$
Equations~\eqref{e:powered-prime-siegel} and
\eqref{e:powered-rho-lower}, together with $2c'<\eta_0$, give
\begin{equation}\label{e:powered-prime-siegel-relative}
\|\nu_{\Lambda}-\nu_{\Siegel}\|_{U^s[M]}
\ll_{d,k}\rho_{Q,b,M}w^{-\eta_1}
\end{equation}
for some $\eta_1>0$.

\emph{Simplification of the Siegel model.}\par
It remains to prove that, for some $\eta_2>0$,
\begin{equation}\label{e:powered-siegel-simplified}
\|\nu_{\Siegel}-\rho_{Q,b,M}\|_{U^s[M]}
\ll_{d,k}\rho_{Q,b,M}w^{-\eta_2}.
\end{equation}
If no Siegel zero exists, then $\rho_{Q,b,M}=1$, and
Lemma~\ref{l:powered-cramer}\textup{(i)}, with $J=[M]$, proves
\eqref{e:powered-siegel-simplified}. Suppose next that a Siegel zero exists and
$q_{\Siegel}\nmid Q$. Again $\rho_{Q,b,M}=1$. We apply
Lemma~\ref{l:twisted-cramer} with its small prime parameter $z=w$, Cram\'er
cutoff $\widetilde Q$, and progression modulus $Q$. Its scale condition is
\eqref{e:powered-scale-hypotheses}, and
$$
Qq_{\Siegel}\le\exp(w^2)\widetilde Q
\le\exp\bigl((\log X)^{1/10}\bigr)
$$
by the choice of $c'$ and the inequality $c_0<1/11$.
Together with Lemma~\ref{l:powered-cramer}\textup{(i)}, this gives the
bound in \eqref{e:powered-siegel-simplified} in the oscillatory case. Finally, if
$q_{\Siegel}\mid Q$, then
Lemma~\ref{l:powered-cramer}\textup{(ii)} proves
\eqref{e:powered-siegel-simplified} directly.

Set $\eta'=\min(\eta_1,\eta_2)$. The triangle inequality,
\eqref{e:powered-gowers-decomposition},
\eqref{e:powered-prime-siegel-relative}, and
\eqref{e:powered-siegel-simplified} give
\eqref{e:powered-gowers-target}. Equation~\eqref{e:powered-gvn} now yields
\eqref{e:powered-polynomial-comparison}.
\end{proof}

\subsection{Proofs of the theorems}

\subsubsection{Multiples of a fixed polynomial}
\label{ss:proof-fixed-polynomial}

\begin{proof}[Proof of
Theorem~\ref{thm:fixed-polynomial-shifted-prime}]
Let $c'$ and $\eta'$ be supplied by Proposition~\ref{prop:wbound} for $d$
and $k-1$, and fix
$$
0<\gamma<\frac{1}{4}\eta'.
$$
Let $c_\gamma,C_\gamma$ be the constants in the final assertion of
Proposition~\ref{p:AS-uniform-counting}. Choose a fixed $\theta>0$ so small
that
\begin{equation}\label{e:fixed-polynomial-theta}
\theta<\frac{1}{3},
\qquad
\theta C_\gamma<\frac{1}{2},
\qquad
\theta<\frac{1}{2}c_0c',
\end{equation}
and set
\begin{equation}\label{e:fixed-polynomial-w-choice}
w=\left\lfloor(\log N)^\theta\right\rfloor.
\end{equation}
Suppose, for a contradiction, that $A$ has density
$\delta=|A|/N$ larger than the relevant bound in the theorem, with its
positive exponent $c$ chosen sufficiently small below. We may suppose that
$\delta\le1/2$, since otherwise we may replace $A$ by a subset of cardinality
$\lfloor N/2\rfloor$ and redefine $\delta$. Partition $[N]$ into
residue classes modulo $W^r$. For $1\le b\le W^r$, write
$$
I_b\coloneqq\{u\in\mathbb N:W^r(u-1)+b\in[N]\},
\qquad
A_b\coloneqq\{u\in I_b:W^r(u-1)+b\in A\}.
$$
Since
$$
\sum_{b=1}^{W^r}|A_b|=|A|
\qquad\text{and}\qquad
\sum_{b=1}^{W^r}|I_b|=N,
$$
there is some $b_0$ for which $|A_{b_0}|\ge\delta|I_{b_0}|$. Set
$X=|I_{b_0}|$ and $F=A_{b_0}$. Thus
\begin{equation}\label{e:fixed-polynomial-fibre-density}
F\subseteq[X],
\qquad |F|\ge\delta X,
\qquad X\asymp\frac{N}{W^r}.
\end{equation}

We have $W=\exp(O_{P,d}(w))=N^{o(1)}$. In particular,
\begin{equation}\label{e:fixed-polynomial-size-check}
X\ge\exp(w^{C_\gamma})
\end{equation}
for $N$ sufficiently large, by \eqref{e:fixed-polynomial-theta}. Moreover,
$$
\log w=\theta\log\log N+O(1),
\qquad
\log\log w=\log\log\log N+O_\theta(1).
$$
It follows that, after choosing the exponent $c$ in the theorem sufficiently
small in terms of $c_\gamma$, the density of $F$ satisfies
\eqref{e:AS-simplified-density}.

Applying Proposition~\ref{p:AS-uniform-counting} to $\1_F$, we now obtain
$$
2\int_{1/2}^1
\mathbb E_{x\le X}\mathbb E_{y\le M_z}
\prod_{i=0}^{k-1}
\1_F(x+a_iP_W(y))\,\mathrm dz
>w^{-\gamma}.
$$
All summands are nonnegative, so there is some $z\in[1/2,1]$ such that,
setting $M=M_z$,
\begin{equation}\label{e:fixed-polynomial-unweighted-count}
\mathbb E_{x\le X}\mathbb E_{y\le M}
\prod_{i=0}^{k-1}\1_F(x+a_iP_W(y))
>w^{-\gamma}.
\end{equation}
The definition of $M_z$ and \eqref{e:fixed-polynomial-fibre-density} give
\begin{equation}\label{e:fixed-polynomial-M-size}
M\asymp_P
X^{1/d}W^{-(d-r)/d}\asymp_P\frac{N^{1/d}}{W}.
\end{equation}
Consequently, for $N$ sufficiently large,
$$
X^{1/d}\exp\bigl(-d\sqrt{\log X}\bigr)
\le M\le X^{1/d}.
$$
The lower bound follows from
$\log W=O(w)=o(\sqrt{\log X})$. For the upper bound, if $d=r$, then
$B=|b_d|$; if $d>r$, then
$|b_d|W^{d-r}\ge W\ge B$. In either case, the definition of $M_z$ gives
$M^d\le X$. This verifies \eqref{e:powered-scale-hypotheses}.

We apply Proposition~\ref{prop:wbound} with $Q=W$, $b=1$,
$R_i=a_iP_W$ for $1\le i\le k-1$, and $f_i=\1_F$ for
$0\le i\le k-1$. Its modulus hypotheses
\eqref{e:powered-modulus-hypotheses} follow from
\eqref{e:AS-W-properties}: for sufficiently large $w$, every prime divisor
of $B$ is at most $w$ and
$$
\log W\ll_{P,d}w\le w^2.
$$
Its upper bound on $w$ follows from \eqref{e:fixed-polynomial-theta}, since
$\log\widetilde Q\asymp(\log X)^{c_0}$. By
\eqref{e:fixed-polynomial-unweighted-count} and
\eqref{e:powered-polynomial-comparison}, for sufficiently large $w$ we have
\begin{align}\begin{split}\label{e:1F}
&\frac{\phi(W)}{W}
\mathbb E_{x\le X}\mathbb E_{y\le M}
\Lambda(Wy+1)
\prod_{i=0}^{k-1}\1_F(x+a_iP_W(y))\\
&\qquad\ge
\rho_{W,1,M}
\mathbb E_{x\le X}\mathbb E_{y\le M}
\prod_{i=0}^{k-1}\1_F(x+a_iP_W(y))
-O_{d,k}\bigl(\rho_{W,1,M}w^{-\eta'}\bigr)\\
&\qquad>\rho_{W,1,M}w^{-2\gamma}.
\end{split}
\end{align}
If there is no Siegel zero or $q_{\Siegel}\nmid W$, then
$\rho_{W,1,M}=1$. Otherwise $q_{\Siegel}\mid W$. Since
$\chi_{\Siegel}(1)=1$, the effective exceptional zero bound
\cite[Theorem~5.28(2)]{iwaniec-kowalski}, applied with $\varepsilon=1$, gives
$$
\rho_{W,1,M}=1-(WM+1)^{\beta-1}
\gg1-\beta\gg q_{\Siegel}^{-1}\ge W^{-1}.
$$
Thus in every case
$\rho_{W,1,M}\gg W^{-1}=\exp(-O_{P,d}(w))$. Since $W/\phi(W)\ge1$, the
estimate~\eqref{e:1F} gives
\begin{align}
&\mathbb E_{x\le X}\mathbb E_{y\le M}
\Lambda(Wy+1)
\prod_{i=0}^{k-1}\1_F(x+a_iP_W(y))\notag\\
&\qquad>\rho_{W,1,M}w^{-2\gamma}
\gg\exp(-O_{P,d}(w))w^{-2\gamma}=N^{-o(1)}.
\label{e:fixed-polynomial-prime-power-count}
\end{align}

Since
$$
\sum_{\substack{n\le T\\n=p^j,\ j\ge2}}\Lambda(n)
\ll T^{1/2}\log(T+2),
$$
the contribution of higher prime powers to the
left-hand side of \eqref{e:fixed-polynomial-prime-power-count} is
$\ll_P WN^{-1/(2d)}\log N=N^{-1/(2d)+o(1)}$, which is negligible compared
with \eqref{e:fixed-polynomial-prime-power-count}. Hence some pair $(x,y)$ counted
on the left-hand side of \eqref{e:fixed-polynomial-prime-power-count} has
$$
p=Wy+1\in\mathbb P.
$$

Recall that $F=A_{b_0}$ and $P_W(y)=W^{-r}P(Wy)$. For the pair $(x,y)$ found above and for every $0\le i\le k-1$, we have the exact identity
\begin{align}
W^r\bigl(x+a_iP_W(y)-1\bigr)+b_0
&=W^r(x-1)+b_0+a_iP(Wy)\notag\\
&=W^r(x-1)+b_0+a_iP(p-1).
\label{e:fixed-polynomial-lift}
\end{align}
The left-hand side belongs to $A$ for every $i$, so
\eqref{e:fixed-polynomial-lift} is the forbidden configuration. Finally,
$W\to\infty$ with $N$, while $P$ has only finitely many positive integer
roots.  Thus, for $N$ sufficiently large, $W$ exceeds every positive integer
root of $P$; since $y\ge1$, this gives $P(Wy)=P(p-1)\ne0$. This
contradiction proves the theorem.
\end{proof}

\subsubsection{Linear progressions}

\begin{proof}[Proof of Corollary~\ref{thm_progression}]
Corollary~\ref{thm_progression} follows from
Theorem~\ref{thm:fixed-polynomial-shifted-prime} by taking $P(y)=y$ and
$a_i=i$ for $1\le i\le k-1$.
\end{proof}

\section{Distinct degree progressions} \label{s:distinct}

In this section, we prove Theorem~\ref{thm_poly}. The argument has four parts.
First, we show that if a set $A$ contains no configuration of the required
form, then there is a large prime-weighted discrepancy between the counting
operators involving $\1_A$ and its model $\delta\1_{[N]}$
(Lemma~\ref{lem:full-set-discrepancy}). Second, a chain
of Gowers norm approximations converts this into a large unweighted discrepancy
(Proposition~\ref{lem:chaining}). Third, we combine this discrepancy with the
weak regularity lemma to obtain a density increment on an arithmetic
progression (Lemma~\ref{l:case3}). Fourth, we iterate the density increment to
prove Theorem~\ref{thm_poly}.

We begin by showing the existence of the prime-weighted discrepancy. Recall the definition of the counting operator $T_{N,M,w,Q_0}$ from~\eqref{e:Q0avgmulti}.

\begin{lemma}[Prime-weighted discrepancy]\label{lem:full-set-discrepancy}
Let $N$ be sufficiently large, let $d,k\in\mathbb N$ with $k\ge2$, and let
$C\ge1$. Suppose
$\exp(- (\log N)^{1/100})<\delta\le1$, and let
$q\in\mathbb N$ satisfy $q\le w_*(N)$. Set
$M=\lfloor(N/q^{d-1})^{1/d}\rfloor$ and $P_0=0$, and let
$P_1,\ldots,P_{k-1}$ have $(C,q)$-coefficients and distinct
increasing degrees, with largest degree $d$. If $A\subseteq[N]$ with
$|A|=\delta N$
contains no configuration
\begin{align*}
&x,x+P_1(y),\ldots,x+P_{k-1}(y)\in A,
 \qquad qy+1\in\mathbb P,\\
&0,P_1(y),\ldots,P_{k-1}(y) \textup{ are pairwise distinct},
\end{align*}
then there is a constant $c_1=c_1(C,d,k)\in(0,1]$ such that
\begin{equation}\label{e:full-set-discrepancy}
\left|T_{N,M,\Lambda,q}(\1_A,\ldots,\1_A)
-T_{N,M,\Lambda,q}(\delta\1_{[N]},\ldots,\delta\1_{[N]})\right|
\ge c_1\delta^k\mathbb E_{y\le M}\Lambda(qy+1).
\end{equation}
\end{lemma}

\begin{proof}
Choose a sufficiently small constant $\kappa=\kappa(C,d,k)>0$ and set
$M_\kappa\coloneqq\lfloor\kappa M\rfloor$. The
$(C,q)$-coefficient condition and the definition of $M$ imply, uniformly for
$1\le y\le M_\kappa$, that
$$
 \max_{1\le i\le k-1}|P_i(y)|\le N/10.
$$
The intersection
$$
 [N]\cap([N]-P_1(y))\cap\cdots\cap([N]-P_{k-1}(y))
$$
therefore has cardinality at least $N/2$.

\noindent\textbf{Claim.} We have
\begin{equation}\label{e:initial-prime-mass}
 \sum_{y\le M_\kappa}\Lambda(qy+1)
 \gg_{C,d,k}\kappa\sum_{y\le M}\Lambda(qy+1).
\end{equation}

For $T\in\{M,M_\kappa\}$ and $N$ sufficiently large, the definition of $M$
and the bound on $q$ give
$$
 N^{1/(3d)}\le qT\le N^{3d}.
$$
Since $q\le w_*(N)$, Lemma~\ref{l:landau-page-compatible}, applied with
$A=3d$, gives, for some $c>0$,
\begin{align}\label{e:landau-page-two-scales}
\mathbb E_{y\le T}\Lambda(qy+1)
&=
\frac{q}{\phi(q)}
\left(
1-\frac{(qT+1)^{\beta-1}}{\beta}
\1_{q_{\Siegel}\mid q}
\right)
+O\left(\exp\bigl(-c\sqrt{\log N}\bigr)\right),
\end{align}
where the exceptional term is omitted if no Siegel zero exists. If there is no
Siegel zero, or if $q_{\Siegel}\nmid q$, then
\eqref{e:landau-page-two-scales} produces the claim immediately. We may
therefore suppose that $q_{\Siegel}\mid q$, and set
$$
 \rho(T)\coloneqq
 1-\frac{(qT+1)^{\beta-1}}{\beta}.
$$
The elementary comparison
$$
 1-\frac{X^{\beta-1}}{\beta}
 =\left(1+O\left(\frac{1}{\log X}\right)\right)
  \bigl(1-X^{\beta-1}\bigr)
$$
therefore applies with $X=qT+1$. The effective exceptional zero bound gives
$$
 1-\beta\gg q_{\Siegel}^{-1}\ge q^{-1},
$$
and hence $\rho(T)\gg\exp(-O((\log N)^{c_0}))$. The error term in
\eqref{e:landau-page-two-scales} is therefore $o(\rho(T))$. Finally,
$M_\kappa\ge\kappa M/2$ for sufficiently large $N$, so concavity of
$u\mapsto1-e^{-(1-\beta)u}$ gives
$$
 1-(qM_\kappa+1)^{\beta-1}
 \gg_\kappa 1-(qM+1)^{\beta-1}.
$$
Thus $\rho(M_\kappa)\gg_{\kappa,d}\rho(M)$. Equation
\eqref{e:landau-page-two-scales} now gives us \eqref{e:initial-prime-mass}.

Consequently, the full set counting operator satisfies
\begin{align}\label{e:lowerncount}
 T_{N,M,\Lambda,q}(\delta\1_{[N]},\ldots,
      \delta\1_{[N]})
 &\ge \frac{\delta^k}{2M}
      \sum_{y\le M_\kappa}\Lambda(qy+1)\\
 &\gg_{C,d,k}\delta^k\mathbb E_{y\le M}\Lambda(qy+1)\nonumber.
\end{align}
By the assumption on $A$, the terms for which $qy+1$ is prime and
$0,P_1(y),\ldots,P_{k-1}(y)$ are pairwise distinct make no contribution to
$T_{N,M,\Lambda,q}(\1_A,\ldots,\1_A)$. Recall that $P_0=0$.
For every $0\le i<j\le k-1$, the polynomial $P_j-P_i$ is nonzero
and has degree at most $d$. Hence there are at most $d\binom{k}{2}$ values of $y$ for which $0,P_1(y),\ldots,P_{k-1}(y)$ are not pairwise distinct. The total
contribution from those values for which $qy+1$ is prime is
$$
 \ll_{d,k}\frac{\log(qM+1)}{M}=N^{-1/d+o(1)}.
$$
The other possible contribution comes from the terms for which
$$
    qy+1=p^a,\qquad a\ge2.
$$
Since all the other factors in the counting form are $1$-bounded and
nonnegative, these terms contribute at most
\begin{align}\label{e:lowerpowerprimesacount}
0\le T_{N,M,\Lambda,q}(\1_A,\ldots,\1_A)
&\le \frac{1}{M}
\sum_{\substack{y\le M\\ qy+1=p^a,\ a\ge2}}
\Lambda(qy+1)
+O_{d,k}\left(\frac{\log(qM+1)}{M}\right)\notag\\
&\ll \frac{(qM)^{1/2}(\log(qM))^2}{M}
\ll N^{-1/(3d)},
\end{align}
where $M=N^{1/d+o(1)}$ and $q=N^{o(1)}$ were used in the final estimate.
The right-hand side of \eqref{e:lowerpowerprimesacount} is negligible compared
with \eqref{e:lowerncount} under the stated lower bound for $\delta$.
Subtracting \eqref{e:lowerpowerprimesacount} from
\eqref{e:lowerncount} proves \eqref{e:full-set-discrepancy}.
\end{proof}

For the remainder of this section, fix $d,k\in\mathbb N$ with $k\ge2$ and
$C\ge1$, and fix
one constant $c_1=c_1(C,d,k)$ furnished by
Lemma~\ref{lem:full-set-discrepancy}. Set
$$
    \mathscr C\coloneqq 2+dC^2.
$$
If $T\ge1$, $1\le y\le T$, and $P$ is a polynomial of degree $e\le d$ with
$(C,q)$-coefficients, then
\begin{equation}\label{e:uniform-support-bound}
    |P(y)|
    \le dC^2q^{e-1}T^e
    \le (\mathscr C-1)q^{d-1}T^d.
\end{equation}

Let $s=s(d,k)$ be furnished by Proposition~\ref{l:fixed-scale-gvn}, with
$3\mathscr C$ in place of its parameter $C$. Let
$t=t(s,d)>0$ be furnished by Lemma~\ref{l:twisted-cramer}, and set
$$
    \eta_0\coloneqq\min\{2^{-s},t\}.
$$
Let $c_s>0$ be supplied by Lemma~\ref{l:prime-siegel-gowers}. Decreasing it if
necessary, assume that $c_s\le1$, and set
$\alpha_{d,k}\coloneqq c_0c_s/2$. Choose $c'=c'(d,k)>0$ sufficiently small
that
\begin{equation}\label{e:constantdesign1}
    0<10c'<
    \min\left\{
        \frac{\alpha_{d,k}}{100},
        \frac{1}{10}
    \right\}.
\end{equation}
Choose $\mathbf C_{d,k}$ and $\mathbf K_{d,k}\ge2$, depending only on
$C,d,k$, sufficiently large that
\begin{align}\label{e:constantdesign2}
    \frac{200}{\mathbf C_{d,k}}<c',
    \qquad
    \mathbf C_{d,k}\eta_0\ge 100k,
    \qquad
    \mathbf K_{d,k}^{-\eta_0/10}
    \le 2^{-4k-50}c_1.
\end{align}
We may enlarge these constants later.

The next proposition uses a chain of Gowers norm approximations to convert a
prime-weighted discrepancy into an unweighted discrepancy at a comparable
shift scale.

\begin{proposition}[Unweighted discrepancy]\label{lem:chaining}
Let $N$ be sufficiently large, and suppose that
$(\log N)^{-1/\mathbf C_{d,k}^{2}}<\delta\le1.$
Let $\sigma\ge2$ be an integer satisfying
\begin{equation}\label{e:chaining-cutoff-range}
    \left\lceil
    \mathbf K_{d,k}\delta^{-\mathbf C_{d,k}}
    \right\rceil
    \le\sigma\le(\log N)^{c'}.
\end{equation}
Let $q\in\mathbb N$ satisfy
\begin{equation}\label{e:chaining-saturated-modulus}
    q\le\exp\bigl((\log N)^{2c'}\bigr),
    \qquad
    P^+(q)\le\sigma,
    \qquad
    \mathcal L(\sigma)\mid q.
\end{equation}
Set
$$
    M\coloneqq
    \left\lfloor
    \left(\frac{N}{q^{d-1}}\right)^{1/d}
    \right\rfloor,
    \qquad
    M_{\min}\coloneqq
    \left\lceil M\exp(-\sqrt{\log M})\right\rceil.
$$
Let $P_1,\ldots,P_{k-1}\in\mathbb Z[y]$ be polynomials with
$(C,q)$-coefficients and
$$
    \deg(P_1)<\cdots<\deg(P_{k-1})=d.
$$
Suppose that $A\subseteq[N]$ with $|A|=\delta N$, and
\begin{align}\label{e:step3goal}
&\left|
T_{N,M,\Lambda,q}
(\1_A,\ldots,\1_A)
-
T_{N,M,\Lambda,q}
\bigl(
\delta\1_{[N]},\ldots,
\delta\1_{[N]}
\bigr)
\right|\ge
c_1\delta^k
\mathbb E_{y\le M}\Lambda(qy+1).
\end{align}
Then there exists an integer $M_0$ with
$M_{\min}\le M_0\le M$ such that
\begin{align}\label{e:step3goal2}
&\left|
    T_{N,M_0,1,q}
(\1_A,\ldots,\1_A)
-
    T_{N,M_0,1,q}
\bigl(
\delta\1_{[N]},\ldots,
\delta\1_{[N]}
\bigr)
\right|\ge
    \frac{c_1\delta^k}{4}.
\end{align}
\end{proposition}

\begin{proof}
Recall that $\widetilde Q=w_*(N)$.
We use the following chain of approximations:
$$
    \Lambda
    \longrightarrow \Lambda_{\Siegel,\widetilde Q}
    \longrightarrow \Lambda_{\Cramer,\widetilde Q}
    \longrightarrow \Lambda_{\Cramer,\sigma}.
$$

\emph{Approximating the primes by the Siegel model.}\par
If $f_0,\ldots,f_{k-1}$ are $1$-bounded functions
supported on $[-\mathscr C N,\mathscr C N]$, then
we apply Proposition~\ref{l:fixed-scale-gvn} to the weight
$$
    g(y)\coloneqq
    \Lambda(qy+1)-\Lambda_{\Siegel,\widetilde Q}(qy+1).
$$
Since $\log(qM+1)\asymp_d\log N$ and $\widetilde Q=w_*(N)$, we have, for
sufficiently large $N$,
$$
\exp\bigl((\log\log(qM+1))^{1/c_s}\bigr)
\le \widetilde Q
\le \exp\bigl((\log(qM+1))^{1/10}\bigr).
$$
Thus Lemma~\ref{l:prime-siegel-gowers} applies with model scale $X=N$ and
Gowers-norm scale $Y=qM+1$. Together with
Lemma~\ref{le:dilate}, it gives
\begin{align*}
&\left|
T_{N,M,\Lambda-\Lambda_{\Siegel,\widetilde Q},q}
(f_0,\ldots,f_{k-1})
\right|\\
&\qquad\ll_{C,d,k}
\left\|
\Lambda(q\cdot+1)-\Lambda_{\Siegel,\widetilde Q}(q\cdot+1)
\right\|_{U^s[M]}\\
&\qquad\ll_{C,d,k}
q^{(s+1)/2^s}
\left\|
\Lambda-\Lambda_{\Siegel,\widetilde Q}
\right\|_{U^s[qM+1]}
\ll_{C,d,k}
\exp\bigl(-\tfrac{1}{2}(\log N)^{\alpha_{d,k}}\bigr),
\end{align*}
where the last estimate absorbs the dilation loss, since
$\log q\le(\log N)^{2c'}$ and
$2c'<\alpha_{d,k}/100$. In particular,
\begin{equation}\label{e:prime-siegel-error}
 \left|
 T_{N,M,\Lambda-\Lambda_{\Siegel,\widetilde Q},q}
 (f_0,\ldots,f_{k-1})
 \right|
 \ll_{C,d,k}
 \exp\bigl(-(\log N)^{5c'}\bigr).
\end{equation}

\emph{A lower bound for the prime mass.}\par
By \eqref{e:chaining-saturated-modulus} and $2c'<c_0$, we have
$q\le\exp((\log N)^{2c'})\le\widetilde Q$.
Moreover, for $N$ sufficiently large, the definition of $M$ and the bound on
$q$ give
$$
 N^{1/(2d)}\le qM\le N^{2d}.
$$
Lemma~\ref{l:landau-page-compatible}, applied with $A=2d$, therefore gives,
for some $c>0$,
$$
 \mathbb E_{y\le M}\Lambda(qy+1)
 =
 \frac{q}{\phi(q)}
 \left(
 1-\frac{(qM+1)^{\beta-1}}{\beta}
 \1_{q_{\Siegel}\mid q}
 \right)
 +O\bigl(\exp(-c\sqrt{\log N})\bigr),
$$
where the exceptional term is omitted if no Siegel zero is present.
When $q_{\Siegel}\mid q$, an effective exceptional zero bound
\cite[Theorem~5.28(2)]{iwaniec-kowalski}, applied with
$\varepsilon=1$, gives
\begin{equation}\label{e:chaining-exceptional-zero}
    1-\beta
    \gg q_{\Siegel}^{-1}
    \ge q^{-1}
    \ge
    \exp\bigl(-(\log N)^{2c'}\bigr).
\end{equation}
Since $\log(qM+1)\asymp_d\log N,$
the elementary estimate
$$
1-\frac{(qM+1)^{\beta-1}}{\beta}
\gg
\min\left\{
1,(1-\beta)\log(qM+1)
\right\}
$$
implies
$$
1-\frac{(qM+1)^{\beta-1}}{\beta}
\gg
\exp\bigl(-O((\log N)^{2c'})\bigr).
$$
Therefore, in all cases,
\begin{align}\label{e:lowvonm}
    \mathbb E_{y\le M}\Lambda(qy+1)
    \gg
    \exp\bigl(-O((\log N)^{2c'})\bigr).
\end{align}
\medskip

\emph{Removing the Siegel effect.}\par
We claim that, for $\tau\gg_{C,d,k}\delta^k$, whenever
\begin{align}\label{T11}
&\left|
T_{N,M,\Lambda_{\Siegel,\widetilde Q},q}
(\1_A,\ldots,\1_A)
-
T_{N,M,\Lambda_{\Siegel,\widetilde Q},q}
(\delta\1_{[N]},\ldots,\delta\1_{[N]})
\right|\ge
\tau\mathbb E_{y\le M}\Lambda(qy+1),
\end{align}
then
\begin{align}\label{e:T1.15}
&\sup_{\substack{t\in\mathbb N\\M_{\min}\le t\le M}}
\left|
T_{N,t,\Lambda_{\Cramer,\widetilde Q},q}
(\1_A,\ldots,\1_A)
-
T_{N,t,\Lambda_{\Cramer,\widetilde Q},q}
(\delta\1_{[N]},\ldots,\delta\1_{[N]})
\right|\\
&\qquad\ge
\frac{\tau}{2}\frac{q}{\phi(q)}.
\nonumber
\end{align}

If no Siegel zero is present, then
$\Lambda_{\Siegel,\widetilde Q}=\Lambda_{\Cramer,\widetilde Q}$, while
Lemma~\ref{l:landau-page-compatible} gives
$$
\mathbb E_{y\le M}\Lambda(qy+1)
=
\frac{q}{\phi(q)}
+
O\left(\exp(-c\sqrt{\log N})\right).
$$
Thus the claim follows immediately from \eqref{T11}. We may therefore
suppose that a Siegel zero is present.
Suppose first that $q_{\Siegel}\nmid q$. On the progression
$q\mathbb Z+1$, one has
$$
\Lambda_{\Siegel,\widetilde Q}(qy+1)-\Lambda_{\Cramer,\widetilde Q}(qy+1)
=
-\Lambda_{\Cramer,\widetilde Q}(qy+1)
\chi_{\Siegel}(qy+1)(qy+1)^{\beta-1}.
$$
Applying Proposition~\ref{l:fixed-scale-gvn} to
$P_1,\ldots,P_{k-1}$ and the weight
$$
    g(y)\coloneqq\frac{\phi(q)}{q}
    \bigl(
    \Lambda_{\Siegel,\widetilde Q}(qy+1)-
    \Lambda_{\Cramer,\widetilde Q}(qy+1)
    \bigr)
$$
gives
$$
\begin{aligned}
&\left|
T_{N,M,\Lambda_{\Siegel,\widetilde Q},q}(f_0,\ldots,f_{k-1})
-
T_{N,M,\Lambda_{\Cramer,\widetilde Q},q}(f_0,\ldots,f_{k-1})
\right|\\
&\ll_{C,d,k}
\frac{q}{\phi(q)}
\left\|
\frac{\phi(q)}{q}
\bigl(
\Lambda_{\Siegel,\widetilde Q}(q\cdot+1)
-
\Lambda_{\Cramer,\widetilde Q}(q\cdot+1)
\bigr)
\right\|_{U^s[M]}.
\end{aligned}
$$
We apply Lemma~\ref{l:twisted-cramer} with progression modulus $q$,
residue class $1$, parameter $\sigma$, and Cram\'er cutoff
$\widetilde Q$. Its hypotheses hold because
$$
\mathcal L(\sigma)\mid q,
\qquad
P^+(q)\le\sigma<\widetilde Q,
\qquad
(1,q)=1.
$$
The scale hypothesis \eqref{eq:twisted-flexible-scale} also holds. Indeed,
$$
 N^{1/d}\exp\bigl(-d\sqrt{\log N}\bigr)
 \le M\le N^{1/d}
$$
for $N$ sufficiently large. These bounds follow from the definition of
$M$ and the estimate
$\log q\le(\log N)^{2c'}=o(\sqrt{\log N})$. Moreover, by
\eqref{e:constantdesign1},
$$
qq_{\Siegel}
\le
\exp\bigl((\log N)^{2c'}\bigr)\widetilde Q
\le
\exp\left(
(\log N)^{2c'}
+
2(\log N)^{c_0}
\right)
\le
\exp\bigl((\log N)^{1/10}\bigr)
$$
for $N$ sufficiently large. Lemma~\ref{l:twisted-cramer} and the definition
of $\eta_0$ therefore give
$$
\left\|
\frac{\phi(q)}{q}
\bigl(
\Lambda_{\Siegel,\widetilde Q}(q\cdot+1)
-
\Lambda_{\Cramer,\widetilde Q}(q\cdot+1)
\bigr)
\right\|_{U^s[M]}
\ll_{C,d,k}\sigma^{-\eta_0},
$$
and therefore
\begin{equation}\label{e:passsicr}
\begin{aligned}
&\left|
T_{N,M,\Lambda_{\Siegel,\widetilde Q},q}(f_0,\ldots,f_{k-1})
-
T_{N,M,\Lambda_{\Cramer,\widetilde Q},q}(f_0,\ldots,f_{k-1})
\right|\ll_{C,d,k}
\frac{q}{\phi(q)}\sigma^{-\eta_0}.
\end{aligned}
\end{equation}
Since
  $\sigma\ge
  \mathbf K_{d,k}\delta^{-\mathbf C_{d,k}}$, by~\eqref{e:constantdesign2} we have
$$
    \sigma^{-\eta_0}
    \le
    \mathbf K_{d,k}^{-\eta_0}\delta^{100k}
    \ll_{C,d,k}
    \mathbf K_{d,k}^{-\eta_0}\tau,
$$
which is sufficiently small compared with $\tau$ by the choice of
$\mathbf K_{d,k}$.
Moreover, Lemma~\ref{l:landau-page-compatible} gives
\begin{equation}\label{e:lpest}
\tau\mathbb E_{y\le M}\Lambda(qy+1)
=
\tau\frac{q}{\phi(q)}
+
O\left(
\tau\exp(-c\sqrt{\log N})
\right).
\end{equation}
Applying \eqref{e:passsicr} to
$(\1_A,\ldots,\1_A)$ and
$(\delta\1_{[N]},\ldots,\delta\1_{[N]})$,
we conclude from \eqref{T11} that
$$
\begin{aligned}
&\left|
T_{N,M,\Lambda_{\Cramer,\widetilde Q},q}
(\1_A,\ldots,\1_A)
-
T_{N,M,\Lambda_{\Cramer,\widetilde Q},q}
(\delta\1_{[N]},\ldots,\delta\1_{[N]})
\right|\ge
\frac{\tau}{2}\frac{q}{\phi(q)}.
\end{aligned}
$$
Taking $t=M$ in the supremum in \eqref{e:T1.15} proves the claim when
$q_{\Siegel}\nmid q$.

\medskip

Suppose now that $q_{\Siegel}\mid q$. By
Lemma~\ref{l:landau-page-compatible},
$$
\mathbb E_{y\le M}\Lambda(qy+1)
=
\frac{q}{\phi(q)}
\left(
1-\frac{(qM+1)^{\beta-1}}{\beta}
\right)
+
O\left(\exp(-c\sqrt{\log N})\right),
$$
and hence \eqref{T11} gives
\begin{align}\label{e:lowerboundsiegelpage}
&\left|
T_{N,M,\Lambda_{\Siegel,\widetilde Q},q}
(\1_A,\ldots,\1_A)
-
T_{N,M,\Lambda_{\Siegel,\widetilde Q},q}
(\delta\1_{[N]},\ldots,\delta\1_{[N]})
\right|\\
&\qquad\ge
\tau\frac{q}{\phi(q)}
\left(
1-\frac{(qM+1)^{\beta-1}}{\beta}
\right)
-
O\left(
\tau\exp(-c\sqrt{\log N})
\right).
\nonumber
\end{align}
Since $q_{\Siegel}\mid q$, we have $\chi_{\Siegel}(qy+1)=1,$
and therefore
$$
\Lambda_{\Siegel,\widetilde Q}(qy+1)
=
\Lambda_{\Cramer,\widetilde Q}(qy+1)
\bigl(1-(qy+1)^{\beta-1}\bigr).
$$
We remove the factor $1-(qy+1)^{\beta-1}$ by partial summation.
Let $g$ be any $1$-bounded function and set
$$
S_g(t)\coloneqq
\sum_{y\le t}\Lambda_{\Cramer,\widetilde Q}(qy+1)g(y).
$$
Partial summation gives
$$
\begin{aligned}
&\left|
\mathbb E_{y\le M}
\Lambda_{\Cramer,\widetilde Q}(qy+1)
\bigl(1-(qy+1)^{\beta-1}\bigr)g(y)
\right|\\
&\le
\frac{1}{M}
\bigl(1-(qM+1)^{\beta-1}\bigr)|S_g(M)|\\
&\quad+
\frac{1}{M}\int_1^M
(1-\beta)q(qt+1)^{\beta-2}|S_g(t)|
\,\mathrm dt
\\
&\quad+
O(M^{-1}\log M).
\end{aligned}
$$
Splitting the integral at
$M\exp(-\frac{1}{2}\sqrt{\log M})$, and using $|S_g(t)|\ll t\log \widetilde Q$
on the initial range, gives an error
$O(\exp(-c\sqrt{\log M}))$. For $N$ sufficiently large, every real
$t\ge M\exp(-\frac{1}{2}\sqrt{\log M})$ satisfies
$\lfloor t\rfloor\ge M_{\min}$. Hence, on the remaining range,
$$
|S_g(t)|
=
|S_g(\lfloor t\rfloor)|
\le
t\sup_{\substack{u\in\mathbb N\\M_{\min}\le u\le M}}
\left|
\mathbb E_{y\le u}
\Lambda_{\Cramer,\widetilde Q}(qy+1)g(y)
\right|.
$$
Moreover,
$$
\frac{1}{M}
\int_{M\exp(-\frac{1}{2}\sqrt{\log M})}^{M}
t(1-\beta)q(qt+1)^{\beta-2}\,\mathrm dt
\ll
\frac{1-(qM+1)^{\beta-1}}{\log(qM+1)}.
$$
Consequently,
\begin{align}\label{e:afterpartial}
&\left|
\mathbb E_{y\le M}
\Lambda_{\Cramer,\widetilde Q}(qy+1)
\bigl(1-(qy+1)^{\beta-1}\bigr)g(y)
\right|\\
&\le
\left(
1+O\left(\frac{1}{\log(qM+1)}\right)
\right)
\bigl(1-(qM+1)^{\beta-1}\bigr)
\nonumber\\
&\times
\sup_{\substack{t\in\mathbb N\\M_{\min}\le t\le M}}
\left|
\mathbb E_{y\le t}
\Lambda_{\Cramer,\widetilde Q}(qy+1)g(y)
\right|
+
O\left(\exp(-c\sqrt{\log M})\right).
\nonumber
\end{align}
We apply this with
$$
g(n)\coloneqq
\frac{1}{N}\sum_x
\left(
\1_A(x)
\prod_{i=1}^{k-1}\1_A(x+P_i(n))
-
\delta^k\1_{[N]}(x)
\prod_{i=1}^{k-1}\1_{[N]}(x+P_i(n))
\right).
$$
This function is $1$-bounded. Using \eqref{e:lowerboundsiegelpage}, we conclude that
\begin{align}\label{e:subginafter}
&\left(
1+O\left(\frac{1}{\log(qM+1)}\right)
\right)
\bigl(1-(qM+1)^{\beta-1}\bigr)\\
&\quad\times
\sup_{\substack{t\in\mathbb N\\M_{\min}\le t\le M}}
\left|
T_{N,t,\Lambda_{\Cramer,\widetilde Q},q}
(\1_A,\ldots,\1_A)
-
T_{N,t,\Lambda_{\Cramer,\widetilde Q},q}
(\delta\1_{[N]},\ldots,\delta\1_{[N]})
\right|
\nonumber\\
&\ge
\tau\frac{q}{\phi(q)}
\left(
1-\frac{(qM+1)^{\beta-1}}{\beta}
\right)
-
O\left(
(1+\tau)\exp(-c\sqrt{\log N})
\right).
\nonumber
\end{align}
The elementary estimate
$$
\frac{(1-\beta)(qM+1)^{\beta-1}}
     {1-(qM+1)^{\beta-1}}
\ll
\frac{1}{\log(qM+1)}
$$
gives
$$
1-\frac{(qM+1)^{\beta-1}}{\beta}
=
\left(
1+O\left(\frac{1}{\log(qM+1)}\right)
\right)
\bigl(1-(qM+1)^{\beta-1}\bigr).
$$
Moreover, \eqref{e:chaining-exceptional-zero} gives
$$
1-(qM+1)^{\beta-1}
\gg
\exp\bigl(-O((\log N)^{2c'})\bigr).
$$
Together with
$\tau\gg\delta^k$ and
$\delta>(\log N)^{-1/\mathbf C_{d,k}^{2}}$, this shows that the error
in \eqref{e:subginafter} is negligible after division by
$1-(qM+1)^{\beta-1}$. Therefore
$$
\begin{aligned}
&\sup_{\substack{t\in\mathbb N\\M_{\min}\le t\le M}}
\left|
T_{N,t,\Lambda_{\Cramer,\widetilde Q},q}
(\1_A,\ldots,\1_A)
-
T_{N,t,\Lambda_{\Cramer,\widetilde Q},q}
(\delta\1_{[N]},\ldots,\delta\1_{[N]})
\right|\\
&\qquad\ge
(1-o(1))\tau\frac{q}{\phi(q)}
\ge
\frac{\tau}{2}\frac{q}{\phi(q)}
\end{aligned}
$$
for $N$ sufficiently large. This proves \eqref{e:T1.15}.

\medskip

\emph{Completing the chain of approximations.}\par
Applying \eqref{e:prime-siegel-error} once with every input equal to
$\1_A$, and once with every input equal to
$\delta\1_{[N]}$, the triangle inequality gives
$$
\begin{aligned}
&\left|
T_{N,M,\Lambda_{\Siegel,\widetilde Q},q}
(\1_A,\ldots,\1_A)
-
T_{N,M,\Lambda_{\Siegel,\widetilde Q},q}
(\delta\1_{[N]},\ldots,\delta\1_{[N]})
\right|\\
&\ge
\left|
T_{N,M,\Lambda,q}
(\1_A,\ldots,\1_A)
-
T_{N,M,\Lambda,q}
(\delta\1_{[N]},\ldots,\delta\1_{[N]})
\right|
-
O\left(
\exp\bigl(-(\log N)^{5c'}\bigr)
\right).
\end{aligned}
$$
By \eqref{e:lowvonm},
$$
\mathbb E_{y\le M}\Lambda(qy+1)
\gg
\exp\bigl(-O((\log N)^{2c'})\bigr).
$$
By this, the assumption
$\delta>(\log N)^{-1/\mathbf C_{d,k}^{2}}$, and the inequality
$5c'>2c'$, we have
$$
\exp\bigl(-(\log N)^{5c'}\bigr)
=
o\left(
\delta^k\mathbb E_{y\le M}\Lambda(qy+1)
\right).
$$
Hence, by \eqref{e:step3goal}, for $N$ sufficiently large,
$$
\begin{aligned}
&\left|
T_{N,M,\Lambda_{\Siegel,\widetilde Q},q}
(\1_A,\ldots,\1_A)
-
T_{N,M,\Lambda_{\Siegel,\widetilde Q},q}
(\delta\1_{[N]},\ldots,\delta\1_{[N]})
\right|\\
&\ge
\frac{5c_1}{6}
\delta^k\mathbb E_{y\le M}\Lambda(qy+1).
\end{aligned}
$$
Applying the claim in \eqref{e:T1.15} with
$\tau=\frac{5c_1}{6}\delta^k$, we obtain an integer
$t_0$ with $M_{\min}\le t_0\le M$ such that
\begin{align}\label{e:siegelremovalbound}
&\left|
T_{N,t_0,\Lambda_{\Cramer,\widetilde Q},q}
(\1_A,\ldots,\1_A)
-
T_{N,t_0,\Lambda_{\Cramer,\widetilde Q},q}
(\delta\1_{[N]},\ldots,\delta\1_{[N]})
\right|\ge
\frac{5c_1}{12}
\delta^k\frac{q}{\phi(q)}.
\end{align}
We compare this Cram\'er model with the constant model.
Since $M_{\min}\le t_0\le M$, we have
$\log t_0\asymp_d\log N$. Hence, for $N$ sufficiently large,
$$
    q\le\exp\bigl((\log t_0)^{2/5}\bigr),
    \qquad
    \widetilde Q+1
    \le\exp\bigl((\log t_0)^{1/10}\bigr).
$$
Moreover, the divisibility hypotheses imply
$$
    \prod_{p\le\sigma}p\mid q,
    \qquad P^+(q)\le\sigma,
    \qquad (1,q)=1.
$$
Proposition~\ref{p:saturated-cramer-gowers}, applied with
$$
    M=t_0,\qquad J=[t_0],\qquad w=\sigma,
    \qquad z=\widetilde Q,\qquad Q=q,\qquad b=1,
$$
therefore gives
$$
\left\|
\frac{\phi(q)}{q}
\Lambda_{\Cramer,\widetilde Q}(q\cdot+1)-1
\right\|_{U^s[t_0]}
\ll_{d,k}\sigma^{-\eta_0}.
$$
Since $\Lambda_{\Cramer,\sigma}(qy+1)=q/\phi(q)$,
Proposition~\ref{l:fixed-scale-gvn} now gives, for any $1$-bounded functions
$h_0,\ldots,h_{k-1}$ supported on $[N]$,
$$
\left|
T_{N,t_0,\Lambda_{\Cramer,\widetilde Q}-\Lambda_{\Cramer,\sigma},q}
(h_0,\ldots,h_{k-1})
\right|
\ll_{C,d,k}
\frac{q}{\phi(q)}\sigma^{-\eta_0}.
$$
By \eqref{e:chaining-cutoff-range}, \eqref{e:constantdesign2}, and the choice
of $\mathbf K_{d,k}$, increased if necessary to absorb the implicit constant,
the right-hand side is at most
$$
    2^{-4k-40}c_1\delta^k\frac{q}{\phi(q)}.
$$
Applying this to the two tuples in \eqref{e:siegelremovalbound}, we obtain
\begin{align}\label{e:difflaststep4}
&\left|
T_{N,t_0,\Lambda_{\Cramer,\sigma},q}
(\1_A,\ldots,\1_A)
-
T_{N,t_0,\Lambda_{\Cramer,\sigma},q}
(\delta\1_{[N]},\ldots,\delta\1_{[N]})
\right|\\
&\ge
\frac{5c_1}{12}
\delta^k\frac{q}{\phi(q)}
-\frac{c_1}{24}\delta^k\frac{q}{\phi(q)}\ge
\frac{c_1}{4}
\delta^k\frac{q}{\phi(q)}\nonumber.
\end{align}
Finally, since
$\mathcal L(\sigma)\mid q$ and $P^+(q)\le\sigma$, we have
$\Lambda_{\Cramer,\sigma}(qy+1)=q/\phi(q)$ for every
$y\in\mathbb Z$. Consequently, for every tuple of $1$-bounded functions
$h_0,\ldots,h_{k-1}$ supported on $[N]$, we have
$$
T_{N,t_0,\Lambda_{\Cramer,\sigma},q}(h_0,\ldots,h_{k-1})
=\frac{q}{\phi(q)}T_{N,t_0,1,q}(h_0,\ldots,h_{k-1}).
$$
Dividing \eqref{e:difflaststep4} by
$q/\phi(q)$ and taking $M_0=t_0$ proves
\eqref{e:step3goal2}.
\end{proof}

We use the following slight variant of the weak regularity lemma of Shao and Wang~\cite{ShW25}, which extends the work of~\cite{PP22}.

\begin{lemma}[Weak regularity]
\label{cl:short-scale-regularization}
There is a constant $A'>0$, depending on $C,d,k$, with
the following property. Let $N,q\in\mathbb N$, let
$0<\varepsilon<1/2$, and suppose that
\begin{align}\label{e:regularity-divisibility-7}
    \mathcal L\left(\left\lceil
    \varepsilon^{-A'}\right\rceil\right)\mid q,
\end{align}
and let $P_1,\ldots,P_{k-1}$ have $(C,q)$-coefficients and distinct
increasing degrees, with largest degree $d$. Set
$\nu_0\coloneqq\deg(P_1)$ and $\nu_i\coloneqq\deg(P_i)$ for
$1\le i\le k-1$. Let $M_0\in\mathbb N$ satisfy
\begin{align*}
    1\le M_0\le \left(\frac{N}{q^{d-1}}\right)^{1/d},
\end{align*}
and let $f_0,\ldots,f_{k-1}$ be $1$-bounded functions supported on
$[N]$. Then either
\begin{align}\label{e:regularity-small-alternative-7}
    M_0\ll_{C,d,k}(q/\varepsilon)^{O_{d,k}(1)},
\end{align}
or there are local factors $\mathcal B_i$ of $[N]$, $0\le i\le k-1$,
such that
\begin{align}\label{e:regularity-conclusion-7}
&\left|
T_{N,M_0,1,q}(f_0,\ldots,f_{k-1})
-
T_{N,M_0,1,q}
\bigl(
\Pi_{\mathcal B_0}f_0,\ldots,
\Pi_{\mathcal B_{k-1}}f_{k-1}
\bigr)
\right|
\le\varepsilon.
\end{align}
The factor $\mathcal B_i$ has resolution at least
\begin{align}\label{e:regularity-resolution-7}
    q^{-O_{d,k}(1)}
    \varepsilon^{O_{d,k}(1)}
    M_0^{\nu_i},
\end{align}
and all the factors have common modulus
\begin{align*}
    q^B,\qquad 1\leq B=O_d(1).
\end{align*}
\end{lemma}

\begin{proof}
We first record the inverse input. Suppose that
$g_0,\ldots,g_{k-1}$ are $1$-bounded, supported on $[N]$, and
\begin{align}\label{e:regularity-inverse-hypothesis-7}
    \left|T_{N,M_0,1,q}(g_0,\ldots,g_{k-1})\right|\ge\rho.
\end{align}
From Lemma~\ref{l:SW-local-form} we see that, unless
$M_0\ll_{C,d,k}(q/\rho)^{O_{d,k}(1)}$, there are common positive integers
\begin{align*}
    q'\ll_{C,d,k}\rho^{-O_{d,k}(1)},
    \qquad b=O_d(1),
\end{align*}
and, for each $1\le i\le k-1$, a $1$-bounded function $\psi_i$ such that
\begin{align}\label{e:regularity-inverse-correlation-7}
    \left|\sum_xg_i(x)\psi_i(x)\right|
    &\gg_{C,d,k}\rho^{O_{d,k}(1)}N.
\end{align}
The functions $\psi_i$ satisfy the Lipschitz estimate
\begin{align}
    |\psi_i(x+q'q^bt)-\psi_i(x)|
    &\ll_{C,d,k}
    (q/\rho)^{O_{d,k}(1)}
    M_0^{-\nu_i}|t|.
    \label{e:regularity-inverse-lipschitz-7}
\end{align}
The same type of conclusion holds for $i=0$. Indeed, the change of variables
$x'=x+P_1(y)$ leaves the functions unchanged, so they remain supported on
$[N]$, and reanchors the shifted system as
\begin{align*}
    -P_1,\quad P_2-P_1,\quad\ldots,\quad P_{k-1}-P_1.
\end{align*}
This system has distinct increasing degrees and $(2C^2,q)$-coefficients.
Applying Lemma~\ref{l:SW-local-form} to the reanchored average supplies a
common pair $q',b$ for its shifted inputs, possibly different from the pair
supplied by the first application. In particular, $g_0$ is attached to
$-P_1$, so its Lipschitz scale is
$M_0^{-\deg(P_1)}=M_0^{-\nu_0}$.

We now construct the factors. The bounds above imply that there is
$A=O_{d,k}(1)$ such that, when $\rho=\varepsilon/(2k)$, the
correlation in \eqref{e:regularity-inverse-correlation-7} is at least
$\eta N$, where
\begin{align*}
    \eta\gg_{C,d,k}\varepsilon^A,
\end{align*}
and the Lipschitz constant in
\eqref{e:regularity-inverse-lipschitz-7} is at most
\begin{align*}
    L_i\ll_{C,d,k}
    q^A\varepsilon^{-A}M_0^{-\nu_i}.
\end{align*}
Choose $A'$ so that every integer $q'$ arising from
such an application satisfies
$q'\le\lceil\varepsilon^{-A'}\rceil$. By
\eqref{e:regularity-divisibility-7}, we then have $q'\mid q$. Choose
$B=O_d(1)$ larger than every possible $b+1$, so that
$q'q^b\mid q^B$.

Choose $C_2=O_{d,k}(1)$ sufficiently large that
\begin{align}\label{e:regularity-C2-conditions-7}
    C_2\ge2A+B,
    \qquad
    C_2+d-1\ge B,
\end{align}
and then choose $0<c_2\le1/4$ sufficiently small in terms of
$C,d,k$. Set
\begin{align*}
    S_i\coloneqq
    \max\left\{
    1,
    \left\lfloor
    c_2q^{-C_2}\varepsilon^{C_2}M_0^{\nu_i}
    \right\rfloor
    \right\}.
\end{align*}
We construct $\mathcal B_i$ as a local factor of modulus $q^B$ and
resolution $S_i$. Fix a residue class modulo $q^B$ and list its
points in $[N]$ in increasing order. If $S_i=1$, partition these
points into singletons. If $S_i\ge2$, then
\begin{align*}
    S_i
    &\le c_2q^{-C_2}\varepsilon^{C_2}M_0^{\nu_i}
    \le c_2q^{-C_2}\frac{N}{q^{d-1}}
    \le \frac{N}{4q^B},
\end{align*}
where we used $\nu_i\le d$, the upper bound on $M_0$,
\eqref{e:regularity-C2-conditions-7}, and $c_2\le1/4$. Thus every
nonempty residue class has at least
$\lfloor N/q^B\rfloor\ge4S_i-1\ge S_i$ points. Partition its
ordered points into consecutive blocks of cardinality $S_i$, merging
the final block with the preceding block if it has fewer than $S_i$
points. The resulting blocks have cardinality between $S_i$ and
$2S_i$ and are of the form $(q^B\mathbb Z+a)\cap I$ for an interval
$I$. They therefore form a local factor $\mathcal B_i$ of modulus
$q^B$ and resolution $S_i$.

Since $\max\{1,\lfloor u\rfloor\}\ge u/2$ for $u>0$, we have
\begin{align*}
    S_i
    \ge \frac{c_2}{2}q^{-C_2}\varepsilon^{C_2}M_0^{\nu_i},
\end{align*}
which proves \eqref{e:regularity-resolution-7}. If $x,x'$ lie in the
same atom of $\mathcal B_i$, then
\begin{align}\label{e:regularity-atom-difference-7}
    x-x'=q^Bt
    \qquad\text{for some }|t|\le2S_i.
\end{align}

Suppose that \eqref{e:regularity-conclusion-7} fails. Set
\begin{align*}
    h_i=\Pi_{\mathcal B_i}f_i,
    \qquad r_i=f_i-h_i.
\end{align*}
A telescoping expansion gives an index $i$ for which the counting
form with $r_i$ in the $i$th coordinate has magnitude greater than
$\varepsilon/k$. Since $|r_i|\le2$, we can apply the inverse input above
with $r_i/2$ in that coordinate and $\rho=\varepsilon/(2k)$. The
small alternative is \eqref{e:regularity-small-alternative-7};
otherwise there is a function $\psi_i$ satisfying
\eqref{e:regularity-inverse-correlation-7} and
\eqref{e:regularity-inverse-lipschitz-7}. Absorbing the factor $2$
from $r_i/2$ into the implicit constant, we may write the correlation
as
\begin{align}\label{e:regularity-correlation-contradiction-7}
    \left|\sum_xr_i(x)\psi_i(x)\right|\ge\eta N,
    \qquad
    \eta\gg_{C,d,k}\varepsilon^A.
\end{align}
If $x,x'$ lie in a nonsingleton atom, then
\eqref{e:regularity-atom-difference-7}, the divisibility
$q'q^b\mid q^B$, and \eqref{e:regularity-inverse-lipschitz-7} give
\begin{align*}
    |\psi_i(x)-\psi_i(x')|
    &\le2L_iq^BS_i\\
    &\ll_{C,d,k}
    c_2q^{A+B-C_2}\varepsilon^{C_2-A}
    \le\frac{\eta}{4},
\end{align*}
by \eqref{e:regularity-C2-conditions-7} and the choice of $c_2$.
The same bound is immediate for singleton atoms.

On every atom $P$ of $\mathcal B_i$, we have
$\sum_{x\in P}r_i(x)=0$. Fixing $x_P\in P$, we therefore obtain
\begin{align*}
    \left|\sum_xr_i(x)\psi_i(x)\right|
    &=
    \left|\sum_P\sum_{x\in P}
    r_i(x)\bigl(\psi_i(x)-\psi_i(x_P)\bigr)\right|\\
    &\le
    \sum_P\sum_{x\in P}2\cdot\frac{\eta}{4}
    =\frac{\eta N}{2},
\end{align*}
contradicting \eqref{e:regularity-correlation-contradiction-7}. This proves
\eqref{e:regularity-conclusion-7}.
\end{proof}

Increase $\mathbf C_{d,k}$ and $\mathbf K_{d,k}$, if necessary, so that
\begin{align}\label{e:constantdesign3}
\mathbf C_{d,k}&\ge3kA',
&
\mathbf K_{d,k}&\ge(16/c_1)^{A'}+1.
\end{align}
Proposition~\ref{lem:chaining} remains valid after these enlargements, since
its assumptions only become stronger.
Then, for every $0<\delta\le1$,
\begin{align}\label{e:regularity-cutoff-comparison}
    \left\lceil
    ((c_1/16)\delta^{3k})^{-A'}
    \right\rceil
    \le
    \left\lceil
    \mathbf K_{d,k}\delta^{-\mathbf C_{d,k}}
    \right\rceil.
\end{align}
The choice ensures that every auxiliary modulus supplied by the Shao--Wang
inverse theorem at the accuracies used below divides $q$ whenever the
right-hand side of \eqref{e:regularity-cutoff-comparison} is at most $\sigma$
and $\mathcal L(\sigma)\mid q$. This property is preserved when $q$ is
replaced by a positive power of itself.

\begin{lemma}[Density increment]\label{l:case3}
Let $N$ be sufficiently large, and suppose that
$(\log N)^{-1/\mathbf C_{d,k}^{2}}<\delta\le1.$
Let $\sigma\ge2$ be an integer satisfying
\begin{equation}\label{e:case3-cutoff}
    \left\lceil
    \mathbf K_{d,k}\delta^{-\mathbf C_{d,k}}
    \right\rceil
    \le\sigma\le(\log N)^{c'}.
\end{equation}
Let $q\in\mathbb N$ satisfy
\begin{equation}\label{e:case3-saturation}
    q\le\exp\bigl((\log N)^{2c'}\bigr),
    \qquad
    P^+(q)\le\sigma,
    \qquad
    \mathcal L(\sigma)\mid q.
\end{equation}
Set
$$
    M\coloneqq
    \left\lfloor
    \left(\frac{N}{q^{d-1}}\right)^{1/d}
    \right\rfloor.
$$
Let $P_1,\ldots,P_{k-1}$ have $(C,q)$-coefficients. Assume that their
degrees are distinct and increasing, with $\deg(P_{k-1})=d$. Set
$\nu_0\coloneqq\deg(P_1)$ and $\nu_i\coloneqq\deg(P_i)$ for
$1\le i\le k-1$. Suppose that $A\subseteq[N]$,
$|A|=\delta N$, and that $A$ contains no configuration
\begin{align*}
&x,x+P_1(y),\ldots,x+P_{k-1}(y)\in A,
\qquad qy+1\in\mathbb P,\\
&0,P_1(y),\ldots,P_{k-1}(y) \textup{ are pairwise distinct}.
\end{align*}
Then there exist $c=c(C,d,k)>0$, integers $B,N''\geq 1$, and $a$, with
$B=O_d(1)$ and $N''>0$, such that
$$
    N''\ge
    q^{-O_{d,k}(1)}
    \delta^{O_{d,k}(1)}
    M^{\deg(P_1)}
    \exp\bigl(-d\sqrt{\log M}\bigr).
$$
Moreover, $a+q^B[N'']\subseteq[N]$
and
\begin{equation}\label{e:case3-density-increment}
    \left|
    A\cap\bigl(a+q^B[N'']\bigr)
    \right|
    \ge
    (1+c)\delta N''.
\end{equation}
\end{lemma}

\begin{proof}
By the density assumption and \eqref{e:constantdesign1}, for $N$
sufficiently large,
\begin{align*}
    \delta
    &>
    (\log N)^{-1/\mathbf C_{d,k}^{2}}
    >
    \exp\bigl(-(\log N)^{1/100}\bigr),\\
    q
    &\le
    \exp\bigl((\log N)^{2c'}\bigr)
    \le w_*(N),
\end{align*}
where the last inequality follows from $2c'<c_0$.
Thus all the hypotheses of
Lemma~\ref{lem:full-set-discrepancy} are satisfied. Hence
$$
\begin{aligned}
&\left|
T_{N,M,\Lambda,q}(\1_A,\ldots,\1_A)
-
T_{N,M,\Lambda,q}
(\delta\1_{[N]},\ldots,\delta\1_{[N]})
\right|\ge
c_1\delta^k
\mathbb E_{y\le M}\Lambda(qy+1).
\end{aligned}
$$
We may therefore apply Proposition~\ref{lem:chaining} with the cutoff
$\sigma$. It gives an integer $M_0$ satisfying
\begin{align}\label{e:density-increment-short-scale-7}
    \left\lceil M\exp(-\sqrt{\log M})\right\rceil
    \le M_0\le M
\end{align}
and
\begin{align}\label{e:density-increment-unweighted-7}
&\left|
T_{N,M_0,1,q}(\1_A,\ldots,\1_A)
-
T_{N,M_0,1,q}
(\delta\1_{[N]},\ldots,\delta\1_{[N]})
\right|
\ge\frac{c_1}{4}\delta^k.
\end{align}

Set $\varepsilon=(c_1/16)\delta^{3k}$. By
\eqref{e:regularity-cutoff-comparison}, \eqref{e:case3-cutoff}, and
\eqref{e:case3-saturation},
\begin{align*}
    \mathcal L\left(
    \left\lceil
    \varepsilon^{-A'}
    \right\rceil
    \right)\mid q.
\end{align*}
We may therefore apply Lemma~\ref{cl:short-scale-regularization} at
the shift scale $M_0$ to
$f_0=\cdots=f_{k-1}=\1_A$. If its small alternative holds,
then
\begin{align*}
    M_0\ll_{C,d,k}(q/\delta)^{O_{d,k}(1)}.
\end{align*}
It follows from the upper bounds on $q$ and $\delta^{-1}$ that
\begin{align*}
    \log M_0
    \ll_{C,d,k}
    (\log N)^{2c'}+\log\log N.
\end{align*}
On the other hand, the definition of $M$, the upper bound on $q$,
and \eqref{e:density-increment-short-scale-7} give, for $N$
sufficiently large,
\begin{align*}
    \log M_0
    \ge\log M-\sqrt{\log M}
    \ge\frac{1}{3d}\log N.
\end{align*}
These estimates contradict \eqref{e:constantdesign1} for $N$
sufficiently large.

We therefore obtain local factors $\mathcal B_i$, $0\le i\le k-1$,
with common modulus $q^B$, $B=O_d(1)$, and resolution at least
\begin{align}\label{e:density-increment-resolution-7}
    q^{-O_{d,k}(1)}
    \delta^{O_{d,k}(1)}
    M_0^{\nu_i}.
\end{align}
Set $F_i=\Pi_{\mathcal B_i}\1_A$. By
\eqref{e:regularity-conclusion-7},
\eqref{e:density-increment-unweighted-7}, and
$\varepsilon\le(c_1/16)\delta^k$,
\begin{align}\label{e:regularized-unweighted-discrepancy-7}
&\left|
T_{N,M_0,1,q}(F_0,\ldots,F_{k-1})
-
T_{N,M_0,1,q}
(\delta\1_{[N]},\ldots,\delta\1_{[N]})
\right|
\ge\frac{c_1}{8}\delta^k.
\end{align}

Fix $0<c\le1$, to be chosen below, and set
\begin{align*}
    S_i\coloneqq
    \{x\in[N]:F_i(x)>(1+c)\delta\}.
\end{align*}
Suppose that every $S_i$ is empty. Then $F_i\le(1+c)\delta$, while
preservation of mass under conditional expectation gives
\begin{align*}
    \sum_x\bigl(F_i(x)-\delta\1_{[N]}(x)\bigr)=0,
    \qquad
    \|F_i-\delta\1_{[N]}\|_{\ell^1}
    \le2c\delta N.
\end{align*}
Telescoping the left-hand side of
\eqref{e:regularized-unweighted-discrepancy-7}, and using translation
invariance of the $\ell^1$-norm in the sum over $x$, now gives
\begin{align*}
&\left|
T_{N,M_0,1,q}(F_0,\ldots,F_{k-1})
-
T_{N,M_0,1,q}
(\delta\1_{[N]},\ldots,\delta\1_{[N]})
\right|\\
&\qquad\le
\frac{((1+c)\delta)^{k-1}}{N}
\sum_{i=0}^{k-1}
\|F_i-\delta\1_{[N]}\|_{\ell^1}
\le2kc(1+c)^{k-1}\delta^k.
\end{align*}
Choosing $c=c(C,d,k)>0$ sufficiently small contradicts
\eqref{e:regularized-unweighted-discrepancy-7}. Hence $S_i$ is
nonempty for some $i$.

The function $F_i$ is constant on every atom of $\mathcal B_i$, so
$S_i$ is a union of such atoms. Choose an atom $B_0\subseteq S_i$.
Writing $B_0=a+q^B[N'']$, for every $x\in B_0$ we have
\begin{align*}
    \frac{|A\cap B_0|}{|B_0|}
    =F_i(x)>(1+c)\delta.
\end{align*}
Moreover, \eqref{e:density-increment-resolution-7} and
\eqref{e:density-increment-short-scale-7} give
\begin{align*}
    N''=|B_0|
    &\ge
    q^{-O_{d,k}(1)}
    \delta^{O_{d,k}(1)}
    M_0^{\nu_i}\\
    &\ge
    q^{-O_{d,k}(1)}
    \delta^{O_{d,k}(1)}
    M^{\deg(P_1)}
    \exp\bigl(-d\sqrt{\log M}\bigr).
\end{align*}
This proves \eqref{e:case3-density-increment}.
\end{proof}

\begin{proof}[Proof of Theorem~\ref{thm_poly}]
We first restrict $A$ to a dense progression whose modulus contains all the
required small prime factors. We then iterate Lemma~\ref{l:case3}; the main
quantitative point is to show that the arising moduli and intervals remain in
the ranges required by that lemma until the density exceeds $1$.

\emph{Initial progression and rescaling.}\par
The case $k=1$ is immediate. Reindex the polynomials so that
$\deg(P_1)<\cdots<\deg(P_{k-1})=d$. Choose $C\ge1$, depending only on the
original polynomial system, so large that, for every $Q\in\mathbb N$ and every
$1\le j\le k-1$, the polynomial $P_j(Qy)/Q$ has $(C,Q)$-coefficients. Set
$$
    r\coloneqq\deg(P_1),
    \qquad
    \theta\coloneqq\frac{r}{d},
    \qquad
    \alpha\coloneqq\frac{|A|}{N}.
$$
Fix a sufficiently small constant $c_*>0$, to be chosen at the end.
There is nothing to prove if
$\alpha\le(\log N)^{-c_*}$, so suppose that
\begin{equation}\label{e:main-contradiction-density}
    \alpha>(\log N)^{-c_*}.
\end{equation}
Set $\sigma\coloneqq
    \left\lceil
    \mathbf K_{d,k}\alpha^{-\mathbf C_{d,k}}\right\rceil$ and $Q_0\coloneqq\mathcal L(\sigma).$
Then
\begin{equation}\label{e:initial-modulus-bound}
    \log Q_0
    \le \sigma\log\sigma
    \le \alpha^{-O_{C,d,k}(1)}.
\end{equation}
By \eqref{e:main-contradiction-density} and
\eqref{e:initial-modulus-bound},
$$
    \log Q_0
    \le
    \alpha^{-O_{C,d,k}(1)}
    \le
    (\log N)^{O_{C,d,k}(c_*)}.
$$
After choosing $c_*>0$ sufficiently small in terms of $C,d,k$,
we therefore have
\begin{equation}\label{e:initial-modulus-small}
    \log Q_0\le\frac{1}{4}\log N,
    \qquad\text{and hence}\qquad
    Q_0\le N^{1/4}.
\end{equation}
In particular, $Q_0\le N/10$ for $N$ sufficiently large, so the
initial pigeonhole argument applies.
Partition $[N]$ into residue classes modulo $Q_0$. By the
pigeonhole principle, there are integers $a_0$ and $N_0>0$ such that
$$
    a_0+Q_0[N_0]\subseteq[N],
    \qquad
    N_0\ge\frac{N}{Q_0}-1,
$$
and the set
$$
    A_0\coloneqq
    \{x\in[N_0]:a_0+Q_0x\in A\}
$$
has density $\delta_0\coloneqq\frac{|A_0|}{N_0}\ge\alpha.$
For $1\le j\le k-1$, define
$$
    P_{0,j}(y)\coloneqq\frac{P_j(Q_0y)}{Q_0}.
$$
Since $P_j(0)=0$, each $P_{0,j}$ is an integer polynomial. Moreover,
if
$$
    P_j(y)=\sum_{t=1}^{\deg(P_j)}a_{j,t}y^t,
$$
then
$$
    P_{0,j}(y)
    =
    \sum_{t=1}^{\deg(P_j)}
    a_{j,t}Q_0^{t-1}y^t.
$$
By the choice of $C$, the polynomials $P_{0,j}$ have
$(C,Q_0)$-coefficients.
We claim that the set $A_0$ contains no configuration
\begin{align*}
&x,\ x+P_{0,1}(y),\ldots,x+P_{0,k-1}(y)\in A_0,
\qquad Q_0y+1\in\mathbb P,\\
&0,P_{0,1}(y),\ldots,P_{0,k-1}(y) \textup{ are pairwise distinct}.
\end{align*}
Indeed, any such configuration would give
$$
\begin{aligned}
    a_0+Q_0x,\quad
    a_0+Q_0x+P_1(Q_0y),\quad\ldots,\quad
    a_0+Q_0x+P_{k-1}(Q_0y)
    \in A.
\end{aligned}
$$
Writing $p=Q_0y+1$, this is a configuration
$$
    a_0+Q_0x,\quad
    a_0+Q_0x+P_1(p-1),\quad\ldots,\quad
    a_0+Q_0x+P_{k-1}(p-1)
$$
in $A$. Moreover, for every $1\le j\le k-1$, we have $P_j(p-1)=Q_0P_{0,j}(y).$
Since multiplication by $Q_0$ preserves pairwise distinctness, the values
$0,P_1(p-1),\ldots,P_{k-1}(p-1)$ are pairwise distinct. This contradicts
the hypothesis.

\emph{The density increment step.}\par
We now iterate Lemma~\ref{l:case3}. Suppose inductively that
$A_i\subseteq[N_i]$ has density $\delta_i\ge\alpha$, that
$$
    P_{i,j}(y)
    \coloneqq
    \frac{P_j(Q_i y)}{Q_i},
    \qquad
    1\le j\le k-1,
$$
and that $A_i$ contains no configuration
\begin{align*}
&x,\ x+P_{i,1}(y),\ldots,x+P_{i,k-1}(y)\in A_i,
\qquad Q_i y+1\in\mathbb P,\\
&0,P_{i,1}(y),\ldots,P_{i,k-1}(y) \textup{ are pairwise distinct}.
\end{align*}
We also maintain
\begin{equation}\label{e:iteration-saturation-invariant}
    P^+(Q_i)\le\sigma,
    \qquad
    \mathcal L(\sigma)\mid Q_i.
\end{equation}
These conditions hold for $i=0$, since
$Q_0=\mathcal L(\sigma)$.
At stage $i$, set
$$
    M_i\coloneqq
    \left\lfloor
    \left(\frac{N_i}{Q_i^{d-1}}\right)^{1/d}
    \right\rfloor.
$$
Since $\delta_i\ge\alpha$, the definition of $\sigma$ gives
\begin{equation}\label{e:iteration-cutoff-range}
    \left\lceil
    \mathbf K_{d,k}\delta_i^{-\mathbf C_{d,k}}
    \right\rceil
    \le\sigma.
\end{equation}
In particular,
\eqref{e:iteration-saturation-invariant} supplies the arithmetic
hypotheses $P^+(Q_i)\le\sigma$ and $\mathcal L(\sigma)\mid Q_i$
required by Lemma~\ref{l:case3}.
Assume for the moment that
\begin{equation}\label{e:iteration-quantitative-hypotheses}
    \delta_i>
    (\log N_i)^{-1/\mathbf C_{d,k}^{2}},
    \qquad
    Q_i\le
    \exp\bigl((\log N_i)^{2c'}\bigr),
    \qquad
    \sigma\le(\log N_i)^{c'}.
\end{equation}
These estimates will be verified uniformly below. Together with
\eqref{e:iteration-saturation-invariant} and
\eqref{e:iteration-cutoff-range}, they show that all the hypotheses of
Lemma~\ref{l:case3} are satisfied with
$$
    N=N_i,\qquad
    \delta=\delta_i,\qquad
    \sigma=\sigma,\qquad
    q=Q_i.
$$
At this application the exceptional data in Definition~\ref{d:Siegel} are
redefined at scale $N_i$, and the chaining argument terminates
at the corresponding cutoff $w_*(N_i)$.
Thus there are integers $B_i,N_{i+1}\geq 1$, and $a_i$, with
$B_i=O_d(1)$, such that
\begin{equation}\label{e:iteration-progression-containment}
    a_i+Q_i^{B_i}[N_{i+1}]
    \subseteq[N_i],
\end{equation}
\begin{equation}\label{e:iteration-density-increment}
    \left|
    A_i\cap
    \bigl(a_i+Q_i^{B_i}[N_{i+1}]\bigr)
    \right|
    \ge
    (1+c)\delta_iN_{i+1},
\end{equation}
and
\begin{equation}\label{e:iteration-length-correct}
    N_{i+1}
    \ge
    Q_i^{-O_{d,k}(1)}
    \delta_i^{O_{d,k}(1)}
    M_i^r
    \exp\bigl(-d\sqrt{\log M_i}\bigr).
\end{equation}
Since $\delta_i\ge\alpha$, we may in particular record the uniform
lower bound
\begin{equation}\label{e:iteration-length-uniform}
    N_{i+1}
    \ge
    Q_i^{-O_{d,k}(1)}
    \alpha^{O_{d,k}(1)}
    M_i^r
    \exp\bigl(-d\sqrt{\log M_i}\bigr).
\end{equation}
Define
$$
    A_{i+1}
    \coloneqq
    \left\{
        x\in[N_{i+1}]:
        a_i+Q_i^{B_i}x\in A_i
    \right\},
    \qquad
    \delta_{i+1}
    \coloneqq
    \frac{|A_{i+1}|}{N_{i+1}}.
$$
Then \eqref{e:iteration-density-increment} gives
\begin{equation}\label{e:iteration-density-growth}
    \delta_{i+1}\ge(1+c)\delta_i.
\end{equation}
Set for $1\le j\le k-1$,
\begin{equation}\label{e:iteration-polynomials-correct}
    Q_{i+1}\coloneqq Q_i^{B_i+1},
    \qquad
    P_{i+1,j}(y)
    \coloneqq
    \frac{P_j(Q_{i+1}y)}{Q_{i+1}}.
\end{equation}
We claim that $A_{i+1}$ contains no configuration
\begin{align*}
&x,\ x+P_{i+1,1}(y),\ldots,x+P_{i+1,k-1}(y)\in A_{i+1},
\qquad Q_{i+1}y+1\in\mathbb P,\\
&0,P_{i+1,1}(y),\ldots,P_{i+1,k-1}(y) \textup{ are pairwise distinct}.
\end{align*}
If such a configuration existed, then, for every
$1\le j\le k-1$,
\begin{align*}
a_i+Q_i^{B_i}\bigl(x+P_{i+1,j}(y)\bigr)
&=
a_i+Q_i^{B_i}x
+
Q_i^{B_i}P_{i+1,j}(y)=
a_i+Q_i^{B_i}x
+
P_{i,j}\bigl(Q_i^{B_i}y\bigr),
\end{align*}
since $P_{i,j}\bigl(Q_i^{B_i}y\bigr)
    =
    Q_i^{B_i}P_{i+1,j}(y).$
Since multiplication by $Q_i^{B_i}$ preserves pairwise distinctness, the numbers
$0,P_{i,1}(Q_i^{B_i}y),\ldots, P_{i,k-1}(Q_i^{B_i}y)$
are pairwise distinct.
Set $y_i\coloneqq Q_i^{B_i}y.$
Then
$$
    Q_i y_i+1
    =
    Q_i^{B_i+1}y+1
    =
    Q_{i+1}y+1
    \in\mathbb P.
$$
Thus the affine map $x\longmapsto a_i+Q_i^{B_i}x$
would produce a forbidden configuration in $A_i$ with shift
parameter $y_i$. This proves the claim.

The polynomials $P_{i+1,j}$ also have
$(C,Q_{i+1})$-coefficients 
and the
degrees are unchanged and therefore remain distinct and increasing.
Finally, since $Q_{i+1}$ is a positive power of $Q_i$,
$$
    P^+(Q_{i+1})
    =
    P^+(Q_i)
    \le\sigma,
    \qquad
    \mathcal L(\sigma)\mid Q_{i+1}.
$$
Thus \eqref{e:iteration-saturation-invariant} is preserved. Moreover,
by \eqref{e:iteration-density-growth}, we have
$\delta_{i+1}\ge\delta_i\ge\alpha$, so
\eqref{e:iteration-cutoff-range} remains valid at the next stage. The
only remaining conditions that require quantitative verification are the
three bounds in \eqref{e:iteration-quantitative-hypotheses}.

\emph{The number of increments.}\par
At every successful application of Lemma~\ref{l:case3}, we have $\delta_{i+1}\ge(1+c)\delta_i.$
Consequently,
$$
    \delta_i\ge(1+c)^i\delta_0\ge(1+c)^i\alpha.
$$
Set
\begin{equation}\label{e:number-density-increments}
    n_*
    \coloneqq
    \left\lfloor
    \frac{\log(\alpha^{-1})}{\log(1+c)}
    \right\rfloor+1.
\end{equation}
Then $n_*
    =
    O_{C,d,k}\bigl(\log(\alpha^{-1})+1\bigr)$ and $(1+c)^{n_*}\alpha>1.$
Thus $n_*$ successful density increments would produce a set of
density greater than $1$.

\emph{Quantitative control of the iteration.}\par
To obtain a contradiction, we prove by induction that
\eqref{e:iteration-quantitative-hypotheses} holds at every stage
$0\le m<n_*$. Lemma~\ref{l:case3} may then be applied successively
at all these stages, producing the impossible inequality
$\delta_{n_*}>1$.

Fix $0\le m<n_*$, and suppose inductively that
\eqref{e:iteration-quantitative-hypotheses} has been verified at every
preceding stage $0\le i<m$. Thus Lemma~\ref{l:case3} has been
applied at the stages $0,\ldots,m-1$, and the construction has
produced $A_m\subseteq[N_m]$, with density $\delta_m$ and modulus
$Q_m$. We show that \eqref{e:iteration-quantitative-hypotheses}
also holds at stage $m$.

Since $B_i=O_d(1)$, there is a constant $B_*=O_d(1)$ such that $B_i+1\le B_*$
at every preceding stage. From $Q_{i+1}=Q_i^{B_i+1}$
we therefore obtain $\log Q_i\le B_*^i\log Q_0$ for every $0\le i\le m$.
Using \eqref{e:initial-modulus-bound} and $m<n_*$, it follows that,
uniformly for $0\le i\le m$,
\begin{equation}\label{e:iteration-modulus-bound}
    \log Q_i
    \le
    \alpha^{-O_{C,d,k}(1)}.
\end{equation}
Here
$$
    B_*^i
    \le
    \exp\bigl(O_{C,d,k}(\log(\alpha^{-1})+1)\bigr)
    =
    \alpha^{-O_{C,d,k}(1)}.
$$

We next obtain a lower bound for the lengths $N_i$. For every
preceding stage $0\le i<m$, the second inequality in the induction
hypothesis gives
$$
    Q_i
    \le
    \exp\bigl((\log N_i)^{2c'}\bigr).
$$
Since $2c'<1$, this implies that
$N_iQ_i^{1-d} \longrightarrow\infty$ as $N\to\infty$,
uniformly over $0\le i<m$.
Consequently, the floor in the definition of $M_i$ contributes only
a bounded error, and
$$
    \log M_i
    =
    \frac{1}{d}\log N_i
    -
    \frac{d-1}{d}\log Q_i
    +
    O(1)
    \qquad\text{for all }0\le i<m.
$$
Taking logarithms in \eqref{e:iteration-length-uniform}, we obtain
\begin{equation}\label{e:iteration-log-recurrence}
    \log N_{i+1}
    \ge
    \theta\log N_i
    -
    O_{C,d,k}\left(
        \log Q_i
        +
        \log(\alpha^{-1})
        +
        \sqrt{\log N_i}
    \right)
\end{equation}
for every $0\le i<m$, where $\theta=r/d$.

By \eqref{e:iteration-progression-containment}, we have
$N_{i+1}\le N_i$, and hence $N_i\le N$. Iterating
\eqref{e:iteration-log-recurrence} and using
\eqref{e:iteration-modulus-bound}, we obtain for every $0\le i\le m$ the lower bound
$$
\begin{aligned}
    \log N_i
    \ge{}&
    \theta^i\log N_0-
    O_{C,d,k}\left(
        i\alpha^{-O_{C,d,k}(1)}
        +
        i\log(\alpha^{-1})
        +
        i\sqrt{\log N}
    \right).
\end{aligned}
$$
Moreover, by 
\eqref{e:initial-modulus-small}, $N_0
    \ge
    \frac{N}{Q_0}-1
    \ge
    \frac{N}{2Q_0}$
for $N$ sufficiently large. Consequently,
$$
    \log N_0
    \ge
    \log N-\log Q_0-O(1)
    \ge
    \frac{3}{4}\log N-O(1)
    \gg\log N.
$$
Also, since $i\le n_*
    =
    O_{C,d,k}\bigl(\log(\alpha^{-1})+1\bigr),$
we have $\theta^i\ge\alpha^{O_{C,d,k}(1)}.$
Combining these estimates and absorbing powers of
$\log(\alpha^{-1})$ into $\alpha^{-O_{C,d,k}(1)}$, we obtain,
uniformly for every $0\le i\le m$,
\begin{equation}\label{e:iteration-final-scale}
\begin{aligned}
    \log N_i
    \gg{}&
    \alpha^{O_{C,d,k}(1)}\log N
    -
    \alpha^{-O_{C,d,k}(1)}-
    O_{C,d,k}\left(
        \log(\alpha^{-1})\sqrt{\log N}
    \right).
\end{aligned}
\end{equation}
Choose $c_*>0$ sufficiently small in
terms of $C,d,k$. Since $\alpha>(\log N)^{-c_*},$
the terms in \eqref{e:iteration-final-scale} satisfy
$$
\begin{aligned}
    \alpha^{O_{C,d,k}(1)}\log N
    &\ge
    (\log N)^{1-O_{C,d,k}(c_*)},\\
    \alpha^{-O_{C,d,k}(1)}
    &\le
    (\log N)^{O_{C,d,k}(c_*)},\\
    \log(\alpha^{-1})\sqrt{\log N}
    &\ll
    (\log\log N)\sqrt{\log N}.
\end{aligned}
$$
After decreasing $c_*$ if necessary, the first term in
\eqref{e:iteration-final-scale} dominates the two error terms.
Therefore there exists a constant
$\gamma=\gamma(C,d,k)>0$ such that
\begin{equation}\label{e:iteration-scale-lower-bound}
    \log N_i\ge(\log N)^\gamma
\end{equation}
for every $0\le i\le m$.
We may decrease $c_*$ once more so that
\eqref{e:iteration-modulus-bound} and
\eqref{e:iteration-scale-lower-bound} imply $\log Q_m
    \le
    (\log N_m)^{2c'},$
and hence $Q_m
    \le
    \exp\bigl((\log N_m)^{2c'}\bigr).$
Similarly, since
$$
    \sigma
    =
    \left\lceil
    \mathbf K_{d,k}\alpha^{-\mathbf C_{d,k}}
    \right\rceil
    \le
    \alpha^{-O_{C,d,k}(1)},
$$
we have $\sigma\le(\log N_m)^{c'}.$

Finally, choosing $c_*<\gamma/\mathbf C_{d,k}^{2}$, we obtain
$$
\begin{aligned}
    (\log N_m)^{-1/\mathbf C_{d,k}^{2}}
    &\le
    (\log N)^{-\gamma/\mathbf C_{d,k}^{2}}<
    (\log N)^{-c_*}
    <
    \alpha
    \le\delta_m.
\end{aligned}
$$
Thus all three estimates in
\eqref{e:iteration-quantitative-hypotheses} hold at stage $m$.

\emph{The final density contradiction.}\par
Since $m<n_*$ was arbitrary, Lemma~\ref{l:case3} remains applicable
through stage $n_*-1$. The iteration therefore produces
$A_{n_*}$ with
$$
    \delta_{n_*}
    \ge
    (1+c)^{n_*}\delta_0
    \ge
    (1+c)^{n_*}\alpha
    >
    1,
$$
which is impossible. Hence the assumption
\eqref{e:main-contradiction-density} cannot hold. It follows that
$$
    |A|
    \ll
    N(\log N)^{-c_*},
$$
as required.
\end{proof}

\bibliography{refs}
\bibliographystyle{plain}

\end{document}